\documentclass[a4paper,11pt]{amsart}

\usepackage{amsmath, amsthm, amsfonts, amssymb}
\usepackage[bookmarks=false]{hyperref}
\usepackage{enumerate}

\usepackage{color}

\numberwithin{equation}{section}

\newtheorem{letterthm}{Theorem}

\newtheorem{thm}{Theorem}[section]
\newtheorem{lem}[thm]{Lemma}

\newtheorem{cor}[thm]{Corollary}
\newtheorem{prop}[thm]{Proposition}

\theoremstyle{definition}
\newtheorem{rem}[thm]{Remark}
\newtheorem{remark}[thm]{Remarks}

\newtheorem{example}[thm]{Example}

\newtheorem{df}[thm]{Definition}

\newtheorem*{newclaim}{Claim}

\newtheorem*{question}{Question}

\newcommand{\R}{\mathbf{R}}
\newcommand{\C}{\mathbf{C}}
\newcommand{\Z}{\mathbf{Z}}
\newcommand{\F}{\mathbf{F}}

\newcommand{\N}{\mathbf{N}}
\newcommand{\B}{\mathbf{B}}

\newcommand{\cH}{\mathcal{H}}
\newcommand{\cF}{\mathcal{F}}
\newcommand{\cL}{\mathcal{L}}
\newcommand{\cK}{\mathcal{K}}

\newcommand{\cN}{\mathcal{N}}

\newcommand{\id}{\text{\rm id}}

\newcommand{\Aut}{\operatorname{Aut}}

\newcommand{\rL}{\mathord{\text{\rm L}}}

\newcommand{\Aff}{\mathord{\text{\rm Aff}}}
\newcommand{\Isom}{\mathord{\text{\rm Isom}}}

\newcommand{\supp}{\mathord{\text{\rm supp}}}

\newcommand{\Mod}{\mathord{\text{\rm Mod}}}

\newcommand{\rd}{\: \mathrm{d}}
\newcommand{\ri}{\mathrm{i}}

\newcommand{\II}{{\rm II}}
\newcommand{\III}{{\rm III}}

\DeclareMathOperator*\lowlim{\underline{lim}}
\DeclareMathOperator*\uplim{\overline{lim}}

\begin{document}

\title[Ergodic theory of affine isometric actions on Hilbert spaces]{Ergodic theory of affine isometric actions\\ on Hilbert spaces}


\author{Yuki Arano}
\email{y.arano@math.kyoto-u.ac.jp}
\thanks{YI is supported by JSPS KAKENHI Grant Number JP18K13424.}
\address{Kyoto University, 606-8502 Japan}

\author{Yusuke Isono}
\email{isono@kurims.kyoto-u.ac.jp}
\thanks{YI is supported by JSPS KAKENHI Grant Number JP17K14201.}
\address{RIMS, Kyoto University, 606-8502 Japan}

\author{Amine Marrakchi}
\email{amine.marrakchi@ens-lyon.fr}
\thanks{AM was a JSPS International Research Fellow (PE18760)}
\address{CNRS, UMPA, 69007 France}



\begin{abstract}
The classical Gaussian functor associates to every orthogonal representation of a locally compact group $G$ a probability measure preserving action of $G$ called a Gaussian action. In this paper, we generalize this construction by associating to every affine isometric action of $G$ on a Hilbert space, a one-parameter family of nonsingular Gaussian actions whose ergodic properties are related in a very subtle way to the geometry of the original action. We show that these nonsingular Gaussian actions exhibit a phase transition phenomenon and we relate it to new quantitative invariants for affine isometric actions. We use the Patterson-Sullivan theory as well as Lyons-Pemantle work on tree-indexed random walks in order to give a precise description of this phase transition for affine isometric actions of groups acting on trees. We also show that every locally compact group without property (T) admits a nonsingular Gaussian that is free, weakly mixing and of stable type $\III_1$.
\end{abstract}

\maketitle


\section{Introduction and main results}
The theory of orthogonal (or unitary) representations of locally compact groups on Hilbert spaces is a central part of representation theory which has deep connections with many other topics in mathematics and physics. Of particular interest is the connection with ergodic theory. Indeed, starting from a probability measure preserving (pmp) action $\sigma$ of a group $G$ on a probability space $(X,\mu)$ one can construct an orthogonal representation of $G$ on $\rL^2(X,\mu)$ called the \emph{Koopman representation} of $\sigma$ which is by now a fundamental tool in ergodic theory. In the other direction, starting from an orthogonal representation $\pi$ of $G$ on a Hilbert space $\cH$, it is possible to construct a pmp action of $G$ called the \emph{Gaussian action} of $\pi$ which has many applications \cite{CW80, Sc81, Sc96}. 

Recently, the study of affine isometric actions of locally compact groups on Hilbert spaces became also an important topic in representation theory and geometric group theory (see \cite{BHV08} and the references therein). The main goal of this paper is to show that this new chapter of representation theory also has an interesting connection with ergodic theory. The fact that such a connection should exist was already hinted by a very recent work of Vaes and Wahl \cite{VW18} where they related nonsingular Bernoulli actions of a given discrete group to the cohomology of its left regular representation (see also \cite{BKV19}). Inspired by their work, we provide  in this article a different and more general construction which associates to every affine isometric action of a locally compact group $G$ on a real affine Hilbert space a \emph{nonsingular Gaussian action}. Let us explain this construction.

Let $\cH$ be a real affine Hilbert space. Suppose first that $\cH$ is finite dimensional. Then, for each $x \in \cH$, one can consider the standard Gaussian probability measure $\mu_x$ on $\cH$ centered at $x$. Observe that for every affine isometry $g \in \Isom(\cH)$ one has $g_* \mu_x=\mu_{gx}$ for all $x \in \cH$. It turns out that this situation generalizes perfectly to the infinite dimensional case, with the only difference that the probability measures $(\mu_x)_{x \in \cH}$ can no longer be defined on the space $\cH$ itself\footnote{The reason is that the Gaussian measure $\mu_x(B)$ of the unit ball centered at $x$ goes to $0$ when the dimension of $\cH$ goes to infinity}. Indeed, we show (Section \ref{Affine Gaussian functor}) that one can naturally associate to every real affine Hilbert space $\cH$ a family of equivalent probability measures $(\mu_x)_{x \in \cH}$ on some standard borel space $\widehat{\cH}$ which behaves exactly as a family of Gaussian probability measures. The pair $(\widehat{\cH}, (\mu_x)_{x \in \cH})$ can be characterized uniquely (up to null-sets) by the property that every continuous affine function $f : \cH \rightarrow \R$ defines a random variable $\widehat{f}$ on $\widehat{\cH}$ which has a Gaussian dsitribution (not necessarily centered) with respect to $\mu_x$ for every $x \in \cH$. We then observe that for every affine isometry $g \in \Isom(\cH)$, there exists a unique (up to null-sets) measurable map $\widehat{g} : \widehat{\cH} \rightarrow \widehat{\cH}$ such that $\widehat{g}_* \mu_x=\mu_{gx}$ for all $x \in \cH$. It follows that every affine isometric action $\alpha : G \curvearrowright \cH$ of a locally compact group $G$ induces a \emph{nonsingular Gaussian action} $\widehat{\alpha} : G \curvearrowright \widehat{\cH}$ which preserves the measure class of $(\mu_x)_{x \in \cH}$. Of course, our construction is new only when $\alpha$ has \emph{no fixed point} in $\cH$. Indeed, when the affine isometric action $\alpha$ fixes a point $x \in \cH$, the action $\widehat{\alpha}$ will preserve the probability measure $\mu_x$ and in that case, by declaring $x$ to be the origin of the Hilbert space, one recovers the classical pmp Gaussian action associated to an orthogonal representation.  Recall that by Guichardet's theorem, a polish locally compact group $G$ admits an affine isometric action without fixed points if and only if it does not have Kazhdan's property (T). We thus obtain a new and large class of nonsingular actions for all groups without property (T). Our main problem is the following.

\begin{question} What are the ergodic properties of the Gaussian action $\widehat{\alpha}$ and how do they relate to the geometry of the original affine isometric action $\alpha$?
\end{question}

The special case of pmp Gaussian actions is very well understood (see \cite{Bo14}). For example, one knows that the pmp Gaussian action associated to an orthogonal representation $\pi$ is ergodic if and only if $\pi$ is \emph{weakly mixing}, it has no finite-dimensional subrepresentation. As we will see, for nonsingular Gaussian actions, the problem becomes much more subtle. Moreover, one can no longer give purely \emph{qualitative} statements. In order to explain why, we need to make the following key observation. Let $\cH$ be an affine Hilbert space. By simply rescaling the metric of $\cH$ by a parameter $t > 0$, one obtains a new affine Hilbert space $\cH^t$, and every affine isometric action $\alpha : G \curvearrowright \cH$ induces an affine isometric action $\alpha^t : G \curvearrowright \cH^t$. The key point is that when $\cH$ is \emph{infinite-dimensional}, the Gaussian actions $\widehat{\alpha}^t : G \curvearrowright \widehat{\cH}^t$ are in general not conjugate for different values of $t$. In fact, as we shall see, the behaviour of $\widehat{\alpha}^t$ can change dramatically when $t$ grows from $0$ to $\infty$ exhibiting a fascinating \emph{phase transition} phenomenon. Intuitively, high values of $t$ correspond to a cold ordered phase where one expects the orbits of the action $\widehat{\alpha}^t$ to be well-behaved, while small values of $t$ correspond to a hot disordered phase where the structure of the orbits becomes more chaotic. If we denote by $\alpha^0$ the linear part of $\alpha$, then we can also think of $\widehat{\alpha}^t$ as converging to the pmp Gaussian action $\widehat{\alpha}^0$ when $t \to 0$.


Let us start right away with a concrete example of such a phase transition. Let $T$ be a locally finite tree which we view as a metric space where the distance $d(x,y)$ between two vertices $x,y \in T$ is simply the length of the segment $[x,y]$. It is well-known that this metric $d$ is of \emph{negative-type}. This means that there is a unique way $\iota : T \rightarrow \cH$ to embed $T$ into a real affine Hilbert space $\cH$ such that:
\begin{itemize}
\item $d(x,y)=\| \iota(x)-\iota(y) \|^2$ for all $x,y \in T$.
\item The affine span of $\iota(T)$ is dense in $\cH$.
\end{itemize}
It follows that any automorphism of $T$ extends uniquely to an affine isometry of $\cH$. Now, let $\Gamma$ be a discrete group of automorphisms of $T$. Then the action of $\Gamma$ on $T$ extends uniquely to an affine isometric action $\alpha : \Gamma \curvearrowright \cH$. The following theorem completely settles the ergodicity and type of the Gaussian actions $\widehat{\alpha}^t$ except for one critical value of $t$. Moreover, we relate this critical value to the \emph{Poincar\'e exponent} of $\Gamma$, denoted by $\delta(\Gamma)$, which is a fundamental invariant of discrete isometry groups of hyperbolic spaces. 

\begin{letterthm} \label{letter tree Gaussian}
Let $T$ be a locally finite tree and $\Gamma < \Aut(T)$ a nonelementary discrete subgroup. Let $\alpha : \Gamma \curvearrowright \cH$ be the associated affine isometric action. Let $\delta:=\delta(\Gamma)$ be the infinimum of all $s > 0$ such that
$$ \sum_{ g \in \Gamma} e^{-sd(gx,y)} < + \infty$$
for some (hence any) $x,y \in T$. Then the Krieger type of the Gaussian actions $\widehat{\alpha}^t$ for $t > 0$ is given by:
\begin{center}
\renewcommand{\arraystretch}{1.5}
\begin{tabular}{l|l}
	 Value of $t$ \quad &  Ergodicity and type of $\widehat{\alpha}^t$ \\
	\hline
	$  t < 2 \sqrt{2 \delta}$ & Ergodic of type $\III_1$. \\
	$ t=2 \sqrt{2 \delta}$ & $?$ \\
	$t > 2 \sqrt{2\delta}$ & Not ergodic and of type $\rm I$ (it has a fundamental domain).\\

	\end{tabular}
	\renewcommand{\arraystretch}{1}
\end{center}
Moreover, we have:
\begin{enumerate}[ \rm (i)]
\item If $\delta < + \infty$, the actions $\widehat{\alpha}^t$ are pairwise non-conjugate for $t < 2 \sqrt{2 \delta}$.
\item  The action $\widehat{\alpha}^t$ is strongly ergodic (and in particular nonamenable) for $t$ small enough.
\end{enumerate}

\end{letterthm}

The proof of Theorem \ref{letter tree Gaussian} relies on ingredients from geometry (Patterson-Sullivan theory for isometry groups of hyperbolic spaces \cite{Pa76, Su79}) and probability theory (Lyons and Pemantle work on tree-indexed random walks \cite{LP92}). 

\begin{example}
Let $\Gamma=\F_d$ where $d \geq 2$ and let it act on its Cayley tree $T$. It is then easy to see that $\delta=\log(2d-1)$. Thus, we know that $\widehat{\alpha}^t$ is ergodic of type $\III_1$ for all $t < 2\sqrt{2 \log(2d-1)}$ and has a fundamental domain for $t > 2 \sqrt{2\log(2d-1)}$. For this very specific example, we can also show that $\widehat{\alpha}^t$ has a fundamental domain at the critical value $t=2 \sqrt{2\delta}$ and we also show that $\widehat{\alpha}^t$ is nonamenable for all $t < 2 \sqrt{\delta}$ (but we do not know whether $\widehat{\alpha}^t$ is amenable or not for $2 \sqrt{\delta} \leq  t < 2 \sqrt{2\delta}$). The result of Theorem \ref{spectral radius cayley} suggests that there should be a second phase transition at $t=2\sqrt{\delta}$.
\end{example}

\begin{rem}
Recall that a nonsingular action $\sigma : G \curvearrowright X$ is \emph{weakly mixing} if the diagonal action $\sigma \otimes \rho : G \curvearrowright X \otimes Y$ is ergodic for every ergodic pmp action $\rho : G \curvearrowright (Y,\nu)$. If moreover $\sigma \otimes \rho$ is of type $\III_1$ for every ergodic pmp action $\rho : G \curvearrowright (Y,\nu)$, we say that $\sigma$ is of \emph{stable type $\III_1$} (this is equivalent to saying that the Maharam extension of $\sigma$ is weakly mixing). In Theorem \ref{letter tree Gaussian}, we actually show that $\widehat{\alpha}^t$ is weakly mixing of stable type $\III_1$ for all $t < 2 \sqrt{2\delta}$.
\end{rem}

\begin{rem}
In Theorem \ref{main trees Gaussian} we generalize Theorem \ref{letter tree Gaussian} by dealing with actions of locally compact groups on trees which are not proper. In that case, the Gaussian actions $\widehat{\alpha}^t$ will be of type $\II_\infty$ for large values of $t$. In Section \ref{section bernoulli trees} (which can be read independently from the rest of the paper), we also study a class of nonsingular Bernoulli actions which are discretized versions of the Gaussian actions of Theorem \ref{letter tree Gaussian}. For these nonsingular Bernoulli actions, we obtain an analog of Theorem \ref{letter tree Gaussian} where the type $\III_\lambda, \; \lambda \in ]0,1[$ appears.
\end{rem}

Let us now state some general results on Gaussian actions which will put Theorem \ref{letter tree Gaussian} into a larger perspective. A free nonsingular action $\sigma : G \curvearrowright X$ of a locally compact group is called \emph{dissipative} (or equivalently of type $\mathrm{I}$) if and only if it is conjugate to a nonsingular action of the form $G \curvearrowright G \otimes Y$ where $G$ acts only on the first coordinate by left translation. This really means that the orbits of $\sigma$ are perfectly ordered and well-behaved. In particular, when $G$ is discrete, this means that $\sigma$ admits a fundamental domain. On the opposite side, a free nonsingular action $\sigma$ is \emph{recurrent} (also called \emph{conservative}) if and only if for every set of positive measure $A \subset X$ and every compact set $K \subset G$, one can find $g \in G \setminus K$ such that $gA \cap A$ has positive measure. This implies that the structure of the orbits is rather chaotic and this situation is of course the most interesting to study from the ergodic theory point of view. Note that, by the Poincar\'e recurrence theorem, a probability measure preserving action of a non-compact group is always recurrent.

We now give a quite sharp estimate of when Gaussian actions are recurrent/dissipative. For this we need to introduce a new invariant for affine isometric actions. Let $\alpha : G \curvearrowright \cH$ be an affine isometric action of a locally compact group $G$. The \emph{(quadratic) Poincar\'e exponent} $\delta(\alpha) \in [0,+\infty]$ is defined as the infinimum of all $s > 0$ such that
$$ \int_G e^{-s\|gx-y\|^2} \rd g < +\infty, \quad x,y \in \cH.$$
This number does not depend on the choice of $x,y \in \cH$ and is defined by analogy with the Poincar\'e exponent of discrete groups of isometries of hyperbolic spaces. In fact, this is not just an analogy and both notions are actually related. For example, when $\alpha : \Gamma \curvearrowright \cH$ is the affine isometric action associated to a discrete automorphism group $\Gamma < \Aut(T)$ of a locally finite tree then $\delta(\alpha)=\delta(\Gamma)$.

The following theorem is inspired by the dissipativity criterion of Vaes and Wahl \cite{VW18} where the Poincar\'e exponent already appears implicitely.

\begin{letterthm} \label{letter dissipativity}
Let $\alpha : G \curvearrowright \cH$ be an affine isometric action of a locally compact group $G$. There exists $t_{\rm diss}(\alpha) \in [0,+\infty]$ such that the Gaussian action $\widehat{\alpha}^t$ is recurrent for all $t < t_{\rm diss}(\alpha)$ and dissipative for all $t > t_{\rm diss}(\alpha)$. Moreover, we have the following inequalities 
$$ \sqrt{2 \delta(\alpha)} \leq  t_{\rm diss}(\alpha) \leq 2 \sqrt{2 \delta(\alpha)}.$$
\end{letterthm}

\begin{rem} Unfortunately, it turned out to be difficult to compute the exact value of $t_{\rm diss}(\alpha)$ in general. It is even more difficult to determine what happens at $t=t_{\rm diss}(\alpha)$. See Proposition \ref{translation action} for a very simple example. For affine isometric actions of automorphism groups of trees, we found that $t_{\rm diss}(\alpha) =2 \sqrt{2 \delta(\alpha)}$ which shows that the upper bound in Theorem \ref{letter dissipativity} is optimal, but we have a priori no reason to believe that the equality $t_{\rm diss}(\alpha) =2 \sqrt{2 \delta(\alpha)}$ is always true.
\end{rem}

\begin{rem}
Let $\alpha : G \curvearrowright \cH$ be an affine isometric action and suppose that $G$ is \emph{nonamenable}. Then $\delta(\alpha) > 0$ (Proposition \ref{poincare strictly positive}). If we assume moreover that $\alpha$ has \emph{almost fixed points}, then one can show that $\delta(\alpha)=+\infty$ (Corollary \ref{poincare infinite}). In contrast, one can easily construct an affine isometric action of $\Z$ which has almost fixed points yet a vanishing Poincar\'e exponent. See also Corollary \ref{lower bound amenability} which gives a lower bound for the Poincar\'e exponent of an affine isometric action of a given nonamenable group $G$ in terms of the spectral radius of $G$.
\end{rem}

 
Once we know when a Gaussian action is recurrent, we can ask about its ergodicity and its Krieger type. It is well-known that a pmp Gaussian action associated to an orthogonal representation $\pi$ is ergodic if and only if $\pi$ is \emph{weakly mixing}, i.e.\ $\pi$ has no finite dimensional subrepresentation. For nonsingular Gaussian actions, the situation is much more subtle and in our opinion, it is hopeless to find a general necessary and sufficient condition for ergodicity (see for instance Example \ref{counter-example weak mixing}). However, we provide a rather satisfactory answer under the assumption that the linear part of the affine isometric action is \emph{mixing}. Recall that an orthogonal representation $\pi : G \rightarrow \mathcal{O}(\cH)$ of a locally compact group $G$ on a Hilbert space $\cH$ is mixing if $ \lim_{g \to \infty} \langle \pi(g) \xi, \eta \rangle  = 0$ for all $\xi, \eta \in \cH$. We also say that $\pi$ has \emph{spectral gap} if it has no almost invariant vectors and we say that $\pi$ has \emph{stable spectral gap} if $\pi \otimes \rho$ has spectral gap for every orthogonal representation $\rho$ of $G$. By \cite{Bo14}, one knows that the pmp Gaussian action of $\pi$ is strongly ergodic if and only if $\pi$ has stable spectral gap. The second part of the following theorem is therefore not surprising.

\begin{letterthm} \label{letter ergodic}
Let $\alpha : G \curvearrowright \cH$ be an affine isometric action such that $\alpha$ has no fixed point and the linear part of $\alpha$ is mixing. Then $\widehat{\alpha}^t$ is weakly mixing for all $t < t_{\rm diss}(\alpha)$ and for all $ t < \frac{1}{\sqrt{2}} t_{\rm diss}(\alpha)$, the action $\widehat{\alpha}^t$ is either of type $\III_0$ or of type $\III_1$.

Suppose moreover that the linear part of $\alpha$ has stable spectral gap. Then there exists $t_0 > 0$ such that the action $\widehat{\alpha}^t$ is strongly ergodic of type $\III_1$ for all $t \in ]0, t_0[$.
\end{letterthm}
\begin{example}
Let $G$ be a locally compact group with a non-vanishing first $\rL^2$-betti number $\beta^{(2)}_1(G) > 0$. Then there exists an affine isometric action $\alpha : G \curvearrowright \cH$ without fixed point whose linear part is the left regular representation. The left regular representation of $G$ is mixing and since $G$ is nonamenable, it also has stable spectral gap. Therefore, we know that $\widehat{\alpha}^t$ is strongly ergodic of type $\III_1$ for $t > 0$ small enough. Even though our method of proof is completely different, notice the analogy with the results of \cite{VW18} and \cite{BKV19} for nonsingular Bernoulli actions of nonamenable groups.
\end{example}

\begin{example}
The actions considered in Theorem \ref{letter tree Gaussian} also satisfy the assumptions of Theorem \ref{letter ergodic}. However, we cannot use Theorem \ref{letter ergodic} to deal with large values of $t > 0$ (notice also that the technique of \cite{BKV19} cannot deal with large values of $t > 0$). For this reason, we use a different geometric approach to obtain the sharp result of Theorem \ref{letter tree Gaussian}.
\end{example}

In the first part of Theorem \ref{letter ergodic}, it is unclear whether the type $\III_0$ case can really appear and it seems that $\widehat{\alpha}^t$ should always be of type $\III_1$ for all $t < t_{\rm diss}(\alpha)$. We were unable to prove this in full generality. However, there is a class of affine isometric actions, which we call \emph{evanescent}, for which one can prove an optimal result. Very roughly, an affine isometric action is evanescent if it admits arbitrarily small invariant subspaces (see Section \ref{section evanescent} and Proposition \ref{evanescent small subspaces}). For this class of affine isometric actions, we obtain the following very sharp dichotomy.

\begin{letterthm} \label{letter evanescent}
Let $\alpha : G \curvearrowright \cH$ be an affine isometric action. Suppose that $\alpha$ is evanescent, that it has no fixed point and that its linear part is mixing. Then for every $t > 0$ and every ergodic nonsingular action $\rho : G \curvearrowright Y$ (not necessarily pmp), the diagonal action $\widehat{\alpha}^t \otimes \rho : G \curvearrowright \widehat{\cH}^t \otimes Y$ is either dissipative or ergodic of type $\III_1$. In particular, for all $t > 0$, $\widehat{\alpha}^t$ is either dissipative or weakly mixing of stable type $\III_1$.
\end{letterthm}

\begin{example}
The assumptions of Theorem \ref{letter evanescent} are satisfied in the following cases:
\begin{enumerate}[ \rm (i)]
\item $G$ is discrete and amenable, $\alpha$ has no fixed point and its linear part is contained in a multiple of the left regular representation (Proposition \ref{amenable evanescent}).
\item $G$ is nilpotent (e.g.\ abelian), $\alpha$ has no fixed point and its linear part is mixing (Proposition \ref{nilpotent evanescent}).
\item $\alpha$ is the affine isometric action associated to a closed \emph{parabolic} subgroup $G < \Aut(T)$ for some locally finite tree $T$ (Proposition \ref{parabolic evanescent}).
\end{enumerate}
Compare $(\rm i)$ and $(\rm ii)$ with \cite[Theorem 3.2]{BKV19} and \cite[Theorem 3.3]{BKV19} for nonsingular Bernoulli actions.
\end{example}

\begin{rem}
Recall that an affine isometric action is \emph{irreducible} if it admits no proper invariant affine subspace (see \cite{BPV16} for a systematic study of this notion). In Section \ref{section evanescent}, we observe that every affine isometric action decomposes as a product of an evanescent part and an irreducible part. Therefore, in the mixing case, if the evanescent part has no fixed point, Theorem \ref{letter evanescent} shows that the Gaussian action can only be of type $\III_1$ if it is recurrent. If the evanescent part has a fixed point, then the problem reduces to the irreducible part. The latter case occurs for example for affine isometric actions whose linear part has spectral gap. This means that the context of Theorem \ref{letter evanescent} is completely ``orthogonal" to the second part of Theorem \ref{letter ergodic}. In particular, Theorem \ref{letter evanescent} does not apply at all to the actions considered in Theorem \ref{letter tree Gaussian}.
\end{rem}

We note that a locally compact but non-compact group admits an evanescent affine isometric action without fixed point whose linear part is mixing if and only if it has the Haagerup property (see Proposition \ref{evanescent properties}). This limits the class of groups to which Theorem \ref{letter evanescent} can be applied. However, every group without property (T) admits an evanescent affine isometric action whithout fixed point whose linear part is \emph{weakly mixing}. This case is much more subtle than the mixing case (to understand why, see Section \ref{section weakly mixing} and Example \ref{counter-example weak mixing}). Nevertheless, under additional assumptions, we provide in Section \ref{section weakly mixing} some technical tools to deal with the weakly mixing case. As an application, we obtain the following result.

\begin{letterthm} \label{letter non T}
Let $G$ be a locally compact group without property (T). Then there exists an affine isometric action $\alpha : G \curvearrowright \cH$ such that for all $t > 0$, the nonsingular Gaussian action $\widehat{\alpha}^t$ satisfies :
\begin{enumerate}[ \rm (i)]
\item $\widehat{\alpha}^t$ is free.
\item $\widehat{\alpha}^t$ is weakly mixing of stable type $\III_1$.
\item $\widehat{\alpha}^t$ has almost vanishing entropy. In particular, $\widehat{\alpha}^t$ has an invariant mean and it is nonamenable if $G$ is nonamenbale.
\end{enumerate}
\end{letterthm}
For groups with the Haagerup property, we obtain a stronger result (Theorem \ref{strong Haagerup}). We refer to Definition \ref{almost vanishing entropy} for the definition of ``almost vanishing entropy". This property implies that the action $\sigma$ admits probability measures with arbitrarily small Furstenberg entropy. Thus we recover, in a stronger form, the main result of \cite{BHT14}.

\subsection*{Acknowledgments} The authors express their deep gratitude to Stefaan Vaes for his interest in our work and for finding several mistakes in an earlier version of this paper. His advice also helped us to improve the presentation. We are grateful to Narutaka Ozawa for his numerous comments and for providing us with a correct proof of Proposition \ref{single vector spectral radius}. We are grateful to Frederic Paulin for his explanations regarding the references \cite{Su79} and \cite{RT13}. We thank Michel Pain for explaining to us the additive martingale argument of Theorem \ref{speed regular tree}. We thank Tushar Das for attracting our attention to the monograph \cite{DSU17} and his comments on our paper. Finally we thank Amaury Freslon for pointing out to us the relations between our work and \cite{Fr18} and \cite{HP84}.

\tableofcontents



\subsection*{Conventions} All locally compact groups are assumed to be Polish (second countable). All Hilbert spaces are assumed to be separable. All measure spaces are standard borel spaces.

\section{Affine Hilbert spaces and affine isometric actions}
In this paper, we will use the language of affine Hilbert spaces as in \cite[Chapter 2]{BHV08}. 

\subsection{Affine Hilbert spaces}
An \emph{affine Hilbert space} is a Hilbert space whithout an origin. More formally, a (real) affine Hilbert space is a nonempty set $\cH$ equipped with a free and transitive action of a (real) linear Hilbert space denoted $\cH^0$ and called the \emph{tangent space} of $\cH$. This action is denoted by
$$ \cH \times \cH^0 \ni (x,\xi) \mapsto x+\xi \in \cH.$$
For every $x,y \in \cH$, we also denote by $y-x$ the unique vector of $\cH^0$ such that $x+(y-x)=y$. Note that $\cH$ is a metric space for the distance $(x,y) \mapsto \|y-x\|$.

\subsubsection*{Affine maps} A continuous map $\theta  :  \cH \rightarrow \cK$ between two affine Hilbert spaces is \emph{affine} if there exists a continuous linear map $\theta^0 : \cH^0 \rightarrow \cK^0$ such that $\theta(x+\xi)=\theta(x)+\theta^0(\xi)$ for all $x \in \cH$ and $\xi \in \cH^0$. Then $\theta^0$ is unique and is called the \emph{linear part} of $\theta$. We equip the set of all continuous affine maps $\Aff(\cH,\cK)$ with the topology of pointwise convergence (also called strong topology). If $\theta \in \Aff(\cH,\cK)$ and $\theta^0$ is an orthogonal operator, we say that $\theta$ is an \emph{affine isometry}. It is a fact that every isometric map $\theta : \cH \rightarrow \cK$ with respect to the natural metrics is automatically an affine isometry. We denote by $\Isom(\cH)$ the (affine) isometry group of $\cH$. We equip it with the topology of pointwise convergence on $\cH$ which turns it into a Polish group. The map $\Isom(\cH) \ni \theta \mapsto \theta^0 \in \mathcal{O}(\cH^0)$ is a continuous group homomorphism. We also have a continuous group homomorphism $j:\cH^0 \rightarrow \Isom(\cH)$ defined by $j_\xi(x)=x+\xi$ for $x \in \cH$ and $\xi \in \cH^0$. These two maps fit together in an exact sequence
 $$ 0 \rightarrow \cH^0 \rightarrow \Isom(\cH) \rightarrow \mathcal{O}(\cH^0) \rightarrow 1.$$

\subsubsection*{Affine subspaces} Let $\cH$ be an affine Hilbert space. A (closed) \emph{affine subspace} $\cK \subset \cH$ is a subset of the form $x+ E$ for some closed linear subspace $E \subset \cH^0$ and some $x \in \cH$. Then $\cK$ is naturally an affine Hilbert space whose tangent space is identified with $\cK^0=E$. Note however, that the choice of $x$ is not unique. By definition, the inclusion map $\iota : \cK \rightarrow \cH$ is an affine isometry. Our convention is that an affine subspace is never empty.

\subsubsection*{Products}
Let $\cK$ and $\cL$ be two affine Hilbert spaces. Then $\cK \times \cL$ has a natural free and transitive action of $\cK^0 \oplus \cL^0$ so that $\cL \times \cL$ is naturally an affine Hilbert space with $(\cK \times \cL)^0=\cK^0 \times \cL^0$. We call it the \emph{product} of $\cK$ and $\cL$. Note that $\cK$ and $\cL$ are \emph{not} affine subspaces of $\cK \times \cL$. If $\cK$ is an affine subspace of $\cH$, then naturally, we have $\cH \cong \cK \times E$ where $E=(\cK^0)^{\perp} \subset \cH^0$.

\subsubsection*{Quotients}
Let $\cH$ be an affine Hilbert space and $E \subset \cH^0$ a closed linear subspace. Then the quotient $\cH/E$ admits a natural free and transitive action of $E^{\perp}$ so that $\cH/E$ is an affine Hilbert space with $(\cH/E)^0=E^{\perp}$. Note that the projection map 
$$\pi : \cH \ni x \mapsto x+E \in \cH/E$$ is an affine map and $\pi^0$ is simply the orthogonal projection on $E^{\perp}$. Finally, if we let $F=E^{\perp}$, then the map
$$ \cH \ni x \mapsto (x+E,x+F) \in \cH/E \times \cH/F$$
is a surjective affine isometry and we will always use it to make the identification $\cH=\cH/E \times \cH/F$.

\subsection{Geometry of subspaces} We prove some elementary facts about the geometry of subspaces in an affine Hilbert space that we will need later.

\begin{prop} \label{empty intersection diverge}
Let $\cH$ be an affine Hilbert space and $(\cH_i)_{i \in I}$ a decreasing net of affine subspaces of $\cH$. Then the following are equivalent:
\begin{enumerate}[ \rm (i)]
\item $\bigcap_i \cH_i=\emptyset$
\item $\lim_i d(x,\cH_i)=+\infty$ for some $x \in \cH$.
\item $\lim_i d(x,\cH_i)=+\infty$ for all $x \in \cH$.
\end{enumerate}
\end{prop}
\begin{proof}
The implicatios $(\rm iii) \Rightarrow (\rm ii) \Rightarrow (\rm i)$ are obvious. Let us show that $(\rm i) \Rightarrow (\rm iii)$. Take $x \in \cH$. Let $e_i(x)$ be the orthogonal projection of $x$ on $\cH_i$. The net $\| e_i(x)-x\|$ is increasing. Thus, if it does not go to $\infty$ when $i \to \infty$ then it must converge $\lim_i \| e_i(x)-x\| \to \ell <+\infty$. But since $e_i(x)$ are orthgonal projections, we have $\| e_i(x)-e_j(x)\|^2 \leq \|e_j(x)-x\|^2 - \|e_i(x)-x\|^2$ for all $i \leq j \in I$. This implies that $\lim_{i,j \to \infty} \| e_i(x)-e_j(x) \| =0$, i.e.\ that $(e_i(x))_{i \in I}$ is a Cauchy net. Therefore it must converge to a point in $\bigcap_i \cH_i$.
\end{proof}

Let $\cH$ be an affine Hilbert space and $\cK \subset \cH$ a closed subspace. Define the subgroup $G_{\cK}=\{ g \in \Isom(\cH) \mid gx=x \text{ for all } x \in \cK \}$. Let $E \subset \cH^0$ be a closed subspace. Define the subgroup $H_{E}=\{ g \in \Isom(\cH) \mid gx-x \in E^{\perp} \text{ for all } x \in \cH \}$. Note that with respect to the decomposition $\cH=(\cH/E) \times (\cH/E^{\perp})$, we have $H_E=\Isom(\cH/E) \times \{\id\}$.

\begin{prop} \label{groups and two subspaces}
Let $\cH$ be an affine Hilbert space and let $\cH_1$ and $\cH_2$ be two affine subspaces of $\cH$. Define $\cK =\cH_1 \cap \cH_2$, $E= \cH_1^0 \cap \cH_2^0$ and $F=(\cH_1^0)^\perp \cap (\cH_2^0)^\perp$. Let $G \subset \Isom(\cH)$ be the closed subgroup generated by $G_{\cH_1} \cup G_{\cH_2}$. 
\begin{enumerate}[ \rm (i)]
\item If $F \neq \{0\}$ and $\cK \neq \emptyset$ then $G=G_{\cK}$. 
\item If $\dim F \geq 2$ and $\cK=\emptyset$, then $G=H_{E}$.
\end{enumerate}
\end{prop}
\begin{proof}
$(\rm i)$ Since $\cK \neq 0$, this is a statement about linear spaces and reduces easily to the fact that $\mathcal{O}(\R^n)$ is generated by $\mathcal{O}(E_1) \cup \mathcal{O}(E_2)$ whenever $E_1$ and $E_2$ are two distinct hyperplanes in $\R^n$ and $n \geq 3$. By using this we can inductively construct larger and larger subgroups of $G_\cK$ inside $G$.

Let us now deal with $(\rm ii)$. We first assume that $\cH_1$ and $\cH_2$ are two parallel subspaces, i.e.\ $\cH_1^0=\cH_2^0=E$. By replacing $\cH$ by $\cH/E$, the problem reduces to the easy fact that $G_{\{x\}} \cup G_{\{y\}}$ generates $\Isom(\cH)$ whenever $x,y$ are two different points of $\cH$ and $\dim \cH \geq 2$. Next, we assume that $\cH_1$ and $\cH_2$ are not parallel. For example, assume that $\cH_1^0 \neq E$. Take $x \in \cH_1$ and let $\cL$ the closed affine subspace spanned by $\cH_2$ and $x$. Observe that $\cH_1 \cap \cL=x+E$. Indeed, the inclusion $x+ E \subset \cH_1 \cap \cL$ is obvious. Conversely, if $y \in \cH_1 \cap \cL$, then we can write $y=z+\lambda(x-P_{\cH_2}(x))$ where $z \in \cH_2$ and $\lambda \in \R$. We have $y-z=\lambda(x-P_{\cH_2}(x))$. This shows that $\cH_1-\cH_2$, which is an affine subspace of $\cH^0$, contains both $x-P_{\cH_2}(x)$ and $\lambda(x-P_{\cH_2}(x))$. By assumption, $\cH_1-\cH_2$ does not contain $0$. This forces $\lambda=1$. We thus get $y-z=x-P_{\cH_2}(x)$. We obtain $y-x=z-P_{\cH_2}(x) \in \cH_1^0 \cap \cH_2^0=E$. We conclude that $y \in x + E$ as we wanted.

Now, since $\cH_2 \subset \cL$, we have $G_{\cL} \subset G_{\cH_2} \subset G$. Since $\cH_1 \cap \cL=x+E$, we conclude by item $(\rm i)$, that $G$ contains $G_{x+E}$. Take $x' \in \cH_1$ such that $x'-x \notin E$. Then we also have $G_{x'+E} \subset G$. By the first case of item $(\rm ii)$, we conclude that $G$ contains $H_E$.
\end{proof}

\begin{prop} \label{groups and decreasing subspaces}
Let $\cH$ be an affine Hilbert space and $(\cH_i)_{i \in I}$ a decreasing net of affine subspaces. Define $\cK =\bigcap_{i \in I} \cH_i$ and $E= \bigcap_{i \in I} \cH_i^{0}.$ Let $G \subset \Isom(\cH)$ be the closed subgroup generated by $\bigcup_i G_{\cH_i}$. If $\cK \neq \emptyset$ then $E=\cK^0$ and $G=G_{\cK}$. If $\cK=\emptyset$, then $G=H_{E}$.
\end{prop}
\begin{proof}
If $\cK \neq \emptyset$, this is easy because every element $g \in G_{\cK}$ can be approximated by an isometry which fixes a subspace $\cK'$ with finite codimension. Since $\cK' \supset \cH_i$ for $i$ large enough, the conclusion follows. 

Next we deal with the more involved case where $\cK=\emptyset$. We will show that $G$ contains all translations by elements of $E^{\perp}=\left( \bigcap_i \cH_i^0 \right)^{\perp}=\overline{\bigcup_i (\cH_i^0)^{\perp}}$ by approximating them by rotations with very small angle. First, let $x \in \cH$ and let $x_i$ be the orthogonal projection of $x$ on $\cH_i$. By the Proposition \ref{empty intersection diverge}, we have $\lim_i \| x_i-x\| = +\infty$. Fix $i$ and pick $\xi \in (\cH_i^0)^{\perp}$. Since $x_j-x_i$ is orthogonal to $\xi$ and $\|x_j - x\| \to \infty$, we know that asymptotically, $x_j-x$ and $\xi$ are almost orthogonal. Consider the rotation $R_j \in \mathcal{O}(\cH^0)$ of the plane generated by $\xi$ and $x-x_j$ and with angle $\theta_j=\| \xi \|/ \|x-x_j\|$. Since $\theta_j \to 0$ and $\theta_j \sim \sin\theta_j$, we get $R_j(x-x_j) \sim (x-x_j)+\xi$. Let $g_j \in G_{\cH_j}$ be the isometry which fixes $x_j$ and with linear part $R_j$. Then we get $g_j x \to x+ \xi$. Since we also have $g_j^0 \to \id$, we conclude easily that $g_j$ converges in the pointwise topology to the translation by $\xi$. Therefore, $G$ contains the translation by $\xi$. This holds for every $\xi \in (\cH_i^0)^{\perp}$ and every $i \in I$. Thus it holds for all $\xi \in E^{\perp}$. Finally, let us show that $H_{E} \subset G$. Fix a point $x \in \cH$. Let $\cH'_i=x+\cH_i^0$ be the subspace parallel to $\cH_i$ and which contains $x$. Since $G$ contains $G_{\cH_i}$ and every translation by an element of $(\cH_i^0)^{\perp}$, then $G$ also contains $G_{\cH'_i}$. Since $\bigcap_i \cH'_i=x+E$ then $G$ contains $G_{x+E}$ by the first part of the proof. Finally, every element of $H_{E}$ can be written as an element of $G_{x+E}$ composed by a translation by a vector in $E^{\perp}$.
\end{proof}

\subsection{Inflations of an affine Hilbert space}
 We now explain a construction which will play a key role in our study. Let $\cH$ be an affine Hilbert space. By rescaling the metric of $\cH$ by a parameter $t > 0$, we obtain a new affine Hilbert space $\cH^t$. The elements of $\cH^t$ are formally denoted $tx$ for $x \in \cH$ so that $\| tx-ty\|=t\|x-y\|$. The space $\cH^t$ has the same tangent space $\cH^0$ but with a rescaled action so that $tx-ty=t(x-y)$ for all $x,y \in \cH$. We have a natural identification $(\cH^t)^s=\cH^{st}$ for all $s,t> 0$ with the obvious relation $s(tx)=(st)x$ for all $x \in \cH$. For every isometry $g \in \Isom(\cH)$, we can define an isometry $g^t \in \Isom(\cH^t)$ by the formula $g^t(tx)=tg(x)$. Then $\Isom(\cH) \ni g \mapsto g^t \in \Isom(\cH^t)$ is an isomorphism of topological group. For all $s,t \geq 0$, we have $(g^t)^s=g^{ts}$ with the obvious identifications.

\begin{rem}
When $H$ is a linear Hilbert space, we of course have a natural identification $H=H^t$ for all $t > 0$. However, when $\cH$ is an affine Hilbert space, there is no \emph{natural} affine isometry between $\cH$ and $\cH^t$ for $t \neq 1$. In other words, the functor $\cH \mapsto \cH^t$ is a nontrivial self-equivalence of the category of affine Hilbert spaces.
\end{rem}

The following observation will be crucial to relate the Gaussian actions $\widehat{\alpha}^t$ for different values of $t$.

\begin{prop} \label{Rotation trick}
Take $\theta \in [0,\frac{\pi}{2}]$ and let $t=\cos \theta$ and $s=\sin \theta$. Then the formula $$ R_\theta(x,\xi)=(t x+s\xi, sx- t\xi), \quad (x,\xi) \in \cH \times \cH^0$$
defines an isometry $R_\theta : \cH \times \cH^0 \rightarrow \cH^t \times \cH^s$. Moreover, the isometry $R_\theta$ is equivariant with respect to the diagonal actions of $\Isom(\cH)$, i.e.\ for all $g \in \Isom(\cH)$, we have
$$ R_\theta \circ (g \times g^0)=(g^t \times g^s) \circ R_\theta$$
\end{prop}
\begin{proof}
This follows from the two equations
$$ g^t(tx+s\xi)=g^t(tx)+g^0(s\xi)=tg(x)+sg^0(\xi),$$
$$ g^s(sx-t\xi)=g^s(sx)-g^0(t\xi)=sg(x)-tg^0(\xi).$$
\end{proof}

\subsection{Affine isometric actions}
An \emph{affine isometric action} $\alpha : G \curvearrowright \cH$ of a locally compact group $G$ on an affine Hilbert space $\cH$ is a continuous group homomorphism $\alpha : G \rightarrow \Isom(\cH)$. By taking linear parts $\alpha^0 : g \mapsto (\alpha_g)^0$ we obtain an orthogonal representation $\alpha^0 : G \rightarrow \mathcal{O}(\cH^0)$ called the \emph{linear part} of $\alpha$. If $\alpha : G \curvearrowright \cH$ is an affine isometric action, then we can form a new affine isometric action $\alpha^t : G \curvearrowright \cH^t$ defined by $\alpha_g^t(tx)=t\alpha_g(x)$. We have $(\alpha^t)^s=\alpha^{ts}$ for all $t,s \geq 0$ with the obvious identifications.

We say that $x \in \cH$ is a \emph{fixed point} of $\alpha$ if $\alpha_g(x)=x$ for all $g \in G$. We recall that $\alpha$ admits a fixed point if and only if for some (hence any) $y \in \cH$, the orbit $G \cdot y$ is bounded.

We say that  $\alpha$ has \emph{almost fixed points} if there exists a sequence $(x_n)_{n \in \N}$ in $\cH$ such that the sequence $\|gx_n-x_n \|$ converges to $0$ uniformly on compact subsets of $G$.

We say that $\alpha$ is (metrically) \emph{proper} if $\lim_{g \to \infty} \|gx-x\|=+\infty$ for some (hence any) $x \in \cH$.

\subsubsection*{Restricted action} If $\cK \subset \cH$ is an affine subspace which is invariant by the action $\alpha$ then we can restrict $\alpha$ to an action $\alpha|_{\cK} : G \curvearrowright \cK$.

\subsubsection*{Product action}
If $\alpha : G \curvearrowright \cK$ and $\beta : G \curvearrowright \cL$ are two affine isometric actions, then we can define naturally a new affine isometric action $\alpha \times \beta : G \curvearrowright \cK \times \cL$. 

\subsubsection*{Projected action}
If $E \subset \cH^0$ is a closed linear subspace which is invariant under $\alpha^0$ then we can define a \emph{projected action} $\alpha/E : G \curvearrowright \cH/E$ by the formula $g(x+E)=gx + E$ for all $x \in \cH$. Note that $F=E^{\perp}$ is also $\alpha^0$-invariant and we have $\alpha=(\alpha/E) \times (\alpha/F)$ on $\cH=(\cH/E) \times (\cH/F)$.

\subsubsection*{The harmonic seminorm} \label{mu-harmonic norm}
Let $\alpha : G \curvearrowright \cH$ be an affine isometric action. Let $\mu$ be a probability measure on $G$. The \emph{$\mu$-harmonic seminorm} of $\alpha$ is defined by
$$ \| \alpha \|_\mu^2 =\inf_{x \in \cH} \int_G \| gx-x\|^2 \rd \mu(g).$$
Observe that if $\mu$ is compactly supported and $\alpha$ has almost fixed points then $\| \alpha\|_\mu=0$.

\subsection{Relation with $1$-cohomology}
We recall the relationship between affine isometric actions and the $1$-cohomology of orthogonal representations \cite[Section 2.2]{BHV08}.

Let $\pi : G \rightarrow \mathcal{O}(H)$ be an orthogonal representation on a linear Hilbert space $H$. A \emph{$1$-cocycle} for $\pi$ is a continuous map $c : G \rightarrow H$ such that $c(gh)=c(g)+\pi(g) c(h)$ for all $g,h \in G$. For every $\xi \in H$, the function $\partial_\xi : g \mapsto \partial_\xi(g)=\pi(g)\xi-\xi$ is a $1$-cocycle, which is called a \emph{$1$-coboundary}. We denote by $Z^1(\pi)$ the space of all $1$-cocycles and by $B^1(\pi)$ the space of all $1$-coboundaries. We equip $Z^1(\pi)$ with the topology of uniform convegence on compact subsets of $G$. The function $\partial : \xi \mapsto \partial_\xi \in Z^1(\pi)$ is continuous but its image is not closed in general. We denote by $H^1(\pi)=Z^1(\pi)/B^1(\pi)$ the quotient space. It is called the \emph{$1$-cohomology} of $\pi$. The quotient $\overline{H}^1(\pi)=Z^1(\pi)/\overline{B^1(\pi)}$ is called the \emph{reduced $1$-cohomology}.

For every  $c \in Z^1(\pi)$, the formula $\alpha_g(\xi)=\pi(g) \xi + c(g)$ defines an affine isometric action $\alpha : G \curvearrowright H$. Observe that $\alpha$ has a fixed point if and only if $c \in B^1(\pi)$. Similarly, $\alpha$ has almost fixed points if and only if $c \in \overline{B}^1(\pi)$. For every $t > 0$, under the natural identification $H^t=H$, the affine isometric action $\alpha^t : G \curvearrowright H$ is given by the formula $\alpha^t_g(\xi)=\pi(g)\xi + tc(g)$, i.e.\ the cocycle $c$ is simply scaled by $t$. If $H=E \oplus F$, under the natural identification $H/E=F$, the projected action $\alpha/E : G \curvearrowright F$ is given by the formula $(\alpha/E)_g(\xi)=\pi(g)\xi+P_Fc(g)$ for all $g \in G$ where $P_F : H \rightarrow F$ is the orthogonal projection. Finally, the $\mu$-harmonic seminorm of $\alpha$ of Definition \ref{mu-harmonic norm} corresponds to the norm of the \emph{harmonic part} of $c$ as in \cite{EO18} (when $\mu$ is cohomologically adapted).

\medskip

Conversely, let $\alpha : G \curvearrowright \cH$ be an affine isometric action on some affine Hilbert space $\cH$ and let $\pi=\alpha^0 : G \rightarrow \mathcal{O}(\cH^0)$ be its linear part. For every $x$, we can define a \emph{cocycle} $\partial_x \in Z^1(\pi)$ by the formula $\partial_x(g)=gx-x$. Observe that if $y \in \cH$ is another point, then $\partial_x=\partial_y + \partial_\xi$ where $\xi=x-y \in \cH^0$. In particular, the cohomology class $[\partial_x]$ does not depend on the choice of $x$ but only on $\alpha$. If we choose $x$ to be the origin of $\cH$ and thus identify $\cH$ with $\cH^0$, then $\alpha$ is identified with the affine isometric action asssociated to $\partial_x$ in the previous paragraph. Therefore, we have a complete correspondance between affine isometric actions and the $1$-cohomology of orthogonal representations.

\subsection{Evanescent affine isometric actions} \label{section evanescent}

A systematic study of irreducible affine isometric actions can be found in \cite{BPV16}. Here we introduce another notion which is somehow the complete opposite.

\begin{df}
Let $\alpha : G \curvearrowright \cH$ be an affine isometric action. The \emph{support} of $\alpha$ is the linear subspace of $\cH^0$ defined by
$$ \supp(\alpha)=\bigcap_{\cK \in \mathcal{E}} \cK^0$$
where $\mathcal{E}$ is the set of all $\alpha$-invariant affine subspaces of $\cH$. If $\supp(\alpha)=\cH^0$, we say that $\alpha$ is \emph{irreducible}. If $\supp(\alpha)=\{0\}$ we say that $\alpha$ is \emph{evanescent}.
\end{df}

\begin{rem} \label{irreducible times evanescent}
Let $\alpha : G \curvearrowright \cH$ be any affine isometric action. Then $\alpha$ can be decomposed as a product of an irreducible action and an evanescent action. Indeed, let $E=\supp(\alpha)$ and $F=E^{\perp}$. Then the projected action $\alpha/F : G \curvearrowright \cH/F$ is irreducible while $\alpha/E : G \curvearrowright \cH/E$ is evanescent.
\end{rem}

\begin{prop} \label{nilpotent evanescent}
Let $\alpha : G \curvearrowright \cH$ be an affine isometric action.  Suppose that $G$ is nilpotent and $\alpha^0$ has no invariant vectors. Then $\alpha$ is evanescent.
\end{prop}
\begin{proof}
By \cite[Corollary 4.21]{BPV16}, such an action $\alpha$ cannot be irreducible. Since the same holds for all of its nontrivial projected actions, we conclude by Remark \ref{irreducible times evanescent} that $\alpha$ is evanescent.
\end{proof}

\begin{prop} \label{amenable evanescent}
Let $\alpha : \Gamma \curvearrowright \cH$ be an affine isometric action of a discrete group $\Gamma$. Suppose that $\Gamma$ is amenable and $\alpha^0$ is contained in a multiple of the left regular representation. Then $\alpha$ is evanescent.
\end{prop}
\begin{proof}
By \cite[Theorem 4.31]{BPV16}, such an action $\alpha$ cannot be irreducible. Since the same holds for all of its nontrivial projected actions, we conclude by Remark \ref{irreducible times evanescent} that $\alpha$ is evanescent.
\end{proof}

The following proposition gives a good picture of evanescent actions which was not obvious from the definition. We are grateful to Stefaan Vaes for noticing a mistake in the proof in an earlier version of the paper.

\begin{prop} \label{evanescent small subspaces}
Let $\alpha : G \curvearrowright \cH$ be an affine isometric action. Then $\alpha$ is evanescent if and only if there exists a decreasing sequence $(\cH_n)_{n \in \N}$ of $\alpha$-invariant affine subspaces of $\cH$  such that $\bigcap_n \cH_n^0=\{0\}$.
\end{prop}
\begin{proof}
The if direction is obvious. Let us prove the other one. Let $\mathcal{E}$ the set of all $\alpha$-invariant affine subspaces of $\cH$. We say that an orthogonal projection $P \in \B(H^0)$ is \emph{reducing} if there exists $\cK \in \mathcal{E}$ such that $P=P_{\cK}^0$ where $P_{\cK} : \cH \rightarrow \cH$ is the orthogonal projection onto $\cK$. We just have to show that there exists a decreasing sequence of reducing projections $P_n \in \B(\cH^0)$ such that $P_n$ converges strongly to $0$. For this, we will use the idea of \cite[Proposition 3.6]{BPV16} which shows that if an affine transformation $T : \cH \rightarrow \cH$ commutes with $\alpha$ then $1_{[0,\varepsilon]}(|T^0-1|)$ is a reducing projection for all $\varepsilon > 0$.

For every finite set $F \subset \mathcal{E}$ and every $\varepsilon > 0$, define the reducing projection $Q_{(F,\varepsilon)}=1_{[1-\varepsilon,1]}(T^0)$ where $T: \cH \rightarrow \cH$ is the affine transformation given by
$$T = \frac{1}{|F|} \sum_{\cK \in F} P_{\cK}.$$
Note that $T$ commutes with $\alpha$ because each $P_{\cK}$ does (by definition of $\mathcal{E}$). Now, take $\mathcal{V} \subset \B(\cH^0)$ an open neighbourhood of $0$ in the strong topology. Since $\supp(\alpha)=\{0\}$, there exists a finite subset $F \subset \mathcal{E}$ such that $P \in \mathcal{V}$ where $P$ is the orthogonal projection onto $\bigcap_{\cK \in F} \cK^0$. Observe that $\lim_{\varepsilon \to 0} Q_{(F,\varepsilon)}=P$ in the strong topology. In particular, we have $Q_{(F,\varepsilon)} \in \mathcal{V}$ for $\varepsilon > 0$ small enough. Thus, we have shown that there exists reducing projections which are arbitrarily small in the strong topology. Now, we need to show that we can take them to be decreasing. Let $P,Q \in \B(\cH^0)$ be two reducing projections. Then $PQP$ is the linear part of some affine transformation which commutes with $\alpha$. Therefore $P'=1_{[\frac{1}{2},1]}(PQP)$ is again a reducing projection which is below $P$. Moreover, if $Q \to 0$ strongly, then also $P' \to 0$ strongly (because $P' \leq 2 PQP$). This shows that we can find arbitrarily small reducing projections below $P$. By using this observation, we can construct inductively a decreasing sequence of reducing projection $P_n \in \B(\cH^0)$ such that $\|P_n \xi_m\| \leq 2^{-n}$ for all $m \leq n$ where $(\xi_n)_{n \in \N}$ is some dense sequence of vectors in $\cH^0$,  and we get the desired conclusion.
\end{proof}

\begin{rem}
The last part of the proof of Proposition \ref{evanescent small subspaces} shows that if $\alpha$ is evanescent and $\cK \in \mathcal{E}$, then $\alpha|_{\cK}$ is again evanescent. This also was not obvious from the definition.
\end{rem}

\begin{prop} \label{evanescent properties}
Let $\alpha : G \curvearrowright \cH$ be an affine isometric action. Suppose that $\alpha$ is evanescent. Then $\alpha$ has almost fixed points. If moreover $\alpha^0$ is mixing and $\alpha$ has no fixed point, then $\alpha$ is proper.
\end{prop}
\begin{proof}
By Proposition \ref{evanescent small subspaces} there exists a sequence of $\alpha$-invariant subspaces $\cH_n \subset \cH$ such that $P_n^0$ converges strongly to $0$ where $P_n$ is the orthogonal projection on $\cH_n$. Take $x \in \cH$ and set $x_n=P_n(x)$. Then for all $g \in G$, we have $gx_n-x_n = gP_n(x)-P_n(x) = P_n^0(gx-x)$. Since $P_n^0$ converges strongly to $0$, we have $\lim_n \|P_n^0(gx-x)\|=0$ uniformly on compact subsets of $G$. Thus $(x_n)_{n \in \N}$ is a sequence of almost fixed points.

Now, suppose that $\alpha^0$ is mixing and $\alpha$ has no fixed points. For every $n \in \N$, we have
\begin{align*}
\| gx-x\|^2	& \geq \|(gx-x)-P_n^0(gx-x)\|^2\\
				&= \|(gx-x)-(gx_n-x_n)\|^2\\
				&= \|g^0(x-x_n)-(x-x_n)\|^2\\
 				&= 2\|x-x_n\|^2-2\langle g^0(x-x_n),x-x_n \rangle.
\end{align*}
Since $\alpha^0$ is mixing, this implies that $\liminf_{g \to \infty} \|gx-x\|^2 \geq 2\|x-x_n\|^2$ for every $n \in \N$. Since $(x_n)_{n \in \N}$ is a sequence of almost fixed points and $\alpha$ has no fixed point, we must have $\lim_n \|x-x_n\|=+\infty$. We conclude that $\lim_{g \to \infty} \|gx-x\|=+\infty$, i.e.\ $\alpha$ is proper.
\end{proof}



The usual method to construct an affine isometric action without fixed point from a representation with almost invariant vectors (see for instance \cite[Proposition 2.4.5]{BHV08}) actually produces evanescent affine isometric actions.

\begin{prop} \label{existence evanescent}
Let $G$ be a locally compact group and $\pi$ be an orthogonal representation of $G$ which has almost invariant vectors but no invariant vector. Then there exists an evanescent affine isometric action $\alpha : G \curvearrowright \cH$ with no fixed point such that $\alpha^0$ is contained in a multiple of $\pi$. Moreover, for any proper function $f : G \rightarrow [1,+\infty)$, we can choose $\alpha$ so that $\| \alpha_g(x)-x\| \leq f(g)$ for all $g \in G$ and some $x \in \cH$.
\end{prop}
\begin{proof}
Choose inductively a sequence of unit vectors $(\xi_n)_{n \in \N}$ in $H_\pi$ such that $$ \| g \xi_n-\xi_n \|^2 \leq  2^{-n}$$ for all $g \in G$ such that $f(g)^2 \leq n$. Now let $F(n)=\frac{1}{\sqrt{n}} \xi_n$. Observe that for all $g$ such that $f(g)^2 \leq m$, we have
$$ \sum_n \|\pi(g) F(n)-F(n)\|^2 \leq \sum_{n < m} \|\pi(g) F(n)-F(n)\|^2 + \sum_{n \geq m} \frac{1}{n} 2^{-n} < +\infty.$$
This shows that $c(g)=\left( \pi(g)F(n)-F(n) \right)_{n \in \N} \in \cH:=\ell^2( \N,H_\pi)$ for all $g \in G$. It also shows that the map $c: G \rightarrow \cH$ is continuous. Thus $c$ is a $1$-cocycle for the orthogonal representation $\pi^{\oplus \N}$ on $\cH$. Let $\alpha : G \curvearrowright \cH$ be the affine isometric action associated to this cocycle $c$, i.e.\ $\alpha_g(\xi)=\pi^{\oplus \N}(g) \xi + c(g)$ for all $\xi \in \cH$. We show first that $\alpha$ is evanescent. Indeed, for all $m \in \N$, let 
$$ \cK_m =\{ (\eta_n)_{n \in \N} \in \cH \mid \forall n \leq m, \; \eta_n=-F(n) \}.$$
Then $\cK_m$ is an $\alpha$-invariant subspace of $\cH$ and clearly, we have $\bigcap_{m \in \N} \cK_m^0=\{0\}$. Moreover, since $$\sum_{n \in \N} \|F(n)\|^2=\sum_{n \in \N} \frac{1}{n}=+\infty,$$ we also have $\bigcap_{m \in \N} \cK_m=\emptyset$ by  Proposition \ref{empty intersection diverge}. Suppose that $\alpha$ had a fixed point $x \in \cH$. Then the projection $x_m$ of $x$ on  $\cK_m$ is also a fixed point for all $m \in \N$. Since $\pi$, hence also $\alpha^0$, has no invariant vectors, this forces $x=x_m$ for all $m \in \N$. Thus $x \in \bigcap_m \cK_m=\emptyset$ which is a contradiction. We conclude that $\alpha$ has no fixed point.

Finally, for every $g \in G$, by taking $m$ to be the smallest integer such that $f(g)^2 \leq m$, we have
\begin{align*}
\|c(g)\|^2   &\leq \sum_{n < m} \|\pi(g) F(n)-F(n)\|^2 + \sum_{n \geq m} \frac{1}{m} 2^{-n}  \\
& \leq \sum_{n < m} \frac{2}{n} + \sum_{n \geq m} \frac{1}{n} 2^{-n}\\
& \leq 2m+2 \\
&\leq 2f(g)^2 + 4\\
&\leq 6f(g)^2
\end{align*} 
Thus, up to replacing $c$ by $\frac{1}{3}c$, we get $\|c(g)\| \leq f(g)$ for all $g \in G$.
\end{proof}

\section{The Poincar\'e exponent}

In the study of discrete isometry groups $\Gamma < \Isom(X)$ where $X$ is a metric space (usually a hyperbolic space, or a tree as in Section \ref{trees}), the \emph{Poincar\'e exponent} of $\Gamma$ is an important invariant which measures the ``size" of $\Gamma$. It plays a central role in the Patterson-Sullivan theory (see Section \ref{trees}). This exponent $\delta(\Gamma) \in [0,+\infty]$ is defined as the infinimum of all $s > 0$ such that
$$ \sum_{g \in \Gamma} e^{-sd(gx,y)} < +\infty$$
where $x,y$ is any given pair of points of $X$. This series is called the \emph{Poincar\'e series} of $\Gamma$ and $\delta(\Gamma)$ is its exponent of convergence. 

We will now define an analog of the Poincar\'e exponent for affine isometric actions on Hilbert spaces. We will later use it to determine when a Gaussian action is recurrent or dissipative. Note that in the discrete case, this Poincar\'e exponent already appears implicitely in \cite[Proposition 4.1]{VW18}.

\begin{df}
Let $\alpha : G \curvearrowright \cH$ be an affine isometric action. The \emph{(quadratic) Poincar\'e exponent} $\delta(\alpha) \in [0,+\infty]$ of this action is the exponent of convergence of the \emph{(quadratic) Poincar\'e integral}
$$ P_s(x,y,m)=\int_G e^{-s\|gx-y\|^2} \rd m(g), \: s > 0$$
where $x,y$ is any pair of points of $\cH$ and $m$ is any (left or right) Haar measure on $G$.
\end{df}

\begin{prop} \label{poincare well-defined}
The Poincar\'e exponent $\delta(\alpha)$ does not depend on the choice of the Haar measure $m$ and the points $x,y \in \cH$.
\end{prop}
\begin{proof}
First, we show that $\delta$ does not depend on the choice of $y$. Take $\xi \in \cH^0$. Fix $s > 0$. We have
$$ P_s(x,y+\xi,m)=\int_G e^{-s\|gx-y-\xi\|^2} \rd m(g).$$
For any $t > s$, there exists a constant $C > 0$ such that
 $$\forall g \in G, \; s\| g x-y-\xi\|^2 \leq t\| g x-y \|^2 + C.$$
Therefore, if $P_s(x,y+\xi,m)$ converges then $P_{t}(x,y,m)$ also converges. Thus $\delta(\alpha)$ does not depend on the choice of $y$. Now, observe that $P_s(x,y,m)=P_s(y,x,m^{\sharp})$ where $m^{\sharp}$ is the Haar measure defined by $\rd m^{\sharp}(g)=\rd m(g^{-1})$. This easily implies that $\delta(\alpha)$ does not depend on the choice of $x$ and of $m$ (if $m$ is a right Haar measure then $m^{\sharp}$ is a left Haar measure).
\end{proof}

\begin{rem}
A priori the convergence or divergence of $P_s(x,y,m)$ at $s=\delta(\alpha)$ might depend on the choice of $x,y$ and $m$.
\end{rem}

\begin{lem} \label{growth exponent}
Let $(X,\mu)$ be a measure space (possibly infinite) and $f : X \rightarrow \R_+$ a positive function. Let $M(t)=\mu\{x \in X \mid f(x) \leq t \}$. Then we have
$$\inf \{ s > 0 \mid \int_X e^{-sf} \rd \mu < +\infty \}= \uplim_{t \to \infty} \frac{1}{t}\log M(t).$$
In particular, for any affine isometric action $\alpha : G \curvearrowright \cH$, we have
$$ \delta(\alpha)=\uplim_{t \to \infty} \frac{1}{t} \log m( \{ g \in G \mid \| gx-y\|^2 \leq t\})$$
where $x,y$ is any pair of points in $\cH$ and $m$ is any Haar measure (left or right) on $G$.
\end{lem}

\begin{prop} \label{uniformly discrete cocompact}
Let $\alpha : G \curvearrowright \cH$ be an affine isometric action and $\Lambda \subset G$ any countable subset of $G$. Take $x,y \in \cH$ and let $d$ be the exponent of convergence of the series
$$ \sum_{g \in \Lambda} e^{-s\|gx-y\|^2}.$$
Then $d \leq \delta(\alpha)$ if $\Lambda$ is uniformly discrete and $d \geq \delta(\alpha)$ if $\Lambda$ is cocompact.
\end{prop}
\begin{proof}
Take $x,y \in \cH$ and $m$ a left Haar measure. Suppose that $\Lambda$ is uniformly discrete. Then we can find a compact neighborhood of the identity $K$ in $G$ such that the sets $(gK)_{g \in \Lambda}$ are pairwise disjoint. We have
$$\int_G e^{-s\|gx-y\|^2} \rd m(g) \geq \sum_{g \in \Lambda} \int_K e^{-s\|ghx-y\|^2} \rd m(h).$$
Take $t > s$. Since $\sup_{h \in K} \|hx-x\| < +\infty$, there exists a constant $C > 0$ such that for all $g \in G$ and all $h \in K$, we have
$$ s\|ghx-y\|^2 \leq t\|gx-y\|^2+C.$$
This shows that if
$$\int_G e^{-s\|gx-y\|^2} \rd m(g) < +\infty$$
then 
$$\sum_{g \in \Lambda} e^{-t\|gx-y\|^2}m(K)e^{-C} <+ \infty.$$
Since $t > s$ is arbitrary, we conclude that $d \leq \delta(\alpha)$.

Suppose that $\Lambda$ is cocompact. Take a compact set $K \subset G$ such that $\bigcup_{g \in \Lambda} gK=G$. Then we have
$$\int_G e^{-s\|gx-y\|^2} \rd m(g) \leq \sum_{g \in \Lambda} \int_K e^{-s\|ghx-y\|^2} \rd m(h).$$ 
Take $t < s$. Then there exists a constant $C > 0$ such that for all $g \in G$ and all $h \in K$, we have
$$ s\|ghx-y\|^2 \geq t\|gx-y\|^2-C.$$
This shows that if
$$\int_G e^{-s\|gx-y\|^2} \rd m(g) = +\infty$$
then 
$$\sum_{g \in \Lambda} e^{-t\|gx-y\|^2}m(K)e^C =+ \infty.$$
Since $t < s$ is arbitrary, we conclude that $d \geq \delta(\alpha)$.
\end{proof}

\begin{prop}
Let $\alpha : G \curvearrowright \cH$ be an affine isometric action and $G_0 \subset G$ a closed subgroup. Then $\delta(\alpha|_{G_0}) \leq \delta(\alpha)$ and equality holds if $G_0$ is cocompact in $G$.
\end{prop}
\begin{proof}
Take $\Lambda \subset G_0$ a uniformly discrete cocompact subset. Then $\Lambda$ is uniformly discrete in $G$ and if $G_0$ is cocompact then $\Lambda$ is also cocompact in $G$. We conclude by Proposition \ref{uniformly discrete cocompact}.
\end{proof}

\begin{prop}
Let $\alpha : G \curvearrowright \cH$ be an affine isometric action. If $\delta(\alpha) < +\infty$ then $\alpha$ is proper.
\end{prop}
\begin{proof}
Take $\Lambda \subset G$ a uniformly discrete cocompact subset. Then by Proposition \ref{uniformly discrete cocompact}, there exists $s > 0$ such that $\sum_{g \in \Lambda} e^{-s\|gx-x\|^2} < +\infty$ for some $x \in \cH$. In particular, $\lim_{g \to \infty, g \in \Lambda} \|gx-x\|=+\infty$. Since $\Lambda$ is cocompact in $G$, this easily implies that $\alpha$ is proper.
\end{proof}

The following fact has already been observed in the proof of \cite[Theorem 4.1]{CTV08} (see also the proof of Lemma 5.2 in \cite{VW18}). Later we will see a more quantitative statement which gives a lower bound on $\delta(\alpha)$ in terms of the spectral radius of $G$ (Corollary \ref{lower bound amenability}). See also Corollary \ref{poincare infinite}.

\begin{prop} \label{poincare strictly positive}
Let $\alpha : G \curvearrowright H$ be an affine isometric action of a nonamenable group $G$. Then $\delta(\alpha) > 0$. 
\end{prop}

\section{The Gaussian functor} \label{Affine Gaussian functor}

Let $\cH$ be an affine Hilbert space. Let $A(\cH)$ denote the linear space of all continuous affine maps from $\cH$ to $\R$. Suppose first that $\cH$ is finite dimensional and for every $x \in \cH$, let $\mu_x$ be the standard Gaussian probability measure on $\cH$ centered at $x$. Then every $f \in A(\cH)$ has a normal distribution $\cN(f(x),\|f\|^2)$ with respect to the probability measure $\mu_x$ where $\| f\|:=\| f^0\|$ is the norm of the linear part of $f$. The Gaussian functor generalizes this to the infinite dimensional case and we will now explain this construction.

Let us first assume that we have a \emph{linear} Hilbert space $H$, i.e.\ that we have a specified origin $0 \in H$. The Gaussian functor associates to $H$ a probability space $(X,\mu)$ together with a linear map $\xi \mapsto \widehat{\xi}$ from $H$ to $\rL^0(X,\mu,\R)$ such that:
\begin{itemize}
\item For every $\xi \in H$, the random variable $\widehat{\xi}$ has a normal distribution $\cN(0,\| \xi \|^2)$ with respect to $\mu$.
\item The random variables $(\widehat{\xi})_{\xi \in H}$ generate the $\sigma$-algebra of $X$.
\end{itemize}
There are many ways to construct this probablity space $(X,\mu)$ and the random process $(\widehat{\xi})_{\xi \in H}$. Details can for example be found in \cite{Bo14}. However, the two properties above uniquely characterize the Gaussian functor and this is all what one needs to work with it.

Now suppose that $\cH$ is an affine Hilbert space. If we want to apply the previous construction, we need to chose an arbitrary point $x_0 \in \cH$ in order to identify $\cH$ with the tangent linear Hilbert space $\cH^0$. However, we are only interested in the properties of the Gaussian functor which do \emph{not} depend on the choice of $x_0$. Let $(X,\mu)$ be the Gaussian probability space associated to $\cH^0$ together with the linear map $\xi \mapsto \widehat{\xi}$ from $\cH^0$ to $\rL^0(X,\mu,\R)$. For every $x \in \cH$, define a new probability measure $\mu_x= e^{-\frac{1}{2}\| \xi \|^2 + \widehat{\xi}} \mu$ where $\xi=x-x_0$. One thinks of $\mu_{x_0}=\mu$ as a Gaussian probability measure centered at $x_0$ while the new measures $\mu_x$ are centered at different points $x \in \cH$. We thus obtain a family of equivalent probability measures $(\mu_x)_{x \in \cH}$ on $X$. Define a linear map $f \mapsto \widehat{f}$ from $A(\cH)$ into $\rL^0(X,\mu,\R)$ by letting $\widehat{f}:=\widehat{\xi}+f(x_0)1$ where $\langle \cdot, \xi \rangle$ is the linear part of $f$. It is straightforward to check that $\widehat{f}$ has a normal distribution $\cN(f(x),\|f\|^2)$ with respect to $\mu_x$ for all $x \in \cH$. We conclude from this discussion that there exists a \emph{Gaussian density} $(\mu_x)_{x \in \cH}$ in the sense of the following definition.

\begin{df} \label{Gaussian density}
Let $\cH$ be an affine Hilbert space. Let $(\mu_x)_{x \in \cH}$ be a family of equivalent probability measures on a standard borel space $X$. We say that $(X,(\mu_x)_{x \in \cH})$ is a \emph{Gaussian density} if there exists a linear map $f \mapsto \widehat{f}$ from $A(\cH)$ to $\rL^0(X,\R)$ such that
\begin{enumerate}[ \rm (i)]
\item For all $x \in \cH$ and all $f \in A(\cH)$, the random variable $\widehat{f}$ has a normal distribution $\mathcal{N}(f(x),\|f\|^2)$ with respect to $\mu_x$.
\item The random variables $\{ \widehat{f} \mid f \in A(\cH) \}$ generate the $\sigma$-algebra of $X$.
\end{enumerate}
\end{df}
\begin{rem}
Here, we implicitely use the measure class of $(\mu_x)_{x \in \cH}$ in order to define $\rL^0(X,\R)$.
\end{rem}

Note that we only require the existence of the map $f \mapsto \widehat{f}$. Indeed, if such a map exists then it is unique by the following proposition.

\begin{prop} \label{quadratic buseman}
Let $\cH$ be an affine Hilbert space. Let $(X,(\mu_x)_{x \in \cH})$ be a Gaussian density. Then there exists a unique linear map $A(\cH) \ni f \mapsto \widehat{f} \in \rL^0(X,\R)$ as in Definition \ref{Gaussian density}. Moreover, for all $x,y \in \cH$, we have
$$ \frac{\rd \mu_x}{\rd \mu_y} =e^{-\widehat{b(x,y)}}$$
where $b(x,y) \in A(\cH)$ is defined by
$$ b(x,y) : z \mapsto  \frac{1}{2}(\|z-x\|^2-\|z-y\|^2)=\langle z- \frac{x+y}{2},y-x \rangle.$$
\end{prop}
\begin{proof}
We just have to prove the Radon-Nikodym formula. Indeed, the uniqueness of $f \mapsto \widehat{f}$ follows from it since any non-constant affine function $f \in A(\cH)$ can be realized as $f=b(x,y)$ for some $x,y \in \cH$. 

So let $f \mapsto \widehat{f}$ be any linear map as in Definition \ref{Gaussian density}. Define $\mathcal{A} \subset \rL^0(X,\C)$ to be the linear span of $\{e^{\ri \widehat{f}} \mid f \in A(\cH) \}$. Then $\mathcal{A}$ is a subalgebra of $\rL^\infty(X,\C)$ which is dense in the measure topology thanks to Definition \ref{Gaussian density}.$(\rm ii)$. By Definition \ref{Gaussian density}.$(\rm i)$, we can check for all $f \in A(\cH)$ the following equalities
$$ \int_X e^{-\widehat{b(x,y)}} e^{\widehat{f}} \rd \mu_y=\exp\left( \frac{1}{2}\|f\|^2+f(x) \right) = \int_X e^{ \widehat{f}} \rd \mu_x$$
which imply by analyticity (see \cite[Lemma A.7.8]{BHV08}) that
$$ \int_X e^{-\widehat{b(x,y)}} e^{\ri \widehat{f}} \rd \mu_y=\exp\left( -\frac{1}{2}\|f\|^2+\ri f(x) \right) = \int_X e^{\ri \widehat{f}} \rd \mu_x$$
and the Radon-Nikodym formula follows.
\end{proof}
\begin{rem}
Notice the analogy between the Radon-Nikodym formula of Proposition \ref{quadratic buseman} and the definition of conformal densities in Patterson-Sullivan theory (Definition \ref{conformal density}).
\end{rem}

We now observe that there exists only one Gaussian density up to unique isomorphism.
\begin{prop}  \label{unique gaussian density}
Let $\cH$ be an affine Hilbert space. If $(X,(\mu_x)_{x \in \cH})$ and $(Y,(\nu_x)_{x \in \cH})$ are two Gaussian densities, then there exists a unique (up to null sets) borel map $\pi : X \rightarrow Y$ such that $\pi_* \mu_x=\nu_x$ for all $x \in \cH$.
\end{prop}
\begin{proof}
Define a unital commutative $*$-algebra $\C[\cH]$ generated by unitaries $u(f), \: f \in A(\cH)$ subject to the following relations: $u(f+g)=u(f)u(g)$ for all $f,g \in A(\cH)$ and $u(f)=e^{\ri \lambda} 1$ if $f(x)=\lambda$ for all $x \in \cH$. The map $u(f) \mapsto e^{\ri \widehat{f}}$ defines a $*$-homomorphism $\theta_X : \C[\cH] \rightarrow \rL^\infty(X,\C)$ with dense range. Similarly, we have $*$-homomorphism $\theta_X : \C[\cH] \rightarrow \rL^\infty(Y,\C)$ with dense range. For every $x \in \cH$, the probability measures $\mu_x$ and $\nu_x$ define the same state $\mu_x \circ \theta_X = \nu_x \circ \theta_Y$ on $\C[\cH]$. Therefore, there exists a unique continuous $*$-isomomorphism  $\theta : \rL^\infty(Y,\C) \rightarrow \rL^\infty(X,\C)$ such that $\theta_X = \theta \circ  \theta_Y$. By the correspondence between nonsingular spaces and commutative von Neumann algebras, we know that there exists a nonsingular map $\pi : X \rightarrow Y$ such that $\pi^*=\theta$. The map $\pi$ satisfies $\pi_* \mu_x=\nu_x$ for all $x \in \cH$. Conversely, it follows from the Radon-Nikodym formula of Proposition \ref{quadratic buseman} that any such map $\pi$ must satisfy $\pi^*=\theta$ so that $\pi$ is unique up to null sets.
\end{proof}

 From now on, for any affine Hilbert space $\cH$, we fix a Gaussian density that we denote by $(\widehat{\cH}, (\mu_x)_{x \in \cH})$. The precise way in which it is constructed is irrelevant since it is uniquely characterized up to unique isomorphism by the properties of Definition \ref{Gaussian density}. Note that $\widehat{\cH}$ has a canonical measure class and we will always neglect null-sets with respect to that measure class. We will use the notations and terminology on nonsingular spaces explained in the appendix.  If $f \in A(\cH)$ is given by $f(y)=\langle y-x,\xi \rangle$ for some $x \in \cH$ and $\xi \in \cH^0$, we will use the notation $\widehat{f}(\omega)=\langle \omega-x, \xi \rangle$ for $\omega \in \widehat{\cH}$, where $f \mapsto \widehat{f}$ is the linear map of Definition \ref{Gaussian density}. Fix $x \in \cH$. Then $(\widehat{\cH}^0, (\mu_{y-x})_{y \in \cH})$ is a $\cH$-Gaussian density and therefore there exists a unique nonsingular isomorphism $\theta : \widehat{\cH} \rightarrow \widehat{\cH}^0$ such that $\theta_* \mu_y=\mu_{y-x}$ for all $y \in \cH$. We will use the notation $\theta(\omega)=\omega-x$ for $\omega \in \widehat{\cH}$. We will use the letter $\varphi$ to denote elements of $\widehat{\cH}^0$. For the inverse map of $\theta$ we will use the notation $\theta^{-1}(\varphi)= x+ \varphi$ for $\varphi \in \widehat{\cH}^0$. Finally, we will use the notation $\widehat{f}(\varphi)=\langle \varphi, \xi \rangle$ when $f \in A(\cH^0)$ is given by $f(\eta)=\langle \eta, \xi \rangle$. All these notations are intuitive and compatible with each others. We record the following useful formulas.
\begin{prop} \label{formula Radon-Nikodym derivatives}
Let $\cH$ be an affine Hilbert space. The following holds:
\begin{enumerate}[ \rm (i)]
\item For all $x,y \in \cH$ and a.e.\ $\omega \in \widehat{\cH}$,\begin{align*}
\frac{\rd \mu_{y} }{\rd \mu_x} (\omega ) &= \exp \left( \langle \omega-\frac{x+y}{2}, y-x \rangle \right)\\
 &= \exp \left(-\frac{1}{2}\|y-x\|^2+ \langle \omega-x, y-x \rangle \right).
\end{align*}
\item For all $f \in A(\cH)$ and all $x \in \cH$,
$$ \int_{ \widehat{\cH}} e^{ z\widehat{f}(\omega)} \rd \mu_x(\omega)=\exp \left( \frac{z^2}{2} \|f\|^2+ z f(x) \right).$$
\item For all $x,y \in \cH$, $$ \langle \mu_x^{1/2},\mu_y^{1/2} \rangle =\exp\left( -\frac{1}{8}\|y-x\|^2 \right).$$
\item For all $x,y \in \cH$, we have (see Section \ref{densities Lp spaces} for the notations)
$$\| \mu_x^{1/2}-\mu_y^{1/2}\|^2=2-2\exp\left( -\frac{1}{8}\|y-x\|^2 \right).$$
\end{enumerate}
\end{prop}

We deduce the following useful properties.

\begin{prop} \label{density}
Let $\cH$ be an affine Hilbert space. 
\begin{enumerate}[ \rm (i)]
\item The linear span of $\{\mu_x \mid x \in \cH \}$ is dense in $\rL^1(\widehat{\cH})$.
\item The map $x \mapsto \mu_x$ is a uniformly continuous homeomorphism from $\cH$ onto a closed subset of $\rL^1(\widehat{\cH})$.
\item On bounded subsets of $\cH$, the map $x \mapsto \mu_x$ is continuous with respect to the weak topologies , i.e.\ if $x_n \in \cH$ is a bounded sequence such that $\langle x_n-x, \xi \rangle \to 0$ for all $\xi \in \cH^0$ and some $x \in \cH$, then $\mu_{x_n}(f) \to \mu_x(f)$ for all $f \in \rL^\infty(\widehat{\cH})$.
\end{enumerate}
\end{prop}
\begin{proof}
$(\rm i)$ Take $a \in \rL^\infty(\widehat{\cH})$ such that $\mu_x(a)=0$ for all $x \in \cH$. Then $\mu_x(e^{\widehat{f}} a)=0$ for all $f \in A(\cH)$. By analyticity, we get $\mu_x(e^{\ri \widehat{f}} a)=0$ for all $f \in A(\cH)$. But $\{ e^{\ri \widehat{f}} \mid f \in A(\cH) \}$ is a dense subalgebra of $\rL^\infty(\widehat{\cH})$. Therefore, we must have $a=0$. By the Hahn-Banach theorem, we conclude that  $\{\mu_x \mid x \in \cH \}$ is dense in $\rL^1(\widehat{\cH})$.
 
$(\rm ii)$ This follows from Proposition \ref{formula Radon-Nikodym derivatives} and the following inequalities:
$$\| \mu_x^{1/2}-\mu_y^{1/2}\|^2 \leq  \| \mu_x - \mu_y \| \leq 2\| \mu_x^{1/2}-\mu_y^{1/2} \|.$$
$(\rm iii)$ By the formula of Proposition \ref{formula Radon-Nikodym derivatives}.$(\rm ii)$, we have $\mu_{x_n}(e^{\ri \widehat{f}}) \to \mu_x(e^{\ri \widehat{f}} )$ for all $f \in A(\cH)$. For this, we only use the weak convergence of $x_n$. Now, since $x_n$ is bounded, we have
\begin{align*}
\mu_{x_n}(|a-b|) &=\mu_x(e^{\widehat{b(x,x_n)}} |a-b|) \\
 &\leq \mu_x(e^{2\widehat{b(x,x_n)}})^{1/2} \mu_x(|a-b|^2)^{1/2} \\
 & = e^{3\|x-x_n\|^2} \mu_x(|a-b|^2)^{1/2} \\
 & \leq C  \mu_x(|a-b|^2)^{1/2} 
\end{align*}
for some constant $C > 0$ which does not depend on $n$ and for all $a,b \in \rL^\infty(\widehat{\cH})$. This implies that we can use the density of $\{ e^{\ri \widehat{f}} \mid f \in A(\cH) \}$ to conclude that the convergence $\mu_{x_n}(a) \to \mu_x(a)$ holds for all $a \in \rL^\infty(\widehat{\cH})$.
\end{proof}

\subsection{Functorial properties}
The uniqueness of $(\widehat{\cH}, (\mu_x)_{x \in \cH})$ allows us to derive all of its functorial properties.

\subsubsection*{Isometries} First observe that if $V : \cK \rightarrow \cH$ is a surjective affine isometry, then $(\widehat{\cH}, (\mu_{Vx})_{x \in \cK})$ is a $\cK$-Gaussian density. Therefore, by Proposition \ref{unique gaussian density}, there exists a unique nonsingular isomorphism $\widehat{V} : \widehat{\cK} \rightarrow \widehat{\cH}$ such that $\widehat{V}_*\mu_x= \mu_{Vx}$ for all $x \in \cK$. In particular, we have a canonical homomorphism
$$ \Isom(\cH) \ni g \mapsto \widehat{g} \in \Aut(\widehat{\cH}).$$
Of course, this homomorphism is continuous.
\begin{prop} \label{continuous functoriality}
Let $\cH$ be an affine Hilbert space. Then the homomorphism $$\Isom(\cH) \ni g \mapsto \widehat{g} \in \Aut(\widehat{\cH})$$ is continuous and its range is closed.
\end{prop}
\begin{proof}
It follows from Proposition \ref{density}, that for every $x$, the map $g \mapsto \widehat{g}_*\mu_x=\mu_{gx}$ is continuous and since the linear span of $\{ \mu_x \mid x \in \cH\}$ is dense in $\rL^1(\widehat{\cH})$, we deduce that $g \mapsto \widehat{g}$ is continuous. Conversely, if $\lim_n \widehat{g_n}=\id$ then $\lim_n \mu_{g_nx}=\lim_n \widehat{g_n}_*\mu_x=\mu_x$, which means that $\lim_n g_nx=x$ for all $x \in \cH$. This shows that $g \mapsto \widehat{g}$ is a homeomorphism on its range.
\end{proof}

It follows from Proposition \ref{continuous functoriality} that every affine isometric action $\alpha : G \curvearrowright \cH$ of a locally compact group $G$ gives rise to a nonsingular action $\widehat{\alpha} : G \curvearrowright \cH$ that we call the \emph{nonsingular Gaussian action} associated to $\alpha$.

\begin{prop} \label{faithful free}
Let $\cH$ be an affine Hilbert space and let $g \in \Isom(\cH)$ with $g \neq \id$. Then $\widehat{g} \in \Aut(\widehat{\cH})$ is essentially free.
\end{prop} 
\begin{proof}
Let $A=\{ \omega \in \widehat{\cH} \mid \widehat{g}(\omega)=\omega \}$. Then for any $x \in \cH$, we have $\frac{\rd \mu_{gx}}{\rd \mu_x}(\omega)=1$ for all $\omega \in A$. If $A$ has positive measure, this forces $gx=x$ by Proposition \ref{formula Radon-Nikodym derivatives}. Since this holds for every $x \in \cH$, we conclude that $g=\id$.
\end{proof}

\begin{rem} \label{freeness}
Proposition \ref{faithful free} shows in particular that for every faithful affine isometric action $\alpha : \Gamma \curvearrowright \cH$ of a discrete group $\Gamma$, the Gaussian action $\widehat{\alpha}$ is essentially free. This, however, does not generalize to arbitrary locally compact groups even for pmp Gaussian actions. Indeed, a very easy counter-example is given by the action of $\mathrm{O}(3)$ on $\R^3$ which is faithful but not essentially free (every point has a nontrivial stabilizer isomorphic to $\mathrm{O}(2)$). However, in the proof \cite[Proposition 1.2]{AEG94}, it is shown that if $\pi$ is a faithful orthogonal representation of a locally compact group $G$, then the pmp Gaussian action associated to the infinite direct sum $\pi^{\oplus \N}$ is essentially free. We thank Cyril Houdayer and Stefaan Vaes for pointing this out.
\end{rem}

\begin{rem}
The fact that translations act in a nonsingular way on the Gaussian probability space associated to a Hilbert space $\cH$ was already known to many specialists under various forms. For example, in the context of probability theory, it is known as the Cameron--Martin theorem. The Koopman representation of $\Isom(\cH)$ on $\rL^2(\widehat{\cH})$ can also be identified with the well-known Weyl representation of $\Isom(\cH)$ on the Fock space (see \cite{HP84} for example). But surprisingly, the functor that associates to every affine isometric action a nonsingular Gaussian action, which is the central object of this paper, has not been studied before.
\end{rem}


\subsubsection*{Products} 
If $\cH=\cK \times \cK'$ is a product of two affine spaces, then it is easy to see that $(\widehat{\cK} \otimes \widehat{\cK'} , (\mu_x \otimes \mu_{x'})_{(x,x') \in \cH})$ is a Gaussian density. Therefore, there exists a unique nonsingular isomorphism $\theta : \widehat{\cH} \rightarrow \widehat{\cK} \otimes \widehat{\cK'}$ such that $\theta_* \mu_{(x,x')}=\mu_x \otimes \mu_{x'}$ for all $(x,x') \in \cH$. We will often omit the map $\theta$ and simply identify $\widehat{\cH}$ with $\widehat{\cK} \otimes \widehat{\cK'}$ and $\mu_{(x,x')}$ with $\mu_x \otimes \mu_{x'}$. We will also view $\rL^\infty(\widehat{\cK})$ as a subalgebra of $\rL^\infty(\widehat{\cH})$.

\begin{rem} \label{conjugacy different gaussian} From what preceeds and Proposition \ref{Rotation trick}, we deduce the following relations between the one-parameter family of Gaussian actions of a given affine isometric action $\alpha: G \curvearrowright \cH$. Write $t=\cos \theta $ and $s=\sin \theta$ for some $\theta \in ]0,\frac{\pi}{2}[$ and let $R_\theta : \cH \times \cH^0 \rightarrow \cH^t \times \cH^s$ be the isometry defined in Proposition \ref{Rotation trick}. Then it induces a nonsingular isomorphism $\widehat{R}_\theta : \widehat{\cH} \otimes \widehat{\cH}^0 \rightarrow \widehat{\cH}^t \otimes \widehat{\cH}^s$ which is $G$-equivariant with respect to the diagonal actions $\widehat{\alpha} \otimes \widehat{\alpha}^0$ and $\widehat{\alpha}^t \otimes \widehat{\alpha}^s$. More generally, we have $\widehat{\alpha}^{r} \otimes \widehat{\alpha}^0 \cong \widehat{\alpha}^t \otimes \widehat{\alpha}^s$ whenever $r^2=t^2+s^2$.
\end{rem}

\subsubsection*{Subspaces} Let $\cH$ be a affine Hilbert space and $\cK$ a nonempty affine subspace. Let $E=(\cK^0)^{\perp} \subset \cH^0$. Then $\cH=E \times \cK$. Therefore, we have an identification $\widehat{\cH} = \widehat{E} \otimes \widehat{\cK}$ and we can view $\rL^\infty(\widehat{\cK})$ as a subalgebra of $\rL^\infty(\widehat{\cH})$.

But we can do more at the level of the modular bundles (see the appendix). Indeed, since $E$ is a linear Hilbert space, we have a canonical probability measure $\mu_0$ on $\widehat{E}$. This implies that we have a canonical identification $\Mod(\widehat{\cH}) = \widehat{E} \otimes \Mod(\widehat{\cK})$ via the partial trivialization map $\kappa_{\mu_0}$ explained in the appendix. In particular, we can view $\rL^\infty(\Mod(\widehat{\cK}))$ as a subalgebra of  $\rL^\infty(\Mod(\widehat{\cH}))$.

Note that if $\cL$ is another affine subspace of $\cH$ such that $\cL^0=\cK^0$ ($\cL$ and $\cK$ are parallel), then we have $\rL^\infty(\widehat{\cL})=\rL^\infty(\widehat{\cK})$ as subalgebras of $\rL^\infty(\widehat{\cH})$. However, it is important to see that $\rL^\infty(\Mod(\widehat{\cK})) \neq \rL^\infty(\Mod(\widehat{\cL}))$ as subalgebras of $\rL^\infty(\Mod(\widehat{\cH}))$. This means that $\rL^\infty(\widehat{\cK})$ depends only on the tangent space $\cK^0$ while $\rL^\infty(\Mod(\widehat{\cK}))$ really depends on the affine subspace $\cK$ itself.

In the next two proposition, we describe how all these subalgebras interact with each others and with the action of $\Isom(\cH)$. We denote by $\Mod(\widehat{\cH})$ the modular bundle of $\widehat{\cH}$ (see Section \ref{the modular bundle}) and we denote by $\rL^\infty(\Mod(\widehat{\cH}))^G$ the fixed point subalgebra of a given subgroup $G \subset \Isom(\cH)$ acting on $\Mod(\widehat{\cH})$ via the Maharam extension.

\begin{prop} \label{fixed point Maharam}
Let $\cH$ be an affine Hilbert space. 
\begin{enumerate}[ \rm (i)]
\item Let $\cK \subset \cH$ be a closed subspace and let $G_{\cK}=\{ g \in \Isom(\cH) \mid gx=x \text{ for all } x \in \cK \}$. If $\cK$ has infinite codimension, then $$\rL^\infty(\Mod(\widehat{\cH}))^{G_{\cK}}=\rL^\infty(\Mod(\widehat{\cK})).$$
\item Let $E \subset \cH^0$ be a closed subspace. Let $H_{E}=\{ g \in \Isom(\cH) \mid gx-x \in E^{\perp} \text{ for all } x \in \cH \}$. If $E$ has infinite codimension, then 
$$\rL^\infty(\Mod(\widehat{\cH}))^{H_{E}}=\rL^\infty(\widehat{\cL})$$
where $\cL=\cH/E^{\perp}$.
\item Let $E \subset \cH^0$ be a closed subspace and let $T_E=\{ j_\xi \mid \xi \in E\}$ be the associated translation group. If $E$ has infinite dimension, then
$$\rL^\infty(\Mod(\widehat{\cH}))^{T_{E}}=\rL^\infty(\widehat{\cL})$$
where $\cL=\cH/E$. 
\end{enumerate}
\end{prop}
\begin{proof}
$(\rm i)$ Let $E=(\cK^{0})^\perp$. Then we can write $\Mod(\widehat{\cH})=\widehat{E} \otimes \Mod(\widehat{\cK})$ and the action of $G_{\cK}$ on $\Mod(\widehat{\cH})$ can be identified with the action of $\mathcal{O}(E)$ acting only on the first coordinate. Since $E$ has infinite dimension, it is well-known (see \cite{PS12}) that the probability measure preserving action $\mathcal{O}(E) \curvearrowright (\widehat{E},\mu_0)$ is weakly mixing, hence ergodic. The conclusion follows.

$(\rm ii)$ Take $f \in \rL^\infty(\Mod(\widehat{\cH}))^{H_E}$. Let $\cK \subset \cH$ be any subspace such that $\cK^0=E$. Then $G_\cK \subset H_E$. Thus $f \in \rL^\infty(\Mod(\widehat{\cK}))$ by $(\rm i)$. Take $x \in \cK$ and a nonzero $\xi \in E^{\perp}$ and let $\cF=\cK+\xi$ and $y=x+\xi$. Then we also have $f \in \rL^\infty(\Mod(\widehat{\cF}))$. By trivializing $\Mod(\widehat{\cH})$ with repsect to $\mu_x$ and $\mu_y$ (see Section \ref{the modular bundle}), we get that $$f \circ \pi_{\mu_x}^{-1} \in \rL^\infty(\widehat{\cK} \otimes \R^*_+)=\rL^\infty(\widehat{\cL} \otimes \R^*_+)$$ and $$f \circ \pi_{\mu_y}^{-1} \in \rL^\infty(\widehat{\cF} \otimes \R^*_+)=\rL^\infty(\widehat{\cL} \otimes \R^*_+).$$ But it is easily seen, by using the formula for $\frac{\rd \mu_y}{\rd \mu_x}$, that $\theta=\pi_{\mu_x} \circ \pi_{\mu_y}^{-1} \in \Aut( \widehat{\cH} \otimes \R^*_+)$ satisfies 
$$\theta(\rL^\infty( \widehat{\cL} \otimes \R^*_+)) \cap \rL^\infty( \widehat{\cL} \otimes \R^*_+)=\rL^\infty(\widehat{\cL}).$$
We conclude that $f \circ \pi_{\mu_x}^{-1} \in \rL^\infty(\widehat{\cL})$ hence $f \in \rL^\infty(\widehat{\cL})$.

$(\rm iii)$ On a finite-dimensional euclidian space, the whole isometry group and the subgroup of translations both act ergodically and they both preserve the Lebesgue measure. From this fact, it follows that if $F \subset E$ is a finite-dimensional subspace, then every $T_F$-invariant function on $\Mod(\widehat{H})$ is also $H_{F^\perp}$-invariant. Since the union of all $H_{F^\perp} \subset H_E$ for finite-dimensional $F \subset E$ is dense in $H_{E^\perp}$, this implies that every $T_E$-invariant function on $\Mod(\widehat{H})$ is also $H_{E^\perp}$-invariant and the conclusion follows from $(\rm ii)$. 
\end{proof}

\begin{prop} \label{intersection maharam extension}
Let $\cH$ be an affine Hilbert space and $(\cH_i)_{i \in I}$ a family of affine subspaces. Define 
$$\cK =\bigcap_{i \in I} \cH_i, \quad E= \bigcap_{i \in I} \cH_i^{0} \quad \text{and} \quad \cL=\cH/E^{\perp}.$$
\begin{enumerate}[ \rm (i)]
\item If $\cK \neq \emptyset$, then $E=\cK^0$ and we have 
$$ \rL^\infty(\widehat{\cL})=\rL^\infty(\widehat{\cK})=\bigcap_{i \in I} \rL^\infty(\widehat{\cH_i}).$$
$$ \rL^\infty(\Mod(\widehat{\cK}))=\bigcap_{i \in I} \rL^\infty(\Mod(\widehat{\cH_i})).$$
\item If $\cK=\emptyset$, then we have
$$ \rL^\infty( \widehat{\cL} )=  \bigcap_{i \in I} \rL^\infty(\widehat{\cH_i})= \bigcap_{i \in I} \rL^\infty(\Mod(\widehat{\cH_i})).$$
\end{enumerate}
\end{prop}
\begin{proof}
Up to embedding $\cH$ into a larger affine Hilbert space, we may assume that $F=\bigcap_{i \in I} (\cH_i^0)^\perp$ has infinite dimension. In particular, $E$ has infinite codimension. Let $G$ be the closed subgroup of $\Isom(\cH)$ generated by $\bigcup_i G_{\cH_i}$. We claim that $G=G_\cK$ if $\cK\neq \emptyset$ and $G=H_{E^{\perp}}$ if $\cK=\emptyset$. Indeed, it is enough to show this when $I=\{1,2\}$ or when $(\cH_i)_{i \in I}$ is a decreasing net. The first case is given by Proposition \ref{groups and two subspaces} and the second case by Proposition \ref{groups and decreasing subspaces}. Now, every element of $\bigcap_{i \in I} \rL^\infty(\Mod(\widehat{\cH_i}))$ is fixed by $G_i$ for all $i \in I$, hence by $G$. Thus, the conclusion follows from Proposition \ref{fixed point Maharam}.
\end{proof}

\section{Dissipativity}
For the notions of dissipative and recurrent nonsingular actions, we refer the reader to Section \ref{appendix dissipative}.

\begin{prop} \label{diss scale} 
Let $\alpha : G \curvearrowright \cH$ be an affine isometric action. There exists $t_{\rm diss}(\alpha) \in [0,+\infty]$ such that the action $\widehat{\alpha}^t$ is recurrent for all $t < t_{\rm diss}(\alpha)$ and dissipative for all $t > t_{\rm diss}(\alpha)$. 
\end{prop}
\begin{proof}
It is enough to show that if $\widehat{\alpha}$ is not dissipative then $\widehat{\alpha}^t$ is recurrent for all $t \in ]0,1[$. Write $t=\cos \theta $ and $s=\sin \theta$ for some $\theta \in ]0,\frac{\pi}{2}[$ and let $R_\theta : \cH \times \cH^0 \rightarrow \cH^t \times \cH^s$ be the isometry defined in Proposition \ref{Rotation trick}. Then it induces a nonsingular isomorphism $\widehat{R}_\theta : \widehat{\cH} \otimes \widehat{\cH}^0 \rightarrow \widehat{\cH}^t \otimes \widehat{\cH}^s$ which is $G$-equivariant with respect to the diagonal actions $\widehat{\alpha} \otimes \widehat{\alpha}^0$ and $\widehat{\alpha}^t \otimes \widehat{\alpha}^s$. Let $D \subset \widehat{\cH}$ be the dissipative part of $\widehat{\alpha}$. Then $D \otimes \widehat{\cH}^0$ is the dissipative part of $\widehat{\alpha} \otimes \widehat{\alpha}^0$. Let $D_t \subset \widehat{\cH}^t$ be the dissipative part of $\widehat{\alpha}^t$. Then $D_t \otimes \widehat{\cH}^s$ is contained in the dissipative part of $\widehat{\alpha}^t \otimes \widehat{\alpha}^s$. Thus $D_t \otimes \widehat{\cH}^s \subset \widehat{R}_\theta(D \otimes \widehat{\cH}^0)$. For every $(\xi,\eta) \in \cH^0 \times \cH^0$, let $j_{(\xi,\eta)}$ denote the translation action. We have $R_\theta \circ j_{(0,\xi)}=j_{(-s\xi,t\xi)} \circ R_\theta$. Therefore, since $D \otimes \widehat{\cH}^0$ is invariant by $\widehat{j}_{(0,\xi)}$, we get
$$ \widehat{j}_{-s\xi}(D_t) \otimes \widehat{\cH}^s=\widehat{j}_{(-s\xi,t\xi)}(D_t \otimes \widehat{\cH}^s) \subset \widehat{R}_\theta(D \otimes \widehat{\cH}^0).$$
Since this holds for all $\xi \in \cH^0$ and the translation action of $\cH^0$ on $\widehat{\cH}^t$ is ergodic (Proposition \ref{fixed point Maharam}), we conclude that if $D_t$ is nonzero, then necessarily $D=\widehat{\cH}$, i.e.\ if $\widehat{\alpha}^t$ is not recurrent, then $\widehat{\alpha}$ is dissipative.
\end{proof}

\begin{thm} \label{affine dissipativity}
Let $\alpha : G \curvearrowright \cH$ be an affine isometric action. Fix any pair of points $x,y \in \cH$ and choose any Haar measure on $G$. 
\begin{enumerate}[ \rm (i)]
\item \label{dissipativity integral} For all $t > 0$, the action $\widehat{\alpha}^t$ is dissipative (resp.\ recurrent) if and only if the following integral converges (resp.\ diverges) almost everywhere on $\widehat{\cH}^0$
$$ \int_G \exp \left(-\frac{1}{2}t^2\| gx-y\|^2 +t\langle \varphi, gx-y \rangle \right) \rd g, \quad \varphi \in \widehat{\cH}^0.$$

\item \label{estimation dissipativity exponent} We have $$ \sqrt{2 \delta(\alpha)} \leq  t_{\rm diss}(\alpha) \leq 2 \sqrt{2 \delta(\alpha)}.$$
\end{enumerate}
\end{thm}
\begin{proof}

(\ref{dissipativity integral}) By Proposition \ref{formula Radon-Nikodym derivatives}, this integral is precisely 
$$ \frac{\rd ( \mu_x \circ T)}{\rd \mu_y}(\varphi + y)$$
where $T$ is the fibered Haar measure of $\widehat{\alpha}^t$. Thus, we just need to apply Theorem \ref{dissipative trivialization}.

(\ref{estimation dissipativity exponent}) We first prove the lower bound. Suppose that $\widehat{\alpha}^t$ is dissipative. Then for almost every $\varphi \in \widehat{\cH}^0$, we have
$$ \int_G \exp \left( -\frac{1}{2}t^2\|gx-y\|^2 +t\langle \varphi, gx-y \rangle \right) \rd g < + \infty$$
and also 
$$ \int_G \exp \left( -\frac{1}{2}t^2\|gx-y\|^2 +t\langle -\varphi, gx-y \rangle \right) \rd g < + \infty$$
By summing this two integrals and using the fact that $$e^{\langle t\varphi, gx-y \rangle} + e^{-t\langle \varphi, gx-y \rangle} \geq 2,$$ we get 
$$ \int_G \exp \left( -\frac{1}{2}t^2\|gx-y\|^2 \right) \rd g < + \infty$$ which means that $t^2 \geq 2 \delta(\alpha)$.

For the upper bound, suppose that $\widehat{\alpha}^t$ is recurrent. Then by Proposition \ref{L2 criterion recurrent} we have 
$$\int_G  \left( \frac{\rd \mu_{tgx}}{\rd  \mu_{tx}} \right)^{1/2} \rd g = +\infty.$$
By Proposition \ref{formula Radon-Nikodym derivatives} we have 
$$\int_{\widehat{\cH}} \left( \frac{\rd \mu_{tgx}}{\rd  \mu_{tx}} \right)^{1/2} \rd \mu_{tx} = e^{-\frac{1}{8}t^2\|gx-x\|^2}.$$
We conclude that
$$ \int_G e^{-\frac{1}{8}t^2\|gx-x\|^2} \rd g=+\infty,$$
which means that $t^2 \leq 8\delta(\alpha)$.

\end{proof}

\begin{prop} \label{dissipativity evanescent}
Let $\alpha : G \curvearrowright \cH$ be an affine isometric action. Suppose that $\alpha$ is evanescent. Then for every ergodic nonsingular action $\sigma : G \curvearrowright X$, the diagonal action $\widehat{\alpha} \otimes \sigma$ is either dissipative or recurrent.
\end{prop}
\begin{proof}
Let $\cK$ be a nonempty $\alpha$-invariant subspace of $\cH$ and wrtie $\cH=E \times \cK$ where $E=\cK^\perp$. Let $\pi=\alpha^0|_E$ and $\beta = \alpha|_\cK$. Let $A$ be the recurrent part of $\widehat{\alpha} \otimes \sigma$. Since $\widehat{\alpha} \otimes \sigma=\widehat{\pi} \otimes \widehat{\beta} \otimes \sigma$ acting diagonally on $\widehat{E} \otimes \widehat{\cK} \otimes X$ and since $\widehat{\pi}$ is measure preserving, we know that $A$ does not depend on the first coordinate, i.e.\  $1_A \in \rL^\infty(\widehat{\cK} \otimes X)$. Since this holds for every nonempty invariant subspace $\cK$ and $\alpha$ is evanescent, we conclude by Proposition \ref{intersection maharam extension} that $1_A \in \rL^\infty(X)$. But $\sigma$ is ergodic, hence $1_A=1$ or $0$.
\end{proof}

Let $\alpha : \Gamma \curvearrowright \cH$ be a proper faithful affine isometric action of a discrete group $\Gamma$. Take $x \in \cH$. Then $\Gamma_x=\{ g \in \Gamma \mid gx=x\}$ is a finite subgroup of $\Gamma$. Up to perturbing $x$, we may assume that $\Gamma_x=\{1\}$. Then a natural way to obtain a fundamental domain of the action $\alpha$ is to use the \emph{Dirichlet domain}
$$ \mathcal{D}:= \{ y \in \cH \mid \forall g \in \Gamma, \; \| y -x \| < \| y - gx \| \}.$$
Indeed, the sets $(g \mathcal{D})_{g \in \Gamma}$ are pairwise disjoint and thanks to the properness of the action, we have $$\cH=\bigcup_{g \in \Gamma} g \overline{\mathcal{D}}.$$
We use this idea to construct a fundamental domain of the dissipative part of the Gaussian action $\widehat{\alpha}$. However, the Dirichlet Domain may be a null-set at the Gaussian level. To give a precise statement, we introduce the following notation 
$$ \| \omega -x \| \leq \| \omega -y \| \; \Leftrightarrow \; \langle \omega -\frac{x+y}{2}, y-x \rangle \leq 0 $$
for all $x,y \in \cH$ and $\omega \in \widehat{\cH}$. Note that the quantity $\| \omega -x \| $ does not make sense on its own but the notation $\| \omega -x \| \leq \| \omega -y \|$ is very intuitive in view of Proposition \ref{quadratic buseman}. It is equivalent to $\frac{\rd \mu_x}{\rd \mu_y}(\omega) \geq 1$. Note also that when $x \neq y$, we have 
$$\{ \omega \in \widehat{\cH} \mid  \| \omega -x \| \leq \| \omega -y \| \text{ and }  \| \omega -x \| \geq \| \omega -y \| \} = 0$$
because the distribution of $\frac{\rd \mu_x}{\rd \mu_y}$ is non-atomic.
\begin{thm} \label{dirichlet domain}
Let $\alpha : \Gamma \curvearrowright \cH$ be a proper faithful affine isometric action of a discrete group $\Gamma$. Take $x \in \cH$ such that $\{ g \in \Gamma \mid gx=x\}=\{1\}$. Then a fundamental domain of the dissipative part of $\widehat{\alpha}$ is given by the Gaussian Dirichlet domain
$$  \widehat{\mathcal{D}}:= \{ \omega \in \widehat{\cH} \mid \forall g \in \Gamma, \; \| \omega -x \| \leq \| \omega - gx \| \}.$$
The recurrent part is given by
$$  \{ \omega \in \widehat{\cH} \mid \;  \| \omega - gx \| \leq \| \omega -x \|  \text{ for infinitely many } g \in \Gamma \}.$$
\end{thm}
\begin{proof}
Observe that for every $g \neq h \in \Gamma $, we have
$$ \widehat{g}  \widehat{\mathcal{D}} \cap \widehat{h}  \widehat{\mathcal{D}} \subset \{ \omega \in \widehat{\cH} \mid \| \omega -gx \|=\| \omega-hx\| \} =0$$
because $gx \neq hx$. This means that the sets $(\widehat{g}  \widehat{\mathcal{D}})_{g \in \Gamma}$ are pairwise disjoint. Now define
$$ A=\bigcup_{g \in \Gamma} \widehat{g} \widehat{\mathcal{D}}.$$
Then $A$ is contained in the dissipative part of $\widehat{\alpha}$ and it only remains to show that $A$ contains the dissipative part of $\widehat{\alpha}$. For every finite set $F \subset \Gamma$ let
$$  \widehat{\mathcal{D}}_F=\{ \omega \in \widehat{\cH} \mid \forall g \in \Gamma \setminus F, \; \| \omega -x \| \leq \| \omega - gx \| \}.$$
Note that for all $g \in F$, we have
$$ \widehat{g}  \widehat{\mathcal{D}}= \widehat{\mathcal{D}}_F \cap \{ \omega \in \widehat{\cH} \mid \forall h \in F \setminus \{g\}, \; \| \omega -gx \| \leq \| \omega - hx \| \}.$$
But since $F$ is finite, one can check algebraically that we have a partition
$$ \widehat{\cH} = \bigsqcup_{g \in F} \{ \omega \in \widehat{\cH}\mid \forall h \in F \setminus \{g\}, \; \| \omega -gx \| \leq \| \omega - hx \| \}.$$
From this two facts, we conclude that
$$  \widehat{\mathcal{D}}_F=\bigsqcup_{g \in F} \widehat{g}  \widehat{\mathcal{D}}.$$
This implies that
$$ A=\bigcup_{F \text{ finite subset of } \Gamma} \widehat{\mathcal{D}}_F.$$
Now, observe that for almost every $\omega \in \widehat{\cH} \setminus A$, we have
$$ \|\omega -x \| \geq \| \omega - gx \| \text{ for infinitely many } g \in \Gamma$$
or equivalently (Proposition \ref{quadratic buseman})
$$ \frac{\rd \mu_{gx}}{\rd \mu_x}(\omega) \geq 1 \text{ for infinitely many } g \in \Gamma$$
which implies that
$$ \sum_{ g \in \Gamma} \frac{\rd \mu_{gx}}{\rd \mu_x}(\omega) =+\infty.$$
We conclude that $\widehat{\cH} \setminus A$ is contained in the recurrent part of $\widehat{\alpha}$ as we wanted.
\end{proof}

\begin{prop} \label{translation action}
Let $(\lambda_n)_{n \in \N}$ be a sequence of positive real numbers and let $\Gamma$ be the subgroup of $\ell^2(\N)$ generated by $\lambda_n e_n, \; n \in \N$ where $(e_n)_{n \in \N}$ is the canonical orthonormal basis of $\ell^2(\N)$. Consider the translation action $\alpha$ of $\Gamma$ on $\cH=\ell^2(\N)$. Then for all $t > 0$, the action $\widehat{\alpha}^t$ is dissipative if and only if
$$ \sum_{n \in \N} \frac{1}{\lambda_n}e^{-\frac{1}{8}t^2\lambda_n^2} < +\infty$$
and recurrent if and only if
$$ \sum_{n \in \N} \frac{1}{\lambda_n}e^{-\frac{1}{8}t^2\lambda_n^2} = +\infty.$$
In particular, we have $t_{\rm diss}(\alpha)=2 \sqrt{2 \delta(\alpha)}$ and for $t=t_{\rm diss}(\alpha)$, the action $\widehat{\alpha}^t$ can be dissipative or recurrent depending on the sequence $(\lambda_n)_{n \in \N}$.
\end{prop}
\begin{proof}
It is easy to check that the Gaussian Dirichlet domain of $\widehat{\alpha}^t$ centered at the origin is given by
$$ \left \{ \varphi \in \widehat{\cH} \mid \forall n \in \N, \; \langle \varphi ,  e_n \rangle \leq \frac{1}{2}t \lambda_n \right \}$$
while the recurrent part is given by
$$ C=\left \{ \varphi \in \widehat{\cH} \mid \langle \varphi ,  e_n \rangle \geq \frac{1}{2}t \lambda_n  \text{ for infinitely many } n \in \N \right \}.$$
Since $(\langle \cdot, e_n \rangle)_{n \in \N}$ is a family of independent random variables, the Borel-Cantelli lemma and its converse shows that the event $C$ has probability $0$ or $1$ according to the convergence/divergence of the series
$$ \sum_{n \in \N} \mathbb{P} \left( \langle \varphi ,  e_n \rangle \geq \frac{1}{2}t \lambda_n \right)=\sum_{n \in \N} Q\left(\frac{1}{2}t\lambda_n \right)$$
where
$$ Q(x) = \int_x^{+\infty} \frac{1}{\sqrt{2\pi}} e^{-\frac{s^2}{2}} \rd s, \; \text{ for } x \geq 0.$$
There is no exact formula for $Q(x)$ but when $x \to \infty$, it is known that $Q(x) \sim \frac{1}{x}e^{-\frac{x^2}{2}}$.
\end{proof}

\begin{prop} \label{dissipative nonconjugate}
Let $\alpha : G \curvearrowright \cH$ be an affine isometric action such that $ t_{\rm diss}(\alpha) < +\infty$. Then the actions $\widehat{\alpha}^t$ are pairwise non-conjugate for $t \leq t_{\rm diss}(\alpha)$.
\end{prop}
\begin{proof}
Take $t < s \leq t_{\rm diss}(\alpha)$. Take $r > 0$ such that $t^2+r^2 < t_{\rm diss}(\alpha)$ but $s^2+r^2 > t_{\rm diss}(\alpha)^2$. Then, by Remark \ref{conjugacy different gaussian}, we have that $\widehat{\alpha}^t \otimes \widehat{\alpha}^r$ is recurrent while $\widehat{\alpha}^s \otimes \widehat{\alpha}^r$ is dissipative. In particular, $\widehat{\alpha}^t$ and $\widehat{\alpha}^s$ are not conjugate.
\end{proof}

\section{Koopman representation, invariant means and amenability}
In this section, we investigate the ``spectral" properties of Gaussian actions. We refer to \ref{appendix koopman} for the various notions used in this section.

\begin{prop}
Let $\alpha : G \curvearrowright \cH$ be an affine isometric action. Then either $\widehat{\alpha}^t$ has an invariant mean for all $t > 0$ or $\widehat{\alpha}^t$ does not have an invariant mean for all $t > 0$. 
\end{prop}
\begin{proof}
By Remark \ref{conjugacy different gaussian}, for all $r,t,s > 0$ such that $r^2=t^2+s^2$, we have that the action $\widehat{\alpha}^r \otimes \widehat{\alpha}^0$ is conjugate to $\widehat{\alpha}^t \otimes \widehat{\alpha}^s$. By Proposition \ref{product action invariant mean}, this means that $\widehat{\alpha}^r$ has an invariant mean if and only if both $\widehat{\alpha}^t$ and $\widehat{\alpha}^s$ have an invariant mean. The conclusion follows easily.
\end{proof}

\begin{prop}\label{amen scale}
Let $\alpha : G \curvearrowright \cH$ be an affine isometric action. There exists $t_{\rm amen}(\alpha) \in [0,+\infty]$ such that $\widehat{\alpha}^t$ is amenable for all $t > t_{\rm amen}(\alpha)$ and nonamenable for all $t < t_{\rm amen}(\alpha)$. 

Moreover, we always have $t_{\rm amen}(\alpha) \leq t_{\rm diss}(\alpha)$.
\end{prop}
\begin{proof}
 It is enough to show that if $\widehat{\alpha}$ is nonamenable, then $\widehat{\alpha}^t$ is nonamenable for all $t \in ]0,1[$. This follows easily from Remark \ref{conjugacy different gaussian} and Proposition \ref{amenable factor map}. The second part follows from Corollary \ref{dissipative amenable}.
\end{proof}

\begin{thm} \label{dictionary}
Let $\alpha : G \curvearrowright \cH$ be an affine isometric action and let $\widehat{\alpha} : G \curvearrowright \widehat{\cH}$ be the associated Gaussian action. Then the following holds:
\begin{enumerate}[ \rm (i)]
\item \label{item invariant probability} $\widehat{\alpha}$ admits an invariant probability measure (not necessarily faithful) if and only if $\alpha$ has a fixed point.
\item \label{item zero-type} $\widehat{\alpha}$ is of zero-type if and only if $\alpha$ is proper.
\item \label{item invariant mean} If $\alpha$ has almost fixed points, then $\widehat{\alpha}$ has almost vanishing entropy and in particular, $\widehat{\alpha}$ has an invariant mean. 
\item \label{item no invariant mean} If $\alpha^0$ has stable spectral gap\footnote{Recall that an orthogonal representation $\pi : G \rightarrow \mathcal{O}(H)$ has \emph{stable spectral gap} if the representation $\pi \otimes \rho$ has no almost invariant vectors for every representation $\rho : G \rightarrow \mathcal{O}(H)$.} and $\alpha$ has no fixed point, then $\widehat{\alpha}$ has no invariant mean.
\end{enumerate}
\end{thm}
\begin{proof}
(\ref{item invariant probability}) If $\alpha$ has a fixed point $x \in \cH$, then $\mu_{x}$ is an invariant probability measure for $\widehat{\alpha}$. Conversely suppose that $\alpha$ has no fixed point. Then there exists a sequence $g_n \in G$ such that $\lim_n \|g_n x-x\|=+\infty$ for any $x \in \cH$. Proposition \ref{formula Radon-Nikodym derivatives} then shows that $\lim_n \langle \pi(g_n) \mu_x^{1/2},\mu_x^{1/2} \rangle =0$. We conclude by Proposition \ref{criterion invariant proba}.

(\ref{item zero-type}) This follows in a similar way from Proposition \ref{formula Radon-Nikodym derivatives} and Proposition \ref{mixing Koopman}.

(\ref{item invariant mean}) This follows from Proposition \ref{formula Radon-Nikodym derivatives} which gives the following formula for the relative entropy
$$ H(\mu_y \mid \mu_x)=\frac{1}{2}\| y-x\|^2 \quad \text{ for all } x,y \in \cH.$$
Then $\widehat{\alpha}$ has an invariant mean by Proposition \ref{zero entropy implies invariant mean}.

(\ref{item no invariant mean}) Suppose that $\nu_n$ is a sequence of almost invariant probability measures for $\widehat{\alpha}$. By using Remark \ref{conjugacy different gaussian}, identify $\widehat{\cH}^{\sqrt{2}} \otimes \widehat{\cH}^0$ with $\widehat{\cH} \otimes \widehat{\cH}$ and $\widehat{\alpha}^{\sqrt{2}} \otimes \widehat{\alpha}^0$ with $\widehat{\alpha} \otimes \widehat{\alpha}$. Let $\mu$ be the canonical Gaussian probability measure of $\widehat{\cH}^0$ and let $P$ be the orthogonal projection of $\rL^2( \widehat{\cH}^{\sqrt{2}} \otimes \widehat{\cH}^0)$ onto $\rL^2( \widehat{\cH}^{\sqrt{2}}) \otimes \C \mu^{1/2}$. Note that the net of vectors $(\nu_n^{1/2} \otimes \nu_m^{1/2})_{n,m \in \N}$ in $\rL^2(\widehat{\cH} \otimes \widehat{\cH})$ is almost invariant for the Koopman representation, when $n,m \to \infty$. Therefore, since $\alpha^0$ has stable spectral gap, we must have 
$$\lim_{n,m \to \infty} \| \nu_n^{1/2} \otimes \nu_m^{1/2} - P(\nu_n^{1/2} \otimes \nu_m^{1/2})\|=0$$
Since the range of $P$ is invariant under the flip map of $\widehat{\cH} \otimes \widehat{\cH}$, this implies that 
$$\lim_{n,m \to \infty} \| \nu_n^{1/2} \otimes \nu_m^{1/2} - \nu_m^{1/2} \otimes \nu_n^{1/2}\|=0$$
which is equivalent to
$$\lim_{n,m \to \infty} \| \nu_n^{1/2} - \nu_m^{1/2} \|=0.$$
We conclude that $(\nu_n^{1/2})_{n \in \N}$ is a Cauchy sequence which must therefore converge to an invariant vector contradicting the assumption that $\alpha$ has no fixed point.
\end{proof}

The following fact follows easily from the properties of Gaussian actions but is not obvious from the definition of the Poincar\'e exponent. Note that the nonamenability assumption is necessary since the group $\Z$ admits affine isometric actions with almost fixed points and a vanishing Poincar\'e exponent.
\begin{cor} \label{poincare infinite}
Let $\alpha : G \curvearrowright \cH$ be an affine isometric action. Suppose that $G$ is nonamenable and $\alpha$ has almost fixed points. Then $\delta(\alpha)=+\infty$
\end{cor}
\begin{proof}
Since $\alpha$ has almost fixed points, Theorem \ref{dictionary} shows that $\widehat{\alpha}^t$ has an invariant mean for all $t$. Since $G$ is nonamenable, $\widehat{\alpha}^t$ is nonamenable hence recurrent for all $t$. By Theorem \ref{affine dissipativity}, we conclude that $\delta(\alpha)=+\infty$.
\end{proof}

We deduce a new characterization of property (T) in terms of nonsingular actions.
\begin{cor} \label{nonsingular property (T)}
Let $G$ be a locally compact group. The following are equivalent:
\begin{enumerate}[ \rm (i)]
\item  $G$ does not have property (T).
\item $G$ admits a nonsingular action which has an invariant mean but no invariant probability measure.
\item $G$ admits a nonsingular action which has almost vanishing entropy but no invariant probability measure.
\end{enumerate}
\end{cor}

We also obtain a similar characterization of the Haagerup property.
\begin{cor} \label{nonsingular Haagerup property}
Let $G$ be a locally compact group. The following are equivalent:
\begin{enumerate}[ \rm (i)]
\item  $G$ has the Haagerup property.
\item $G$ admits a nonsingular action of zero-type which has an invariant mean.
\item $G$ admits a nonsingular action of zero-type with almost vanishing entropy.
\end{enumerate}
\end{cor}

We now give a formula for the spectral radius of Gaussian actions expressed in terms of the random walk on the group.

\begin{thm} \label{formula spectral radius}
Let $\alpha : G \curvearrowright \cH$ be an affine isometric action. Let $b(g)=gx-x$ for some $x \in \cH$. Take any symmetric probability measure $\mu$ on $G$ and define the function $f:[0,+\infty[ \rightarrow [0,1]$ by
$$ f(s)=\lim_{n \to \infty} \mathbb{E}\left( e^{-s\|b(g_n)\|^2} \right)^{1/n}$$
where $(g_n)_{n \in \N}$ is the $\mu$-random walk on $G$.
\begin{enumerate}[ \rm (i)]
\item For all $t \geq 0$, we have
$$ \rho_\mu(\widehat{\alpha}^t) =f\left(\frac{t^2}{8}\right).$$
\item For all $s \geq 0$, we have
$$  f(s) \geq \rho_{\mu}(G) \quad  \text{ and } \quad  f(s) \geq e^{-s \| \alpha \|_\mu^2} $$
\item The function $f$ is continuous, decreasing and $\log(f)$ is convex.
\end{enumerate}
\end{thm}
\begin{proof}
$(\rm i)$  Let $P=\pi(\mu)=\int_G \pi(g) \rd \mu(g)$ where $\pi : G \curvearrowright \rL^2(\widehat{\cH}^t)$ is the Koopmann representation of $\widehat{\alpha}^t$. Then by definiton $\rho_\mu(\widehat{\alpha}^t)=\|P\|$. Take $x \in \cH$ and let $\xi=\mu_{tx}^{1/2} \in \rL^2(\widehat{\cH}^t)$. By Proposition \ref{single vector spectral radius}, we have
\begin{align*}
\|P\| &=\lim_{n} \| P^n \xi \|^{1/n}\\
 & =\lim_{n} \langle P^n \xi,\xi \rangle^{1/n} \\
 & =\lim_{n}\left( \int_G \langle \mu_{tgx}^{1/2} ,\mu_{tx}^{1/2}  \rangle \rd \mu^{*n}(g) \right)^{1/n} \\
 &= \lim_{n}\left( \int_G \exp \left( -\frac{t^2}{8}\|gx-x\|^2 \right)  \rd \mu^{*n}(g)  \right)^{1/n} \\
 &=f\left(\frac{t^2}{8}\right)
\end{align*}

$(\rm ii)$ The first inequality follows from item $(\rm i)$ and Proposition \ref{spectral radius action group}. The second inequality can be deduced from \cite[Lemma 2.2]{EO18} and the concavity of $\lambda \mapsto \lambda^{1/n}$, but one can also make the following easy computation. Let $\pi$ be the Koopman representation of $\widehat{\alpha}^t$. By Proposition \ref{formula Radon-Nikodym derivatives}, we have 
$$ \rho_\mu(\widehat{\alpha}^t) \geq  \langle \pi(\mu) \mu_{tx}^{1/2},\mu_{tx}^{1/2} \rangle = \int_G \exp \left( - \frac{t^2}{8} \| gx-x\|^2 \right) \rd \mu(g).$$
Thus, by convexity of $\exp$, we get
$$ \rho_\mu(\widehat{\alpha}^t) \geq \exp\left(- \int_G \frac{t^2}{8} \| gx-x\|^2 \rd \mu(g) \right).$$
Since this holds for all $x \in \cH$, we get 
$$ \rho_\mu(\widehat{\alpha}^t) \geq \exp\left( -\frac{t^2}{8} \| \alpha \|_\mu^2 \right).$$
hence our desired inequality by item $(\rm i)$.

$(\rm iii)$ The function $f$ is obviously decreasing. For $s \geq 1$, the convexity of $\lambda \mapsto \lambda^s$ shows that $f(st) \geq f(t)^s$, hence by letting $s \to 1$, we see that $f$ is continuous. Next, Cauchy-Schwartz inequality gives $f(\frac{t+s}{2}) \leq \sqrt{f(t)f(s)}$ for all $s,t > 0$ which implies that $\log(f)$ is convex since it is continuous. 
\end{proof}


Here is a consequence of Theorem \ref{formula spectral radius} and Proposition \ref{dictionary}.
\begin{cor}
Let $\pi : G \rightarrow \mathcal{O}(H)$ be an orthogonal representation which has stable spectral gap. Let $b \in Z^1(\pi)$ be a $1$-cocycle which is not a coboundary. Then for any symmetric probability measure $\mu$ on $G$ which spans $G$, we have 
$$ \uplim_{n \to \infty} \mathbb{E}\left(e^{-\|b(g_n)\|^2} \right)^{1/n} < 1$$
where $(g_n)_{n \in \N}$ is the $\mu$-random walk on $G$.
\end{cor}

We also derive the following inequality relating the Poincar\'e exponent to the spectral radius. Under the normalization condition $\| \alpha \|_\mu^2=1$, this gives a \emph{universal} lower bound for the Poincar\'e exponent. Note that both Proposition \ref{poincare strictly positive} and Corollary \ref{poincare infinite} follow from this inequality.

\begin{prop} \label{lower bound amenability}
Let $\alpha : G \curvearrowright \cH$ be an affine isometric action. Let $\mu$ be a symmetric probability measure on $G$. If $t_{\rm amen}(\alpha) <+ \infty$ and $\| \alpha\|_\mu^2 < +\infty$ then
$$-8 \log\rho_\mu(G) \leq t_{\rm amen}(\alpha)^2  \| \alpha\|^2_\mu.$$
In particular, if $\delta(\alpha) < +\infty$, we have the following inequality:
$$ - \log\rho_\mu(G) \leq \delta(\alpha)  \| \alpha\|^2_\mu.$$
\end{prop}
\begin{proof}
Take $t > t_{\rm amen}(\alpha)$. Since $\widehat{\alpha}^t$ is amenable, we have $\rho_\mu(\widehat{\alpha}^t)=\rho_\mu(G)$ by Proposition \ref{weakly contained amenable action}. On the other hand, we know by Theorem \ref{formula spectral radius} that 
$$- \log \rho_\mu(\widehat{\alpha}^t) \leq \frac{t^2}{8} \| \alpha\|_\mu^2.$$
For the second part, we use the fact that $t_{\rm amen}(\alpha) \leq t_{\rm diss}(\alpha) \leq 2\sqrt{2\delta(\alpha)}$.
\end{proof}


Here is an example where the function $f$ of Theorem \ref{formula spectral radius} can be computed explicitely.

\begin{thm} \label{computation free group random walk}
Let $(g_n)_{n \in \N}$ be the canonical symmetric random walk on the free group $\F_d$ on $d \geq 1$ generators. Let $|g|$ denote the word length of $g \in \F_d$. Then for all $s \geq 0$, we have

\renewcommand{\arraystretch}{2}
$$ \uplim_{n \to \infty} \mathbb{E}\left( e^{-s|g_n|} \right)^{1/n} =\left\lbrace
\begin{array}{ll}
\frac{1}{2d}\left( (2d-1)e^{-s}+e^{s} \right)  & \mbox{if $s < \frac{1}{2} \log(2d-1)$}\\
\frac{\sqrt{2d-1}}{d} & \mbox{if $s \geq \frac{1}{2}\log(2d-1)$.}\\
\end{array}
\right.$$
\renewcommand{\arraystretch}{1}
\end{thm}
\begin{proof}
For all $n,m \in \N$, let $p_{n,m}=\mathbb{P}(|g_n|=m)$. Extend the notation to $n, m \in \Z$ by letting $p_{n,m}=0$ if $n < 0$ or $m< 0$. Define a formal series
$$ A(x,y)=\sum_{n,m} p_{n,m} x^ny^m.$$
For a given $s \geq 0$, we have
$$\uplim_{n \to \infty} \mathbb{E}\left( e^{-s|g_n|} \right)^{1/n}=\uplim_{n \to \infty} \left( \sum_{m} p_{n,m} e^{-sm} \right)^{1/n}$$ and by the Cauchy-Hadamard theorem this is precisely the inverse of the radius of convergence in $x$ of the formal series 
$$ A(x,e^{-s})=\sum_n \left( \sum_{m} p_{n,m}e^{-sm} \right) x^n.$$ We will determine this radius of convergence by computing explicitely $A(x,y)$. 

First, observe that we have the following relation
$$ \forall n \in \Z, \quad p_{n,m}=\frac{2d-1}{2d} p_{n-1,m-1}+ \frac{1}{2d}p_{n-1,m+1}$$
which holds for all $m \geq 2$. It also obviously holds for all $m \leq -2$. To make it work for all $n, m \in \Z$, one can modify this equation as follows:
$$ \forall n,m \in \Z, \quad p_{n,m}=\frac{2d-1}{2d} p_{n-1,m-1}+ \frac{1}{2d}p_{n-1,m+1} + \delta_{m,0}\delta_{n,0} + (\delta_{m,1}-\delta_{m,-1})\frac{1}{2d}p_{n-1,0}.$$
This can be checked easily in each case $m=-1,0,1$. Now, we multiply this equation by $x^ny^m$ and we sum over all $n,m \in \Z$. We obtain
$$ A(x,y)=\frac{2d-1}{2d}xy A(x,y) + \frac{1}{2d}xy^{-1}A(x,y)+1+\frac{1}{2d}(y-y^{-1})xB(x)$$
where $B(x)=A(x,0)$. By multiplying by $2dy$ and rearranging the terms, we get
$$(2dy-(2d-1)xy^2-x)A(x,y)=2dy+x(y^2-1)B(x).$$
Finally, we obtain
$$A(x,y)=\frac{2dy+x(y^2-1)B(x)}{2dy-((2d-1)y^2+1)x}.$$
Now, the formal series $B(x)$ has been computed by Kesten \cite[Proof of Theorem 3]{Kes59}. He found
$$B(x)=\frac{\sqrt{d^2-(2d-1)x^2}-(d-1)}{1-x^2}$$
and its radius of convergence is $\frac{d}{\sqrt{2d-1}}.$
Thus if we fix $y=e^{-s} \in ]0,1[$, the radius of convergence of 
$$A(x,y)=\frac{2dy+x(y^2-1)B(x)}{2dy-((2d-1)y^2+1)x}.$$
with respect to the variable $x$ is given by 
$$ \min\left( \frac{2dy}{(2d-1)y^2+1},\frac{d}{\sqrt{2d-1}} \right)$$
which is what we wanted.
\end{proof}

\section{Ergodicity and Krieger type}
In this section we investigate the ergodicity and type of the Gaussian actions. We first observe the following fact.

\begin{prop}\label{erg scale} 
Let $\alpha : G \curvearrowright \cH$ be an affine isometric action. There exists $t_{\rm erg}(\alpha) \in [0,+\infty]$ such that $\widehat{\alpha}^t$ is ergodic for all $0 < t < t_{\rm erg}(\alpha) $ and not ergodic for all $t > t_{\rm erg}(\alpha)$.
\end{prop}
\begin{proof}
It is enough to show that if $\widehat{\alpha}$ is ergodic, then $\widehat{\alpha}^t$ is ergodic for all $t \in ]0,1[$. Write $t=\cos \theta $ and $s=\sin \theta$ for some $\theta \in ]0,\frac{\pi}{2}[$ and let $R_\theta : \cH \times \cH^0 \rightarrow \cH^t \times \cH^s$ be the isometry defined in Proposition \ref{Rotation trick}. Then it induces a $G$-equivariant nonsingular isomorphism $\widehat{R}_\theta : \widehat{\cH} \otimes \widehat{\cH}^0 \rightarrow \widehat{\cH}^t \otimes \widehat{\cH}^s$, hence a $G$-equivariant isomorphism $\psi  : \rL^\infty(\widehat{\cH}^t \otimes \widehat{\cH}^s) \rightarrow \rL^\infty(\widehat{\cH} \otimes \widehat{\cH}^0)$ which satisfies
$$\mu_{(x,\xi)}(\psi(f))=\mu_{R_\theta(x,\xi)}(f)$$
for all $f \in  \rL^\infty(\widehat{\cH}^t \otimes \widehat{\cH}^s)$ and all $(x,\xi) \in \cH \times \cH^0$.
Now, suppose that $f \in \rL^\infty( \widehat{\cH}^t)$ is an $ \widehat{\alpha}^t$-invariant function. Then $\psi(f \otimes 1)$ is $\widehat{\alpha} \otimes \widehat{\alpha}^0$-invariant. Moreover, for all $(x,\xi) \in \cH \times \cH^0$, we have 
$$\mu_{(x,\xi)}(\psi(f \otimes 1))=\mu_{R_\theta(x,\xi)}(f \otimes 1)=\mu_{tx+s\xi}(f).$$
Let $E: \rL^\infty(\widehat{\cH} \otimes \widehat{\cH}^0) \rightarrow \rL^\infty(\widehat{\cH})$ be the canonical Gaussian conditional expectation characterized by $\mu_{x} \circ E = \mu_{(x,0)}$ for all $x \in \cH$. Since $E$ is $G$-equivariant and $\widehat{\alpha}$ is ergodic, then $E(\psi(f \otimes 1))=\lambda 1$ for some $\lambda \in \R$. This implies that 
$$\mu_{tx}(f)=\mu_{(x,0)}(\psi(f \otimes 1))=\mu_{(x,0)}(E(\psi(f \otimes 1)))=\lambda$$ for all $x \in \cH$. Thus $f=\lambda 1$ thanks to Proposition \ref{density}. This shows that $\widehat{\alpha}^t$ is ergodic.
\end{proof}

\begin{rem}
When $G$ is a countable group, a dissipative action on a diffuse probability space can never be ergodic. Thus, in that case, we have $t_{\rm erg}(\alpha) \leq t_{\rm diss}(\alpha)$ for every isometric action $\alpha$ of $G$.
\end{rem}

\subsection{Ergodicity and type in the mixing case}
We first study the ergodicity and type of Gaussian actions under the assumption that the linear part of the affine isometric action is mixing. The main tool we use is the following result of Schmidt and Walters.

\begin{thm} \cite{SW81} \label{schmidt}
Let $G$ be a locally compact group. Let $\sigma : G \curvearrowright X$ be a nonsingular action and $\rho : G \curvearrowright (Y,\mu)$ a probability measure preserving action. Suppose that $\sigma$ is recurrent and $\rho$ is mixing. Then every $\sigma \otimes \rho$-invariant function in $\rL^\infty(X \otimes Y)$ is contained in $\rL^\infty(X)$.
\end{thm}
\begin{rem}
Note that the recurrence assumption is weaker then the proper ergodicity assumption in \cite{SW81} but it is enough and the proof is the same.
\end{rem}

The following key lemma is a direct consequence of Theorem \ref{schmidt}.

\begin{lem} \label{affine recurrent and mixing}
Let $\alpha : G \curvearrowright \cH$ be an affine isometric action such that $\widehat{\alpha}^0$ is mixing. Suppose that $\cK \subset \cH$ is an $\alpha$-invariant affine subspace. Then for any nonsingular action $\sigma : G \curvearrowright X$ such that $\widehat{\alpha} \otimes \sigma$ is recurrent, we have
$$ \rL^\infty( \Mod( \widehat{\cH} \otimes X))^{\Mod(\widehat{\alpha} \otimes \sigma)} \subset \rL^\infty(\Mod(\widehat{\cK} \otimes X))$$
and in particular
$$ \rL^\infty( \widehat{\cH} \otimes X)^{\widehat{\alpha} \otimes \sigma} \subset \rL^\infty(\widehat{\cK} \otimes X).$$
\end{lem}
\begin{proof}
Let $\beta$ be the restriction of $\alpha$ to $\cK$ and let $\pi$ be the restriction of $\alpha^0$ to $(\cK^0)^\perp \subset \cH^0$. Then $\widehat{\alpha} \otimes \sigma=\widehat{\pi} \otimes \widehat{\beta} \otimes \sigma$. Since $\widehat{\pi}$ preserves a probability measure, $\widehat{\alpha}\otimes \sigma$ is recurrent if and only if $\widehat{\beta} \otimes \sigma$ is recurrent. Therefore, since $\widehat{\pi}$ is mixing, we can apply Theorem \ref{schmidt}. We can also apply it to $\Mod(\widehat{\alpha} \otimes \sigma)=\widehat{\pi} \otimes \Mod( \widehat{\beta} \otimes \sigma)$.
\end{proof}

\begin{thm} \label{evanescent ergodicity}
Let $\alpha : G \curvearrowright \cH$ be an affine isometric action. Suppose that $\alpha$ is evanescent, $\alpha^0$ is mixing and $\alpha$ has no fixed point. Then for any ergodic nonsingular action $\sigma : G \curvearrowright X$, the diagonal action $\widehat{\alpha} \otimes \sigma$ is either dissipative or ergodic of type $\III_1$.
\end{thm}
\begin{proof}
By Proposition \ref{dissipativity evanescent}, we know that $\widehat{\alpha} \otimes \sigma$ is either dissipative or recurrent. Suppose that it is recurrent. Take $\cK$ a nonempty $\alpha$-invariant subspace of $\cH$. Lemma \ref{affine recurrent and mixing} shows that
$$ \rL^\infty( \Mod( \widehat{\cH} \otimes X))^{\Mod(\widehat{\alpha} \otimes \sigma)} \subset \rL^\infty( \Mod( \widehat{\cK} \otimes X)).$$
Recall (see the appendix) that $\rL^\infty( \Mod( \widehat{\cK} \otimes X))$ is the fixed point subalgebra of $\rL^\infty( \Mod( \widehat{\cK}) \otimes \Mod(X))$ under the action of the diagonal flow $\lambda \mapsto \theta_\lambda \otimes \theta_{\lambda}^{-1}$. Therefore, since $\cK$ is any invariant subspace and $\alpha$ is evanescent, Proposition \ref{intersection maharam extension} allows us to conclude that 
$$ \rL^\infty( \Mod( \widehat{\cH}) \otimes X)^{\Mod(\widehat{\alpha} \otimes \sigma)} \subset \rL^\infty(\Mod(X))^\theta=\rL^\infty(X).$$
By assumption, $\sigma$ is ergodic, hence we have proved that $\Mod(\widehat{\alpha} \otimes \sigma)$ is ergodic, i.e.\ $\widehat{\alpha} \otimes \sigma$ is ergodic of type $\III_1$.
\end{proof}

\begin{thm} \label{strong Haagerup}
Let $G$ be a locally compact but non-compact group with the Haagerup property. Then there exists an affine isometric action $\alpha : G \curvearrowright \cH$ such that for all $t > 0$, the nonsingular Gaussian action $\widehat{\alpha}^t$ satisfies :
\begin{enumerate}[ \rm (i)]
\item $\widehat{\alpha}^t$ is free.
\item $\widehat{\alpha}^t$ is weakly mixing of stable type $\III_1$.
\item for every ergodic nonsingular action $\rho : G \curvearrowright Y$, the diagonal action $\widehat{\alpha}^t \otimes \rho : G \curvearrowright \widehat{\cH} \otimes Y$ is either dissipative or ergodic of type $\III_1$.
\item $\widehat{\alpha}^t$ is of zero-type.
\item $\widehat{\alpha}^t$ has almost vanishing entropy. In particular, $\widehat{\alpha}^t$ has an invariant mean and it is nonamenable if $G$ is nonamenbale.
\end{enumerate}
\end{thm}
\begin{proof}
Since $G$ is not compact, we can find a continuous proper function $f:G \rightarrow [1,+\infty)$ such that
$$ \forall s > 0, \int_G e^{-sf(g)} \rd g =+\infty.$$
Since $G$ has the Haagerup property, it has a mixing orthogonal representation which has almost invariant vectors. By Proposition \ref{existence evanescent}, we can thus find an affine isometric action $\alpha : G \curvearrowright \cH$ such that $\alpha$ is evanescent and has no fixed point and $\alpha^0$ is mixing. Moreover, we can choose $\alpha$ such that $\|gx-x\|^2 \leq f(g)$ for all $g \in G$ and some $x \in \cH$. Then we have $\delta(\alpha)=+\infty$. Therefore $\widehat{\alpha}^t$ is recurrent for all $t > 0$ by Theorem \ref{affine dissipativity}. By Proposition \ref{evanescent properties}, $\alpha$ is proper, hence $\widehat{\alpha}^t$ is of zero-type. By Proposition \ref{evanescent properties}, $\alpha$ has almost fixed points, hence $\widehat{\alpha}^t$ has almost vanishing entropy (Theorem \ref{dictionary}). Finally, by Theorem \ref{general mixing}, we have that $\widehat{\alpha}^t \otimes \rho$ is either dissipative or ergodic of type $\III_1$. Finally, by Remark \ref{freeness}, up to replacing $\alpha$ by $\alpha \oplus \pi^{\N}$ for some faithful mixing orthogonal representation $\pi$ of $G$, we can assume that $\widehat{\alpha}^t$ is free for all $t$. 
\end{proof}

\begin{lem} \label{functional equation}
Let $w : \cH \rightarrow \{ z \in \C \mid |z| \leq 1 \}$ be a continuous function such that $w(\frac{a+b}{2})^2=w(a)w(b)$ for all $a,b \in \cH$. Then there exists a continuous affine function $f : \cH \rightarrow \R$ and a constant $\lambda \in [0,1]$ such that $w(x)=\lambda e^{\ri f(x)}$ for all $x \in \cH$. 
\end{lem}
\begin{proof}
If $w(x)=0$ for some $x \in \cH$, then clearly $w(x)=0$ for all $x \in \cH$ and we are done. Now suppose that $w(x) \neq 0$ for all $x \in \cH$. Since $\cH$ is simply connected and $\exp : \C \rightarrow \C^{\times}$ is a covering map, then there exists a continuous function $h : \cH \rightarrow \C$, such that $w(x)=e^{h(x)}$ for all $x \in \cH$. The assumption on $w$ and the continuity of $h$ imply that $h(\frac{a+b}{2})=\frac{h(a)+h(b)}{2}$ for all $a,b \in \cH$. This means that $h: \cH \rightarrow \C$ is an affine map. But since $|w(x)| \leq 1$ for all $x \in \cH$, the real part of $h$ must be negative, hence constant because it is affine. This is precisely what we wanted.
\end{proof}

\begin{thm} \label{general mixing}
Let $\alpha : G \curvearrowright \cH$ be an affine isometric action such that $\alpha^0$ is mixing and $\alpha$ has no fixed point. Then we have
\begin{enumerate}[ \rm (i)]
\item  $\widehat{\alpha}^t$ is weakly mixing for all $t < t_{\rm diss}(\alpha)$. In particular, $t_{\rm erg}(\alpha) \geq t_{\rm diss}(\alpha)$, with equality if $G$ is countable.
\item For all $0 < t < \frac{1}{\sqrt{2}} t_{\rm diss}(\alpha)$, the Krieger $T$-invariant of $\widehat{\alpha}^t$ is trivial. In particular, $\widehat{\alpha}^t$ can only be of type $\III_0$ or of type $\III_1$.
\end{enumerate}
\end{thm}
\begin{proof}
$(\rm i)$ It is enough to show that $\widehat{\alpha}^t$ is ergodic for all $t \in ]0,1[$ under the assumption that $\widehat{\alpha}$ is recurrent. Write $t=\cos \theta $ and $s=\sin \theta$ for some $\theta \in ]0,\frac{\pi}{2}[$. Let $R_\theta : \cH \times \cH^0 \rightarrow \cH^t \times \cH^s$ be the isometry defined in Proposition \ref{Rotation trick}. Then it induces a $G$-equivariant nonsingular isomorphism $$\widehat{R}_\theta : \widehat{\cH} \otimes \widehat{\cH}^0 \rightarrow \widehat{\cH}^t \otimes \widehat{\cH}^s$$ which induces a $G$-equivariant isomorphism $$\psi  : \rL^\infty(\widehat{\cH}^t \otimes \widehat{\cH}^s) \rightarrow \rL^\infty(\widehat{\cH} \otimes \widehat{\cH}^0).$$
Pick $f \in \rL^\infty(\widehat{\cH}^t)$ an $\widehat{\alpha}^t$-invariant function. By Theorem \ref{schmidt}, we must have $\psi(f \otimes 1) \in \rL^\infty( \widehat{\cH})$. But $f \otimes 1 \in \rL^\infty(\widehat{\cK})$ where $\cK=\cH^t \times \{sy\} \subset \cH^t \times \cH^s$ for any point $sy \in \cH^s$. Since $R_{\theta}(\cH) \cap \cK=\{(ty,sy)\}$, it follows from Proposition \ref{intersection maharam extension} that
$$ \psi(\rL^\infty(\widehat{\cK})) \cap \rL^\infty(\widehat{\cH})=\C.$$
This implies that $\widehat{\alpha}^t$ is ergodic.

$(\rm ii)$ We assume that $\widehat{\alpha}^{\sqrt{2}}$ is recurrent and we only have to show that if the $T$-set of $\widehat{\alpha}$ is nontrivial then $\alpha$ has a fixed point. Suppose that $r \neq 0$ is in the $T$-set of $\widehat{\alpha}$. Pick $x \in \cH$. Then we can find a unitary $u \in \rL^\infty(\widehat{\cH})$ such that $u\mu_{x}^{\ri r} \in \rL^\infty(\Mod(\widehat{\cH}))$ is $\Mod(\widehat{\alpha})$-invariant (here we view $\mu_x$ as a $1$-density hence as a function on $\Mod(\widehat{\alpha})$, see the appendix). Then $v=(u \otimes u)\mu_{(x,x)}^{\ri r} \in \rL^\infty(\Mod(\widehat{\cH} \otimes \widehat{\cH}))$ is $\Mod(\widehat{\alpha} \otimes \widehat{\alpha})$-invariant. Using the isometry $V : \cH^{\sqrt{2}} \times \cH^0 \rightarrow \cH \times \cH$ given by $V(\sqrt{2}x,\sqrt{2}\xi)=(x+\xi,x-\xi)$, we obtain a $G$-equivariant isomorphism 
$$ \psi : \rL^{\infty}( \widehat{\cH} \otimes \widehat{\cH}) \rightarrow \rL^\infty(\widehat{\cH}^{\sqrt{2}} \otimes \cH^0).$$
Then we know that
$$\Mod(\psi)(v)=\psi(u \otimes u) \mu_{(\sqrt{2}x,0)}^{\ri r} \in \rL^{\infty}(\Mod(\widehat{\cH}^{\sqrt{2}} \otimes \widehat{\cH}^0))$$ is $\Mod(\widehat{\alpha}^{\sqrt{2}} \otimes \widehat{\alpha}^0)$-invariant. By Theorem \ref{schmidt}, we must have $\psi(u \otimes u) \in \rL^\infty(\Mod(\widehat{\cH}^{\sqrt{2}}))$. This means that $\mu_{(\sqrt{2}y,\xi)}(\psi(u \otimes u))=\mu_{(\sqrt{2}y,0)}(\psi(u \otimes u))$ for all $(\sqrt{2}y, \xi) \in \cH^{\sqrt{2}} \times \cH^0$. This yields $\mu_{y+\xi}(u) \mu_{y-\xi}(u)=\mu_{y}(u)^2$ for all $(y, \xi) \in \cH \times \cH^0$. By Lemma \ref{functional equation}, we can find a constant $\lambda \in [0,1]$ and a continuous affine map $f : \cH \rightarrow \R$ such that $\mu_{y}(u)=\lambda e^{\ri f(y)}$ for all $y \in \cH$. Since $\mu_{y}(e^{\ri \widehat{f}})=e^{-\|f\|^2} e^{\ri f(y)}$ for all $y \in \cH$, we conclude that $u=\lambda e^{\|f\|^2} e^{\ri \widehat{f}}$, hence $u=e^{\ri \widehat{f}}$ because $u$ is a unitary. Now, by assumption, we have that $u\mu_{x}^{\ri r}$ is invariant by $\Mod(\widehat{\alpha})$. This easily yields for all $g \in G$
$$ \exp\left( -\ri r b(gx,x) \right)=\exp( \ri (f - gf))$$
where $b(gx,x)=\langle \cdot -\frac{gx+x}{2}, gx-x \rangle$ and $gf=f \circ g^{-1}$.
Let $f^0=\langle \cdot, \xi \rangle$ be the linear part of $f$. Then we get $r(gx-x)=\xi-g^0\xi$ for all $g \in G$, which means that $\alpha$ fixes $x+r^{-1}\xi$. Thus, we have shown that if $\widehat{\alpha}$ has a nontrivial $T$-set, then $\alpha$ has a fixed point.
\end{proof}

\subsection{Ergodicity and type in the weakly mixing case} \label{section weakly mixing}

We now investigate the ergodicity and type of the Gaussian actions when the linear part of the affine isometric action is only weakly mixing. Recall that an orthogonal representation $\pi$ is weakly mixing if and only if it does not admit any finite-dimensional subrepresentation. Recall also that $\pi$ is weakly mixing if and only if there exists a sequence $g_n \in G$ such that $\pi(g_n)$ converges weakly to $0$, i.e.\ $\langle \pi(g_n) \xi, \eta \rangle \to 0$ for all $\xi, \eta \in \cH$. By using this, it is easy to see that the pmp Gaussian action $\widehat{\pi}$ is ergodic if and only if $\pi$ is weakly mixing. This is no longer true in the nonsingular case since, as we know, the Gaussian action can be dissipative. In the mixing case, we saw that this is the only obstruction. But in the weakly mixing case, the situation is even more delicate as demonstrated by the following example.

\begin{example}  \label{counter-example weak mixing}
Let $\alpha : \Z \curvearrowright \cH$ be an affine isometric action of $\Z$ such that the linear part of $\alpha$ is mixing and $\delta(\alpha)=0$. Take $\pi : \Gamma \rightarrow O(\cK)$ be any weakly mixing orthogonal representation of a group $\Gamma$. Consider the product action $\beta = \alpha \times \pi : \Z \times \Gamma \rightarrow \cH \times \cK $. Then the linear part of $\beta$ is weakly mixing, $\widehat{\beta}^t$ is recurrent for all $t$ but $\widehat{\beta}^t$ is never ergodic for $t > 0$ because $\widehat{\alpha}^t$ is dissipative.
\end{example}

Intuitively, to avoid the situation of Example \ref{counter-example weak mixing}, one really needs to assume that the ``weakly mixing direction" is recurrent in some sense by adding a growth condition on the orbits. This section is an attempt to make this idea more precise. But let us first start with a very simple criterion.
\begin{prop} \label{weak-mixing basic criterion}
Let $\alpha : G \curvearrowright \cH$ be an affine isometric action without a fixed point and suppose that there exists a closed subgroup $L \subset G$ such that $\alpha|_L$ has a fixed point and $\alpha^0|_L$ is weakly mixing. Then $\widehat{\alpha}^t$ is weakly mixing of stable type $\III_1$ for all $t > 0$.
\end{prop}
\begin{proof}
We may assume that $t=1$. For the first part, take $\rho : G \curvearrowright (Y,\nu)$ a probability measure preserving action. Since the pmp Gaussian action $\widehat{\alpha}|_L$ is weakly mixing and preserves the probability measure $\mu_x$ for some $x \in \cH$, we know that $\rL^\infty( \Mod(\widehat{\cH}) \otimes Y)^L \subset \rL^\infty( \Mod(\{x\}) \otimes Y)$. Take $g \in G$ such that $gx \neq x$. Then $\widehat{\alpha}|_{gLg^{-1}}$ is also weakly mixing and preserves $\mu_{gx}$. We conclude that
$$\rL^\infty( \Mod(\widehat{\cH}) \otimes Y)^G \subset \rL^\infty( \Mod(\widehat{\{x\}}) \otimes Y) \cap \rL^\infty( \Mod(\widehat{\{gx\}}) \otimes Y)=\C \otimes \rL^\infty(Y)$$
by Proposition \ref{intersection maharam extension}. This shows that $\Mod(\widehat{\alpha})$ is weakly mixing.
\end{proof}

%

\begin{df}
Let $\pi : G \rightarrow \mathcal{O}(\cH)$ be an orthogonal representation.  A countable subset $\Lambda \subset G$ is called \emph{mixing with respect to $\pi$} if 
$$\forall \xi, \eta \in \cH, \; \lim_{g \in \Lambda, \: g \to \infty} \langle \pi(g) \xi,\eta \rangle=0.$$
When $\pi$ is the reduced Koopman representation of a probability measure preserving action $\sigma : G \curvearrowright (X,\mu)$, we will say that $\Lambda \subset G$ is \emph{mixing with respect to $\sigma$}.
\end{df}

\begin{df}
Let $\sigma : G \curvearrowright X$ be a nonsingular action. A countable subset $\Lambda \subset G$ is called \emph{recurrent with respect to $\sigma$} if for every subset $A \subset X$ with $A \neq \emptyset$ there exists infinitely many $g \in \Lambda$ such that $g A \cap A \neq \emptyset$.
\end{df}

Here is the key criterion which generalizes Theorem \ref{schmidt}.
\begin{thm}  \label{generalized schmidt}
Let $\sigma : G \curvearrowright X$ be a nonsingular action and $\rho : G \curvearrowright (Y,\mu)$ a probability measure preserving action. Suppose that there exists a subset $\Lambda \subset G$ such that $\Lambda$ is recurrent with respect to $\sigma$ and mixing with respect to $\rho$. Then every $\sigma \otimes \rho$-invariant function in $\rL^\infty(X \otimes Y)$ is contained in $\rL^\infty(X)$.
\end{thm}
\begin{proof}
The proof is essentially the same as in \cite{SW81}. Represent $X$ as the nonsingular space associated to a compact metrizable space $\Omega$ equipped with a borel probability measure $m$. Let $A \subset X \otimes Y$ be a $\sigma \otimes \rho$-invariant subset. Then we can represent $A$ as a borel function $\Omega \ni x \mapsto A_x \in \mathfrak{P}(Y)$ where $\mathfrak{P}(Y)$ is the Polish space of all equivalence classes of measurable subsets of $Y$ equipped with the distance $d_\mu(B,C)=\mu(B \triangle C)$.  The $G$-invariance of $A$ implies that for every $g \in G$ and $m$-almost every $x \in \Omega$, we have $A_{gx}=g(A_x)$. Take $\varepsilon > 0$. By Lusin's theorem, we can find a continuous function $F : \Omega \rightarrow \mathfrak{P}(Y)$ and a closed subset $C \subset \Omega$ with $m(C) \geq 1- \varepsilon$ such that $F_x=A_x$ for all $x \in C$. We may further assume that $F_{gx}=g(F_x)$ for all $x \in C$ and all $g \in G$ such that $gx \in C$.

Now, for $m$-almost every $x \in C$, we have $m(C \cap V) > 0$ for every neighborhood $V$ of $x$. Therefore, by assumption, for $m$-almost every $x \in C$, we can find a sequence $g_n \in \Lambda$ and a sequence $x_n \in \Omega$ such that both $x_n$ and $g_n x_n$ converge to $x$ and $\lim_n \mu(g_n(F_x) \cap F_x)=\mu(F_x)^2$. But then, by the continuity of $F$ and the $G$-invariance of $d_\mu$, we get $$\lim_n g_n(F_x)=\lim_n g_n(F_{x_n})=\lim_n F_{g_n x_n} = F_x.$$ We conclude that $\mu(F_x)=\mu(F_x)^2$ for $m$-almost every $x \in C$. Since $m(C) \geq 1-\varepsilon$ and $\varepsilon > 0$ is arbitrary, this shows that $\mu(A_x)=\mu(A_x)^2$ for $m$-almost every $x \in \Omega$. This means precisely that $A=A_0 \otimes Y$ for some subset $A_0 \subset X$ or equivalently that $1_A \in \rL^\infty(X)$ as we wanted.
\end{proof}

The following lemma is inspired by \cite[Proposition 4.3]{BKV19}.

\begin{lem} \label{general recurrent criterion}
Let $\sigma : G \curvearrowright X$ be a nonsingular action. Let $\Lambda \subset G$ be a countable subset such that
$$ \sum_{g \in \Lambda} \frac{g_* \mu}{\mu}=+\infty$$
for some faithful probability measure $\mu$ on $X$. Then $\Lambda \Lambda^{-1}$ is recurrent with respect to $\sigma$. This condition is satisfied if there exists a sequence of finite subsets $F_n \subset \Lambda$ such that
$$\lim_n \frac{1}{|F_n|^2}\sum_{g \in F_n} \int_X \frac{\mu}{g_*\mu} \rd \mu =0.$$
\end{lem}
\begin{proof}
Suppose that $\Lambda \Lambda^{-1}$ is not recurrent. Then there exists $\emptyset \neq A \subset X$ and a finite subset $F \subset G$ such that for all $g \in \Lambda \Lambda^{-1} \setminus F$, we have $gA \cap A=\emptyset$. We can moreover choose $A$ and $F$ such that $|F|$ is as small as possible. In that case, we claim that $gA=A$ for all $g \in F$. Indeed, assume that there exists $g \in F$ such that $gA \neq A$. Then we can find $\emptyset \neq B \subset A$ such that $gB \cap B=\emptyset$. Then by replacing $(A,F)$ by $(B,F \setminus \{g\})$,  we would contradict the minimality of $|F|$.

It follows from the choice of $(A,F)$ and the claim, that for all $g,h \in \Lambda$, we have either $g^{-1}A=h^{-1}A$ or $g^{-1}A \cap h^{-1}A =\emptyset$. Moreover, the map $\Lambda \ni g \mapsto g^{-1}A$ has fibers of cardinality at most $|F|$. Thus, we have
$$ \sum_{g \in \Lambda} \mu(g^{-1}A) \leq |F| \mu\left( \bigcup_{g \in \Lambda} g^{-1}A \right) \leq |F| \mu(X) <+\infty.$$
This shows that the measure $\nu=\sum_{g \in \Lambda} g_*\mu$ is not purely infinite, contradicting the assumption. Finally, the convexity of the function $t \mapsto t^{-1}$ gives
$$ \frac{1}{|F|^2}\sum_{g \in F} \int_X \frac{\mu}{g_*\mu} \rd \mu \geq \int_X \left( \sum_{g \in F} \frac{g_*\mu}{\mu} \right)^{-1} \rd \mu \geq \int_X \left( \sum_{g \in \Lambda} \frac{g_*\mu}{\mu} \right)^{-1} \rd \mu$$
for any finite subset $F \subset \Lambda$.
 \end{proof}
 
The following lemma is inspired by \cite[Proposition 5.3]{BKV19}.
\begin{lem} \label{affine recurrent criterion}
Let $\alpha : G \curvearrowright \cH$ be an affine isometric action. Suppose that for some subset $\Lambda \subset G$, some $s > 0$ and some $x \in \cH$, we have
$$ \sum_{g \in \Lambda} e^{-s\|gx-x\|^2} =+\infty $$
Then $\Lambda \Lambda^{-1}$ is recurrent with respect to $\Mod(\widehat{\alpha}^t) \otimes \rho$ for all $t < \sqrt{s}$ and for every probability measure preserving action $\rho$ of $G$.
\end{lem}
\begin{proof}
Take $t < \sqrt{s}$. Let $\pi_{\mu} : \Mod(\widehat{\cH}^t) \rightarrow \widehat{\cH}^t \otimes \R^*_+$ be the trivialization with respect to the Gaussian measure $\mu=\mu_{tx}$. Then we have to show that $\Lambda \Lambda^{-1}$ is recurrent with respect to the action $\sigma=\pi_{\mu} \circ \Mod(\widehat{\alpha}^t) \circ \pi_{\mu}^{-1}$ on $\widehat{\cH}^t \otimes \R^*_+$ given by 
$$g \cdot (\omega,\lambda)=(g\omega,\lambda \frac{g_* \mu}{\mu}(\omega)), \quad (\omega, \lambda) \in \widehat{\cH}^t \otimes \R^*_+.$$
Take $s_0 < s$ such that $t < \sqrt{s_0}$ and take $a > 1$ such that $(a^2+a)t^2 < 2s_0$. Let $\nu$ be the measure on $\R^*_+$ with density $\rd \nu(\lambda)=\min(\lambda,\lambda^{-1})^a \rd \lambda$. Then by the invariance of $\mu \otimes \rd \lambda$ with respect to $\sigma$, we have
\begin{align*}
& \int_{\widehat{\cH}^t \otimes \R^*_+} \frac{\mu \otimes \nu}{(\sigma_g)_* (\mu \otimes \nu)}\rd\mu \rd\nu =\\
& \int_{\widehat{\cH}^t \otimes \R^*_+}\left( \frac{\mu \otimes \nu}{\mu \otimes \rd \lambda} \right) \sigma_g\left(\frac{\mu \otimes \nu}{\mu \otimes \rd \lambda} \right)^{-1}   \rd\mu \rd\nu =\\
& \int_{\widehat{\cH}^t \otimes \R^*_+} \left(1 \otimes \frac{\rd \nu}{\rd \lambda}\right) \sigma_g \left( 1 \otimes \frac{\rd \nu}{\rd \lambda} \right)^{-1} \rd \mu \rd\nu =\\
& \int_{\widehat{\cH}^t \otimes \R^*_+} \min(\lambda,\lambda^{-1})^{a} \min\left(\lambda \frac{ g_*\mu}{\mu}(\omega), \left(\lambda \frac{ g_*\mu}{\mu}(\omega) \right)^{-1} \right)^{-a} \rd \mu(\omega) \rd\nu(\lambda)   \leq \\
&\nu(\R^*_+) \int_{\widehat{\cH}^t } \left( \frac{ g_*\mu}{\mu} \right)^{-a}+\left( \frac{ g_*\mu}{\mu}  \right)^{a} \rd \mu.
\end{align*}
Then one computes
$$ \int_{\widehat{\cH}^t } \left( \frac{ g_*\mu}{\mu} \right)^{\pm a} = \exp\left( \frac{1}{2}(a^2 \mp a)t^2 \|gx-x\|^2 \right) \leq \exp\left( s_0 \|gx-x\|^2 \right).$$
Now, since $s_0 < s$, Lemma \ref{growth exponent} shows that there exists a sequence $r_n> 0$ with $\lim_n r_n= +\infty$ and a sequence of finite subsets $F_n \subset \Lambda$ such that $\lim_n |F_n| e^{-s_0 r_n} = +\infty$ and $\|gx-x\|^2 \leq r_n$ for all $g \in F_n$. This implies
$$\frac{1}{|F_n|^2}\sum_{g \in F_n} \int_{\widehat{\cH}^t \otimes \R^*_+} \frac{\mu \otimes \nu}{(\sigma_g)_* (\mu \otimes \nu)}\rd\mu \rd\nu \leq \frac{2 \nu(\R^*_+)}{|F_n|} e^{s_0r_n} \to 0$$
By Lemma \ref{general recurrent criterion} we conclude that $\Lambda \Lambda^{-1}$ is recurrent with respect to $\Mod(\sigma) \otimes \rho$ for any probability measure preserving action $\rho$ (apply the criterion of Lemma $\ref{general recurrent criterion}$ to the measure $\mu \otimes \nu \otimes \eta$ where $\eta$ is an invariant probability measure for $\rho$).

\end{proof}

\begin{lem} \label{mixing set}
Let $\pi : G \rightarrow \mathcal{O}(\cH)$ be an orthogonal representation of a locally compact group $G$ which is weakly mixing. Then there exists an infinite subset $\Lambda \subset G$ such that $ \Lambda \Lambda^{-1}$ is mixing with respect to $\pi$.
\end{lem}
\begin{proof}
Let $d$ be a distance on the unit ball of $\B(\cH)$ which metrizes the weak operator topology. Define inductively a sequence $(g_n)_{n \in \N}$ of elements of $G$ as follows: take $g_0$ to be any element of $G$ and then for each $n \geq 1$, choose $g_n \in G$ such that
$$ \forall k < n, \quad d(\pi(g_{k}g_{n}^{-1}),0)+d(\pi(g_{n}g_{k}^{-1}),0) \leq 2^{-n}.$$
Now, let $\Lambda=\{ g_{n} \mid n \in \N \}$. By construction, we have $d(\pi(g_{n}g_{m}^{-1}),0) \leq 2^{-\max(n,m)}$ for all $n,m \in \N$ such that $n \neq m$. This implies that $ \Lambda \Lambda^{-1}$ is mixing.
\end{proof}

The main application of this section is the following theorem.

\begin{thm}[Theorem \ref{letter non T}] \label{strong non T}
Let $G$ be a locally compact group without property (T). Then there exists an affine isometric action $\alpha : G \curvearrowright \cH$ such that for all $t > 0$, the nonsingular Gaussian action $\widehat{\alpha}^t$ satisfies :
\begin{enumerate}[ \rm (i)]
\item $\widehat{\alpha}^t$ is free.
\item $\widehat{\alpha}^t$ is weakly mixing of stable type $\III_1$.
\item $\widehat{\alpha}^t$ has almost vanishing entropy. In particular, $\widehat{\alpha}^t$ has an invariant mean and it is nonamenable if $G$ is nonamenbale.
\end{enumerate}
\end{thm}
\begin{proof}
Since $G$ does not have property (T), it admits a weakly mixing orthogonal representation $\pi$ that has almost invariant vectors. Let $\Lambda \subset G$ be an infinite subset such that $\Lambda \Lambda^{-1}$ is mixing with respect to $\pi$. By Proposition \ref{existence evanescent}, we can find an affine isometric action $\alpha : G \curvearrowright \cH$ such that $\alpha$ is evanescent and has no fixed point and $\alpha^0$ contained in a multiple of $\pi$. Moreover, we can choose $\alpha$ such that
$$ \sum_{g \in \Lambda} e^{-s\|gx-x\|^2}=+\infty$$
for all $s > 0$ and some $x \in \cH$. Then $\Lambda \Lambda^{-1}$ is mixing with respect to $\alpha^0$ and by Lemma \ref{affine recurrent criterion}, it is recurrent with respect to $\Mod(\widehat{\alpha}^t) \otimes \rho$ for all $t > 0$ and every probability measure preserving action $\rho$ of $G$. Then, as in Lemma \ref{affine recurrent and mixing}, Theorem \ref{generalized schmidt} shows that for any $\alpha$-invariant subspace $\cK$, we have
$$ \rL^\infty( \Mod( \widehat{\cH} \otimes X))^{\Mod(\widehat{\alpha} \otimes \rho)} \subset \rL^\infty(\Mod(\widehat{\cK} \otimes X)).$$
Then we can conclude as in Theorem \ref{evanescent ergodicity} that $\widehat{\alpha} \otimes \rho$ is ergodic of type $\III_1$. For the freeness property, we can modify $\alpha$ as in Theorem \ref{strong Haagerup}.
\end{proof}

\subsection{Strong ergodicity}
We now give a criterion for the strong ergodicity of Gaussian actions. We recall the definition of strong ergodicity. Let $\sigma : G \curvearrowright X$ be a nonsingular action. We say that a bounded sequence $(a_n)_{n \in \N}$ in $\rL^\infty(X)$ is \emph{almost invariant} with respect to $\sigma$, if the sequence of functions $g \mapsto \sigma_g(a_n)-a_n$ converges to $0$ in the measure topology, uniformly on compact subsets of $G$. We say that $\sigma$ is strongly ergodic if every almost invariant bounded sequence $(a_n)_n$ is \emph{trivial}, i.e.\ there exists a bounded sequence $(\lambda_n)_n$ in $\C$ such that $a_n-\lambda_n1$ converges to $0$ in the measure topology.

 Recall that an orthogonal representation $\pi : G \rightarrow \mathcal{O}(H)$ has \emph{spectral gap} if it has no almost invariant vectors. We say that $\pi$ has \emph{stable spectral gap} if the representation $\pi \otimes \rho$ has spectral gap for every representation $\rho : G \rightarrow \mathcal{O}(H)$. Finally, we say that a pmp action $\rho : G \curvearrowright (Y,\nu)$ has (stable) spectral gap if its reduced Koopman representation $\rho : G \curvearrowright \rL^2(Y,\nu)^0$ has (stable) spectral gap.

\begin{lem} \label{lemma spectral gap}
Let $\sigma : G \curvearrowright X$ be an ergodic nonsingular action and $\rho : G \curvearrowright (Y,\nu)$ a pmp action. Suppose that:
\begin{enumerate}[ \rm (i)]
\item $\rho$ has stable spectral gap, hence there exists $\kappa > 0$ and a symmetric compact subset $K \subset G$ such that
$$ \| \xi \|^2 \leq \kappa \int_{K} \| \pi(g) \xi-\xi \|^2 \rd g \quad \text{ for all } \; \xi \in \rL^2(X) \otimes \rL^2(Y,\nu)^0$$
where $\pi$ is the Koopman representation of $\sigma \otimes \rho$.
\item We have
$$ \kappa \int_K \| g \mu^{1/2}- \mu^{1/2} \|^2 \rd g < 1$$
for some faithful probability measure $\mu$ on $X$.
\end{enumerate}
Then for any bounded almost invariant sequence $(a_n)_n$ in $\rL^\infty(X \otimes Y)$, we have $\lim_n a_n-E(a_n)=0$ in the measure topology, where $E : \rL^\infty(X \otimes Y) \rightarrow \rL^\infty(X)$ is the unique conditional expectation such that $\mu \circ E=\mu \otimes \nu$.
\end{lem}
\begin{proof}
Let $b_n=a_n-E(a_n)$. We have to show that $b_n \to 0$. Observe that the sequence $(b_n)_n$ is again almost invariant as well as the sequence $(E(|b_n|^2))_n$. Up to extracting a subsequence, we may assume that $(E(|b_n|^2))_n$ converges to some $f \in \rL^\infty(X)$ in the weak$^*$-topology. Then the function $f$ is $G$-invariant and since $\sigma$ is ergodic, $f$ must be constant equal to $\lambda=\lim_n \mu( E(|b_n|^2))=\lim_n (\mu\otimes \nu)(|b_n|^2)$.

Let $\eta=(\mu \otimes \nu)^{1/2}$ and $\xi_n=b_n \eta  \in \rL^2(X) \otimes \rL^2(Y,\nu)^0 $. Then by item $(\rm i)$, we have
\begin{align*}
 \| \xi_n \|^2 &\leq \kappa \int_K \| \pi(g) \xi_n- \xi_n \|^2 \rd g \\
 &\simeq \kappa \int_K \| b_n (\pi(g) \eta-\eta) \|^2  \rd g \\
  &= \kappa \int_K \| |b_n|^2 \left( \pi(g) \eta-\eta \right)^2 \|   \rd g\\
 &= \kappa \int_K \| E(|b_n|^2)  \left( \pi(g) \eta-\eta \right)^2 \| \rd g\\
 &\simeq \kappa \lambda \int_K \| \pi(g)\eta-\eta \|^2 \rd g \\
 &\simeq \kappa \| \xi_n \|^2 \int_K \| \pi(g) \mu^{1/2}-\mu^{1/2} \|^2 \rd g.
\end{align*}
We conclude by item $(\rm ii)$ that $\lim_n \| \xi_n \|=0$ as we wanted. 
\end{proof}

\begin{thm} \label{strongly ergodic}
Let $\alpha : G \curvearrowright \cH$ be an affine isometric action such that $\alpha^0$ has stable spectral gap. Suppose that $t_{\rm erg}(\alpha) > 0$. Then there exists $t_0 > 0$ such that $\widehat{\alpha}^t$ is strongly ergodic for all  $t < t_0$.
\end{thm}
\begin{proof}
Take $x \in \cH$. Since $\| g \mu_{tx}^{1/2}-\mu_{tx}^{1/2} \|^2 =2-2e^{-\frac{t^2}{8}\|gx-x\|}$ converges to $0$ uniformly on compact subsets of $G$ when $t \to 0$, we can find $s > 0$ with $s < t_{\rm erg}(\alpha)$ such that $\sigma = \widehat{\alpha}^t$ and $\rho=\widehat{\alpha}^0$ satisfy the assumption of Lemma \ref{lemma spectral gap} for all $t < s$. 

Take $r=\frac{1}{\sqrt{2}}t$ and identify as usual $\widehat{\cH}^t \otimes \widehat{\cH}^0$ with $\widehat{\cH}^r \otimes \widehat{\cH}^r$ and $\widehat{\alpha}^t \otimes \widehat{\alpha}^0$ with $\widehat{\alpha}^r \otimes \widehat{\alpha}^r$. Take $(a_n)_n$ in $\rL^\infty(\widehat{\cH}^r)$ a bounded almost invariant sequence for $\widehat{\alpha}^r$. Then $a_n \otimes 1 \in \rL^\infty(\widehat{\cH}^r \otimes \widehat{\cH}^r)$ is also $\widehat{\alpha}^r \otimes \widehat{\alpha}^r=\widehat{\alpha}^t \otimes \widehat{\alpha}^0$-almost invariant. Therefore, by Lemma \ref{lemma spectral gap}, we know that the sequence $a_n \otimes 1 - E(a_n \otimes 1)$ converges to $0$ in the measure topology, where $E$ is the conditional expectation from $\rL^\infty(\widehat{\cH}^t \otimes \widehat{\cH}^0)$ onto $\rL^\infty(\widehat{\cH}^t)$. But the isometry $\xi \mapsto -\xi$ of $\cH^0$ induces a $G$-equivariant isometry of $\cH^t \times \cH^0$ which flips the two copies of $\cH^r$ and leaves $\cH^t$ fixed. We conclude that $a_n \otimes 1-1 \otimes a_n$ converges to $0$ in the measure topology and this easily implies that $(a_n)_n$ is trivial.
\end{proof}

\begin{cor}
Let $\alpha : G \curvearrowright \cH$ be an affine isometric action such that $\alpha^0$ is mixing and has stable spectral gap. Then there exists $t_0 > 0$ such that the actions $\widehat{\alpha}^t$ are strongly ergodic of type $\III_1$ for all $t \in ]0,t_0[$.
\end{cor}
\begin{proof}
Observe first that $G$ is nonamenable. Therefore $t_{\rm diss}(\alpha) > 0$. By Theorem \ref{general mixing}, we know that $\widehat{\alpha}^t$ is ergodic of type $\III_1$ or type $\III_0$ for $t > 0$ small enough. By Theorem \ref{strongly ergodic}, we know that $\widehat{\alpha}^t$ is strongly ergodic for $t$ small enough. Since a type $\III_0$ action is never strongly ergodic, we conclude that $\widehat{\alpha}^t$ is of type $\III_1$ for $t > 0$ small enough.
\end{proof}

%

\section{Trees, random walks and groups acting on trees} \label{trees}
A tree $T$ is a connected unoriented graph without cycles. We will abusively identify $T$ with its set of vertices. We denote by $\overrightarrow{E}(T) \subset T^2$ the set of oriented edges of $T$ (i.e.\ couples of adjacent vertices). We denote by $E(T)$ the set of all unoriented edges of $T$ (pairs of adjacent vertices). An \emph{orientation} of $T$ is a map 
$$\omega : E(T) \ni e \mapsto (\omega_+(e),\omega_-(e)) \in \overrightarrow{E}(T)$$ such that $e=\{ \omega_+(e),\omega_-(e) \}$ for all $e \in E(T)$.

 A vertex of degree one is called a \emph{leaf}. A \emph{subtree} $S \subset T$ is a connected subgraph (this implies that $S$ itself is a tree). We say that $S$ is a \emph{line} (resp.\ a \emph{half-line}, resp.\ a \emph{segment}) if every vertex of $S$ has degree at most two and $S$ has no leaves (resp.\ only one leaf, resp.\ two leaves). For every $x,y \in T$, we denote by $[x,y] \subset T$ the segment defined as the smallest subtree containing $x$ and $y$. We view $T$ as a metric space where $d(x,y)$ is simply the cardinality of $[x,y]$ minus $1$. Note that one can recover the graph structure of $T$ from this metric. If we fix a distinguished vertex $\rho \in T$, we will call it a \emph{root}. In that case, we will use the notation $|x|=d(x,\rho)$ for all $x \in T$. Our trees $T$ are always assumed to be countable. We will also often assume that they are \emph{locally finite}, i.e.\ every vertex of $T$ has finite degree. For $q \geq 2$, we say that $T$ is \emph{$q$-regular} if every vertex of $T$ has degree $q$.

\subsection{The boundary of a tree}
Let $T$ be a tree. Let $\partial T$ be the set of all possible orientations of $T$ such that every vertex admits at most one outgoing edge. We equip $\overline{T}$ with the topology of pointwise convergence on $T$. Then $\overline{T}$ is a compact space. For every $x \in T$, we can define an ortientation of $T$ for which every edge is oriented towards $x$. This allows to identify $T$ with an open discrete subset of $\overline{T}$. The set $\partial T = \overline{T} \setminus T$ is called the \emph{boundary} of $T$. It consists of all orientations for which every vertex admits exactly one outgoing edge. Since $\overline{T}$ is compact, the boundary $\partial T$ is nonempty when $T$ is infinite.

For every $\omega \neq \omega' \in \partial T$ there exists a smallest subtree $S \subset T$ such that $\{ \omega, \omega'\} \subset \overline{S}$. The tree $S$ is a line and is denoted $[\omega, \omega']$.

Fix a root $\rho \in T$. Then one can metrize the topology of $\partial T$ by the following distance
$$ d_\rho(\omega, \omega')=e^{-d(\rho, [\omega, \omega'])}, \quad \omega, \omega' \in \partial T.$$
It is a fact that the Hausdorff dimension of the metric space $(\partial T,d_\rho)$ does not depend on the choice of $\rho$. It is called the \emph{Hausdorff dimension} of $\partial T$ and we denote it by $\dim_H \partial T$. When $\Aut(T) \backslash T$ is finite the Hausdorff dimension of $\partial T$ can be computed by the following simple formula
$$ \dim_H \partial T = \lim_{n \to \infty} \frac{1}{n} \log | \{ x \in T \mid |x|\leq n\} | $$
which (see Lemma \ref{growth exponent}) is also the exponent of convergence of the series
$$ \sum_{x \in T} e^{-s|x|}, \; s > 0.$$
For example, if $T$ is $q$-regular with $q \geq 2$, we have $\dim_H \partial T=\log(q-1)$.

\subsection{Tree-indexed random walks}
Let $T$ be an infinite tree and fix a root $\rho \in T$. Let $X$ be a random variable and let $(X_e)_{e \in E(T)}$ be a family of i.i.d.\ random variables with the same distribution as $X$. For each $v \in T$, set 
$$S_v = \sum_{e \in E([\rho,v])}  X_e.$$
The random process $(S_v)_{v \in T}$ is a random walk on $\R$ indexed by the tree $T$. In the study of Gaussian actions (and nonsingular Bernoulli actions) of groups acting on trees, we will need to estimate the speed of this tree-indexed random walk, i.e.\ the growth of $\frac{1}{|v|} S_v$ when $v \to \infty$.

\begin{thm}[{\cite[Theorem 4]{LP92}}]
Set
	\[m(y) = \inf_{x \leq 0} {\mathbb E}[e^{x(X-y)}]\]
	and
	\[m_1(z) = \sup \{y \mid m(y) < z\}.\]	
If $\delta = \dim_H \partial T > 0$ then almost surely we have
	\[\inf_{\omega \in \partial T} \uplim_{v \in [\rho,\omega]} \frac{1}{|v|}S_v = m_1(e^{-\delta})\]
\end{thm}

\begin{example}[Gaussian random walk] \label{gaussian random walk}
Suppose that $X$ is a standard Gaussian random variable. Then 
$$m(y)=\inf_{x \leq 0} e^{\frac{1}{2}x^2-xy}$$
hence $m(y)=e^{-\frac{1}{2}y^2}$ if $y \leq 0$ and $m(y)=1$ if $y \geq 0$. This implies that
$$m_1(e^{-\delta})=-\sqrt{2\delta}.$$
Therefore, we obtain
$$\inf_{\omega \in \partial T} \uplim_{v \in [\rho,\omega]} \frac{1}{|v|}S_v = -\sqrt{2\delta}.$$
By replacing $S_v$ by $-S_v$, we also obtain
$$ \sup_{\omega \in \partial T} \lowlim_{v \in [\rho,\omega]} \frac{1}{|v|}S_v = \sqrt{2\delta}.$$
\end{example}

\begin{example}[Bernoulli random walk] \label{bernoulli random walk}
Let $X$ be a Bernoulli random variable with values in $\{-1,1\}$ with parameter $p \in ]0,1[$. Then
$$ m(y) = \inf_{x \leq 0} (p e^{x(\pm 1-y)}+ (1-p) e^{x(\mp 1-y)}).$$
Observe that $m(0)=2\sqrt{p(1-p)}$. Therefore, by continuity of $m$, we get 
$$\inf_{\omega \in \partial T} \uplim_{v \in [\rho,\omega]} \frac{1}{|v|}S_v < 0$$
if and only if $2\sqrt{p(1-p)} > e^{-\delta}$. By replacing $S_v$ with $-S_v$, we also have
$$\sup_{\omega \in \partial T} \lowlim_{v \in [\rho,\omega]} \frac{1}{|v|}S_v > 0$$
if and only if $2\sqrt{p(1-p)} > e^{-\delta}$.
\end{example}

The following statement gives a more precise estimate for regular trees. It is based on an additive martingale argument which is classical in the theory of branching random walks. We are very grateful to Michel Pain for explaining it to us.
\begin{thm} \label{speed regular tree}
Suppose that $T$ is a $q$-regular tree with $q \geq 3$ and that $X$ is a standard Gaussian random variable. Then 
$$\mathbb{P}\left( \frac{1}{|v|} S_v \leq \sqrt{2\delta} \text{ for all } v \in T \right) > 0.$$
\end{thm}
\begin{proof}
Let $\delta=\log(q-1)=\dim_H \partial T$.  Fix $s > 0$ and define a sequence of random variables
$$ W_{n}=\sum_{|v|=n} \exp\left(sS_v \right)$$
Observe that for every $v,w \in T$ with $|v|=n$ and $|w|=n+1$, the random variable $S_w-S_v$ is independent from $S_u$ for all $u \in T$ with $|u| \leq n$. From this observation, one can easily compute the conditional expectation
$$ \mathbb{E}(W_{n+1} \mid  S_v, \: |v| \leq n )=(q-1) e^{\frac{1}{2}s^2} W_n.$$
This means that the random process 
$$\left( e^{-n(\frac{1}{2}s^2+\delta)} W_n \right)_{n \geq 1}$$
is a martingale. Since it is positive, it converges almost surely to a finite limit when $n \to \infty$. Since we have
$$ W_n \geq \max_{|v|=n} e^{sS_v}$$
we conclude that
$$ \uplim_{n \to \infty}  \max_{|v|=n}  \left(  -n \left(\frac{1}{2}s^2+\delta \right) + sS_v \right) < +\infty$$
almost surely. Now take $s=\sqrt{2\delta}$. We get
$$ \uplim_{n \to \infty}  \max_{|v|=n}  \left(  - n\sqrt{2\delta}+ S_v \right) < +\infty$$
almost surely. In particular, there exists a constant $C > 0$ such that
$$\mathbb{P}\left( S_v \leq \sqrt{2\delta}|v|+C \text{ for all } v \in T \right) > 0.$$
Note that for any $w \in T$, by changing the root $\rho$ to $w$, we also get
$$\mathbb{P}\left( S_v-S_w \leq \sqrt{2\delta}|v|+C \text{ for all } v \in T \right) > 0.$$
Now, let $w$ be a neighbour of $\rho$. Let $T_w=\{ v \in T \mid |v| \geq d(v,w) \}$ be the set of all childs of $w$. Since $\mathbb{P}(S_w < -C) > 0$ and $S_v-S_w$ is independent from $S_w$ for all $v \in T_w$, we get
$$\mathbb{P}\left( S_v \leq \sqrt{2\delta}|v| \text{ for all } v \in T_w \right) > 0.$$
Since this holds for all (finitely many) neighbours $w$ of $\rho$, we conclude, again by independence, that
$$\mathbb{P}\left( S_v \leq \sqrt{2\delta}|v| \text{ for all } v \in T \right) > 0.$$
\end{proof}

\subsection{Actions of locally compact groups on trees}
The automorphism group $\Aut(T)$ of a locally finite tree $T$ is a locally compact group for the topology of pointwise convergence on $T$. Moreover, for this topology, the stabilizer of every vertex is an open compact subgroup of $\Aut(T)$. Thus $\Aut(T)$ is a totally disconnected group. An action of a locally compact group $G$ on a tree $T$ is simply a continuous homomorphism $\sigma : G \rightarrow \Aut(T)$.

We start with two very elementary propositions. The proofs are  left to the reader.

\begin{prop}
Let $\sigma : G \curvearrowright T$ be an action of a locally compact group on a locally finite tree. Then the following are equivalently:
\begin{enumerate}[ \rm (i)]
\item $\sigma(G)$ is relatively compact in $\Aut(T)$.
\item $G$ fixes some vertex or some edge of $T$.
\item The orbit $G \cdot x$ is finite for some $x \in T$.
\item The orbit $G \cdot x$ is finite for every $x \in T$.
\end{enumerate}
\end{prop}

\begin{prop}
Let $\sigma : G \curvearrowright T$ be a faithful action of a locally compact group on a locally finite tree. Then the following are equivalently:
\begin{enumerate}[ \rm (i)]
\item $\sigma(G)$ is closed in $\Aut(T)$.
\item $\sigma$ is proper.
\item For some $v \in T$, the stabilizer $G_v=\{ g \in G \mid gv=v\}$ is compact.
\item For every $v \in T$, the stabilizer $G_v=\{ g \in G \mid gv=v\}$ is compact.
\end{enumerate}
\end{prop}

Let $T$ be a locally finite tree. Let $g \in \Aut(T)$ be an automorphism. Then either $g$ fixes a point or and edge (in which case, we say that $g$ is \emph{elliptic}) or there exists a unique line $L \subset T$ such that $g$ acts by translation on $L$ (in which case we say that $g$ is \emph{hyperbolic} and $L$ is called the \emph{axis} of $g$). Note that in the latter case, the end points of $L$ are the only fixed points of $g$ in $\partial T$. More generally, we have the following classification of closed automorphism groups of trees \cite[Section 3]{Ti70} which is a particular case of Gromov's classification of isometry groups of hyperbolic spaces \cite[Section 3.1]{Gro87} (see also \cite{CCMT15} and \cite[Section 1.1.4]{DSU17}).

\begin{thm}
Let $T$ be a locally finite tree and $G < \Aut(T)$ a closed subgroup. Let $\Lambda_G$ be the limit set of $G$, i.e.\ the set of accumulation points of the orbits of $G$ in $\partial T$. Then $G$ belongs to one and exactly one of the following types:
\begin{enumerate}[ \rm (i)]
\item Elliptic: $\Lambda_G$ is empty, $G$ is compact and fixes a vertex or an edge of $T$ and every element of $G$ is elliptic.
\item Parabolic: $\Lambda_G$ consists of a single point, which is fixed by $G$ and every element of $G$ is elliptic.
\item Axial: $\Lambda_G$ consists of only two points, $G$ preserves the line $L$ joining them and there exists a hyperbolic element whose axis is $L$.
\item Focal: $G$ has a unique fixed point in $\partial T$, $G$ contains hyperbolic elements (all of their axes sharing a common end) and $\Lambda_G$ is uncountable.
\item General type: $G$ does not fix any point or pair of points in $T \cup \partial T$, $G$ contains two hyperbolic elements without a common end point and $\Lambda_G$ is uncountable.
\end{enumerate}
\end{thm}

For the following proposition we refer to \cite[Section 2]{CM11}

\begin{prop} \label{convex hull}
Let $T$ be a locally finite tree and $G < \Aut(T)$ a closed subgroup. Suppose that $G$ contains hyperbolic elements (i.e.\ $G$ is not elliptic or parabolic). Then there exists a unique minimal nonempty $G$-invariant subtree $S \subset T$ and we have $\Lambda_G=\partial S$. Moreover, $G$ acts cocompactly on $S$ if and only if $G$ is compactly generated.
\end{prop}

\begin{lem} \label{limit set cocompact}
Let $T$ be a locally finite tree and $G < \Aut(T)$ a closed subgroup. Let $H < G$ be a closed compactly generated subgroup such that $H$ contains hyperbolic elements and $\Lambda_H=\Lambda_G$. Then $G/H$ is compact.
\end{lem}
\begin{proof}
By Proposition \ref{convex hull}, the subtree $S \subset T$ is the convex hull of $\partial S=\Lambda_H=\Lambda_G$ and $H$ acts cocompactly on $S$. Since $\Lambda_G$ is $G$-invariant, this implies that $S$ is also $G$-invariant. Take $v \in S$. Then $G \cdot v \subset S$. Thus the action of $H$ on $G \cdot v$ has finitely many orbits. But as $G$-spaces, we have $G \cdot v \simeq G/K$ where $K=G_v$ is the stabilizer of $v$ in $G$. Therefore the action of $H$ on $G/K$ has finitely many orbits. Equivalently, the action of $K$ on $G/H$ has finitely many orbits. Since $K$ is compact, we conclude that $G/H$ is compact.
\end{proof}

\subsection{The Poincar\'e exponent and Patterson-Sullivan theory}
In this section, we review a construction of Patterson relating the Poincar\'e exponent of isometry groups of hyperbolic spaces to the existence of \emph{conformal densities} on the boundary. This theory has been generalized to large classes of negatively curved spaces including trees \cite{DSU17}. Moreover, in the case of trees, the theory can be easily adapted to the case of locally compact group of isometries and not only discrete groups. The reason is that even when the group is not discrete, the \emph{orbits} remain discrete as we see in the following proposition.

\begin{df}
Let $T$ be a locally finite tree and $G < \Aut(T)$ a closed subgroup. The Poincar\'e exponent $\delta(G)$ of $G$ is the exponent of convergence of the Poincar\'e integral
$$ \int_G e^{-sd(gx,y)} \rd g , \; s > 0$$
where $x,y$ is any pair of elements of $T$.
\end{df}

\begin{prop} \label{poincare tree}
Let $T$ be a locally finite tree and $G < \Aut(T)$ a closed subgroup. The Poincar\'e exponent $\delta(G)$ of $G$ is given by the exponent of convergence of the following series:
$$ \sum_{z \in G \cdot x} e^{-sd(z,y)}, \; s > 0$$
where $x,y$ is any pair of elements of $T$. In particular, if $G$ acts cocompactly on $T$, then $\delta(G)$ is the Hausdorff dimension of $\partial T$. 
\end{prop}
\begin{proof}
Let $K=G_x$ be the stabilizer of $x$. Then $K$ is a compact open subgroup of $G$. Therefore if $\Delta \subset G$ is a set of representatives of $G/K$, we have for all $s > 0$,
\begin{align*}
\int_G e^{-sd(gx,y)} \rd g &=\sum_{g \in \Delta} \int_{K} e^{-sd(gkx,y)} \rd k \\
 &=\sum_{g \in \Delta} \int_{K} e^{-sd(gx,y)} \rd k \\
&=M \sum_{z \in G \cdot x} e^{-sd(z,y)}
\end{align*}
where $M$ is the Haar volume of $K$.

Now, suppose that $G$ acts cocompactly on $T$. Take $F$ a finite set of representatives of each orbit of $G \curvearrowright T$. Fix $y \in T$. Then $\dim_H \partial T$ is the exponent of convergence of the series
$$ \sum_{z \in T} e^{-sd(z,y)} = \sum_{x \in F} \sum_{z \in G \cdot x} e^{-sd(z,y)}, \quad s > 0.$$
Since $F$ is finite, this exponent of convergence is precisely $\delta(G)$.
\end{proof}

\begin{rem}
Let $G \subset \Aut(T)$ be a subgroup which is not necessarily closed. Then $G \cdot x=\overline{G} \cdot x$ for every $x \in T$. Therefore $\delta(\overline{G})$ is also equal to the exponent of convergence of the series
$$ \sum_{y \in G \cdot x} e^{-sd(y,x)}, \; s > 0.$$
\end{rem}

The definition of a conformal density uses the following Busemann cocycle.

\begin{df} \label{Busemann cocycle}
Let $T$ be a locally finite tree. For $\omega \in \partial T$, the \emph{Busemann cocycle} $b : T \times T \rightarrow C(\partial T)$ is defined by the formula
$$b(x,y) : \omega  \mapsto \lim_{z \to \omega} d(x,z)-d(y,z)$$ for all $x,y \in T$.
\end{df}

\begin{df} \label{conformal density}
Let $T$ be a locally finite tree. A \emph{conformal density of dimension $s$} on $T$ is a family of equivalent positive finite measures $(\nu_x)_{x \in T}$ on $\partial T$ such that 
$$ \frac{\rd \nu_y}{\rd \nu_x}=e^{-sb(y,x)}$$
for all $x,y \in T$. We say that $\nu$ is $G$-invariant, with respect to some subgroup $G \subset \Aut(T)$, if $\nu_{gx}=g_* \nu_x$ for all $x \in T$ and all $g \in G$.
\end{df}

\begin{thm} \label{Patterson-Sullivan}
Let $T$ be a locally finite tree and $G < \Aut(T)$ a closed subgroup of general type. Then there exists a $G$-invariant conformal density of dimension $\delta(G)$ on $\partial T$. Conversely, if there exists a $G$-invariant conformal density of dimension $s$, then $s \geq \delta(G)$.
\end{thm}
\begin{proof}
A proof in the discrete case can be found in the original articles \cite{Pa76} and \cite{Su79} or in \cite[Theorem 4.8 and 4.11]{Qui06}. The proof in the locally compact case is exactly the same as in the discrete case if one replaces sums over the group by sums over an orbit as in Proposition \ref{poincare tree}. We leave the details to the reader.
\end{proof}

\begin{thm} \label{coamenable exponent}
Let $T$ be a locally finite tree and $G \subset \Aut(T)$ a closed subgroup of general type. If $H \lhd G$ is a normal open subgroup such that $G/H$ is amenable then $\delta(G)=\delta(H)$.
\end{thm}
\begin{proof}
The proof when $G$ is discrete can be found in \cite[Th\'eor\`eme 0.1]{RT13}. The proof in the locally compact case is exactly the same as in the discrete case if one replaces sums over the group by sums over an orbit as in Proposition \ref{poincare tree}. We leave the details to the reader.
\end{proof}

\begin{thm} \label{subgroup of large critical exponent}
Let $T$ be a locally finite tree and $G < \Aut(T)$ a closed subgroup of general type. Then for any $s < \delta(G)$, there exists a closed compactly generated subgroup of general type $H < G$ such that $\delta(H) > s$. Moreover, we can choose $H$ such that $G/H$ is not compact.
\end{thm}
\begin{proof}
The proof of the first part is the same as in \cite[Corollary 6]{Su79}.
Let $(G_i)_{i \in I}$ be an increasing net of closed compactly generated subgroups of $G$ such that $G=\bigcup_{i \in I} G_i$. By adding hyperbolic elements, we can assume that every $G_i$ is of general type. For each $i \in I$, take $(\nu^i_x)_{x \in T}$ a $G_i$-invariant conformal density of dimension $\delta(G_i)$. Fix $y \in T$, and for all $i \in I$, normalize $\nu^i_y$ so that it becomes a probability measure. Then up to extracting a subnet, we can assume that $\nu^i_y$ converges weakly to a probability measure $\nu_y$ on $\partial T$. Then, we have $g_* \nu_y=e^{-d b(gy,y)} \nu_y$ for all $g \in \bigcup_{i \in I} G_i=G$ where $d=\lim_i \delta(G_i)$. This means that the formula $\nu_x=e^{-db(x,y)} \nu_y$ defines a $G$-invariant conformal density of dimension $d$. We conclude that $d \geq \delta(G)$ by Theorem \ref{Patterson-Sullivan}, hence $\delta(G_i) > s$ for $i$ large enough.

For the second part, it is enough to show that if $G$ is a compactly generated group of general type, then there exists a closed subgroup of general type $H < G$ with $\delta(H)=\delta(G)$ such that $H$ is not cocompact in $G$. Suppose first that $G$ is unimodular. Then, by \cite[Corollary 4.8.(a)]{BK90}, $G$ contains a uniform lattice $\Gamma < G$ which is a finitely generated free group. Take $H \lhd \Gamma$ a co-amenable normal subgroup which is not of finite index (for example take $H$ to be the kernel of some surjection $\Gamma \twoheadrightarrow \Z$). Then, by Corollary \ref{coamenable exponent}, we have $\delta(H)=\delta(\Gamma)=\delta(G)$ and $H$ is not cocompact in $G$. Next, assume that $G$ is not unimodular. Let $\Delta : G \rightarrow \R^*_+$ be the modular function and let $H=\ker \Delta$. Then $H$ is a coamenable open normal subgroup of $G$ which is not cocompact and we conclude again by Corollary \ref{coamenable exponent}.
\end{proof}


\section{Gaussian actions of groups acting on trees}
For every tree $T$, there exists a unique embedding $\iota : T \rightarrow \cH$ into an affine Hilbert space $\cH$ such that:
\begin{enumerate}[ \rm (i)]
\item $d(x,y)=\| \iota(x)-\iota(y) \|^2$.
\item $\cH$ is the closed affine span of $\iota(T)$.
\end{enumerate}
The proof of the existence can be found in \cite{BHV08}. The uniqueness of $\iota : T \rightarrow \cH$ (up to conjugacy by an affine isometry) is trivial since condition $(\rm i)$ determines all the scalar products $\langle \iota(x)-\iota(z),\iota(y)-\iota(z) \rangle$ for all $x,y,z \in T$. We call $\iota$ the \emph{quadratic embedding} of $T$. From now on we will omit $\iota$ and view directly $T$ as a subset of $\cH$. Observe that if $S \subset T$ is a subtree, then the quadratic embedding of $S$ can be realized by taking its closed affine span $\cK$ inside $\cH$. 

\begin{remark} \label{orientation basis}
Observe that if $\{x_0,\dots, x_n\}$ is a segment in $T$, where $x_i$ is adjacent to $x_{i-1}$ and $x_{i+1}$ for all $ 0 < i < n$, then we have 
$$\| x_i-x_j\|^2+\|x_j-x_k\|^2= |i-j|+|j-k|=|i-k|=\|x_i - x_k\|^2$$ for all $i \leq j \leq k \in \{0,\dots,n\}$. It follows easily from this that $x_i-x_{i+1} \perp x_{j}-x_{j+1}$ whenever $j \neq i$. Therefore, whenever we choose an orientation $\omega$ of $T$, we obtain an orthonormal basis $(\omega_+(e)-\omega_{-}(e))_{e \in E(T)}$ of $\cH^0$. 
\end{remark}

The following Proposition wil be useful later and is also helpful to familiarize with the quadratic embedding $T \subset \cH$.

\begin{prop} \label{intersection space vs intersection tree}
Let $T$ be a tree and $T \subset \cH$ its affine quadratic embedding. Let $(T_i)_{i \in I}$ be a family of subtrees of $T$ and define the (possibly empty) subtree $S=\bigcap_{i \in I} T_i$. Let $\cH_i \subset \cH$ be the affine span of $T_i$ for all $i \in I$ and $\cK$ the affine span of $S$ (with $\cK=\emptyset$ if $S=\emptyset$). Then we have $\bigcap_{i \in I} \cH_i=\cK$ and $\bigcap_{i \in I} \cH_i^0=\cK^0$ (with the convention $\cK^0=\{0\}$ if $\cK=\emptyset$).
\end{prop}
\begin{proof}
First, we start with a general observation. Let $S \subset T$ be a nonempty subtree. Then it is easy to see that for every $x \in T$, there exists a unique $y \in S$ which minimizes the distance $d(x,y)$. We call $y$ the projection of $x$ onto $S$ and we denote it by $P_S(x)$. Let $\cK \subset \cH$ be the affine span of $S$ and let $P_\cK : \cH \rightarrow \cK$ be the orthogonal projection. 
\begin{newclaim} We have $P_\cK(v)=P_S(v)$ for all $v \in T$.
\end{newclaim}
\begin{proof}[Proof of the claim]
It is enough to show that $w-P_S(v)$ is orthogonal to $v-P_S(v)$ for all $w \in S$. This follows from the following equalities
$$ \| w-v\|^2=d(w,v)=d(w,P_S(v))+d(P_S(v),v)=\| w-P_S(v)\|^2+\|P_S(v)-v\|^2.$$ \end{proof}

Now, we can prove the proposition. Let $(T_i)_{i \in I}$ be a family of subtrees of $T$ and define the subtree $S=\bigcap_{i \in I} T_i$. Let $e$ be any oriented edge which is not in $S$. Then there exists $i \in I$ such that $e$ is not in $T_i$. This implies that $e$ is orthogonal to $\cH_i^0$ hence to $\bigcap_{i \in I} \cH_i^0$. Since this holds for every oriented edge which is not in $S$, we conclude that $\bigcap_i \cH_i^0 \subset \cK^0$ and the other inclusion is obvious. 

Next, we show that $\bigcap_{i \in I} \cH_i=\cK$. First, we deal with the case where $I=\{1,2\}$. If $S=T_1 \cap T_2 \neq \emptyset$, then it is easy to see that $P_{T_1} \circ P_{T_2}=P_{S}$. It follows from the claim that $P_{\cH_1} \circ P_{\cH_2} =P_{\cK}$. If $S=\emptyset$, it is easy to see that $P_{T_1} \circ P_{T_2}=P_{\{v_1\}}$ for some $v_1 \in T_1$ and $P_{T_2} \circ P_{T_1} = P_{\{v_2\} }$ for some $v_2 \in T_2$. It follows from the claim that $P_{\cH_1} \circ P_{\cH_2}=P_{\{v_1\}}$ and $P_{\cH_2} \circ P_{\cH_1}=P_{\{v_2\}}$ which implies that $\cH_1 \cap \cH_2=\emptyset$ as we wanted. 

Now, we are left with the case where $(T_i)_{i \in I}$ is a decreasing net of subtrees. If $S \neq \emptyset$, then it is easy to see that $P_{T_i}$ converges pointwise to $P_S$. It follows from  the claim that $P_{\cH_i}$ also converges pointwise to $P_{\cK}$ and we are done. If $S=\emptyset$, then for every $v \in T$, we have $\|v-P_{\cH_i}(v)\|^2=d(v,P_{T_i}(v)) \to \infty$ when $i \to \infty$. This implies that $\bigcap_i \cH_i = \emptyset$.
\end{proof}

Because of the uniqueness of the quadratic embedding of a given locally finite tree $T$, the action of $\Aut(T)$ on $T$ and hence of any locally compact group $G$ acting on $T$ extends uniquely to an affine isometric action $\alpha : G \curvearrowright \cH$. Here are some basic facts which are left to the reader:

\begin{prop}
Let $T$ be a locally finite tree and $G \subset \Aut(T)$ a closed subgroup. Let $\alpha : G \curvearrowright \cH$ be the associated affine isometric action. The following holds:
\begin{enumerate}[ \rm (i)]
\item $\alpha$ is proper.
\item $\alpha^0$ is mixing.
\item $\delta(\alpha)=\delta(G)$.
\end{enumerate}
\end{prop}

\begin{prop} \label{general type spectral gap}
Let $G < \Aut(T)$ be a closed subgroup for some locally finite tree $T$. Let $\alpha : G \curvearrowright \cH$ be the associated affine isometric action.  If $G$ is of general type, then $\alpha^0$ has stable spectral gap.
\end{prop}
\begin{proof}
Up to taking an index $2$ subgroup of $G$, we may assume that $G$ preserves some orientation $\omega$ of the tree $T$. Then by Remark \ref{orientation basis}, we see that $\alpha^0$ is a sum of quasi-regular representations of $G$ on $\rL^2(G/K)$ for some compact subgroups $K < G$. Since $G$ is of general type, it is nonamenable and the conclusion follows.
\end{proof}

\begin{prop} \label{parabolic evanescent}
Let $G < \Aut(T)$ be a parabolic closed subgroup for some locally finite tree $T$. Then the associated affine isometric action $\alpha : T \curvearrowright \cH$ is evanescent. 
\end{prop}
\begin{proof}
Let $\omega \in \partial T$ the unique fixed point of $G$. Then every horoball centered at $\omega$ is a $G$-invariant subtree of $T$ and the intersections of all this horoballs is empty. By taking their closed affine spans in $\cH$, we obtain a decreasing family of $\alpha$-invariant affine subspaces of $\cH$ and we conclude by Proposition \ref{intersection space vs intersection tree}.
\end{proof}

\begin{thm} \label{dissipativity general type}
Let $T$ be a locally finite tree and $G < \Aut(T)$ a closed subgroup of general type. Let $\alpha : G \curvearrowright \cH$ be the associated affine isometric action. Then $t_{\rm diss}(\alpha)=2\sqrt{2\delta(\alpha)}=2\sqrt{2\delta(G)}$.
\end{thm}
\begin{proof}
The inequality $t_{\rm diss}(\alpha) \leq 2\sqrt{2\delta(G)}$ is already proved in Theorem \ref{affine dissipativity}. Let us prove that $t_{\rm diss}(\alpha) \geq 2\sqrt{2\delta(G)}$. Take $t < 2\sqrt{2\delta(G)}$. Then by Corollary \ref{subgroup of large critical exponent}, there exists a closed compactly generated subgroup $G' < G$ such that $2\sqrt{2\delta(G')} > t$. Let $T'$ be the minimal $G'$-invariant subtree of $T$. Then $G'$ acts cocompactly on $T'$ by Proposition \ref{convex hull}. Fix a root $\rho \in T'$. For all $v \in T'$, define the random variable $S_v=\langle \cdot, v-\rho\rangle$ on  $(\widehat{\cH^0},\mu)$. Observe that for all $v \in T'$, we have 
$$S_v=\sum_{e \in E([\rho,v])} \langle \cdot, e \rangle$$ where the sum is over the oriented edges pointing from $\rho$ to $v$ and $\langle \cdot, e \rangle$ are independent standard Gaussian random variables. Thus $(S_v)_{v \in T'}$ is a Gaussian random walk indexed by the tree $T'$. By Example \ref{gaussian random walk}, we then know that
$$ \mathbb{P}\left(  S_v > \frac{t}{2}|v| \text{ for infinitely many } v \in T' \right) > 0$$
where $|v|=d(v,\rho)$ for all $v \in T'$. Since $G' \backslash T'$ is finite, this implies that there exists some $x \in T'$ such that
$$ \mathbb{P}\left(  S_v > \frac{t}{2}|v| \text{ for infinitely many } v \in G' \cdot x \right) > 0 .$$ 
This implies that the series
$$ \sum_{v \in G' \cdot x} \exp\left( - \frac{1}{2}t^2 \| v-\rho\|^2+t\langle \varphi,v-\rho \rangle \right), \; \varphi \in \widehat{\cH^0} $$
diverges with positive probability. We conclude by Proposition \ref{affine dissipativity}, that the action $\widehat{\alpha}^t$ is not dissipative. This shows that $t_{\rm diss}(\alpha) \geq 2 \sqrt{2\delta(G)}$ as we wanted.
\end{proof}


\begin{thm} \label{general type proper weakly mixing}
Let $T$ be a locally finite tree and $G < \Aut(T)$ a closed subgroup of general type. Let $\alpha : G \curvearrowright \cH$ be the quadratic affine isometric action associated to the action of $G$ on $T$. Let $\delta:=\delta(G)=\delta(\alpha)$. Then for all $0 < t < 2\sqrt{2\delta}$, the Gaussian action $\widehat{\alpha}^t$ is weakly mixing of stable type $\III_1$. Moreover, for all $ t \leq  2\sqrt{2\delta}$, the actions $\widehat{\alpha}^t$ are pairwise nonconjugate and for $t$ small enough, $\widehat{\alpha}^t$ is strongly ergodic.
\end{thm}
\begin{proof}
For the first part, we have to show that $\widehat{\alpha}^t$ has a weakly mixing Maharam extension. By Corollary \ref{subgroup of large critical exponent}, we can find a compactly generated closed subgroup of general type $H < G$ such that $t < 2 \sqrt{2\delta(H)}$. By Theorem \ref{dissipativity general type}, $\widehat{\alpha}^t|_H$ is recurrent. Let $S \subset T$ be the minimal $H$-invariant subtree. Let $\cK \subset \cH$ be the closed affine span of $S$. By Lemma \ref{affine recurrent and mixing}, we know that for any probability measure preserving action $\sigma : G \curvearrowright (Y,\nu)$, we have
$$ \rL^\infty( \Mod( \widehat{\cH}^t) \otimes Y)^{\Mod(\widehat{\alpha}^t) \otimes \sigma} \subset \rL^\infty(\Mod(\widehat{\cK}^t) \otimes Y)$$
and the same holds if we replace $\cK$ by $g\cK$ for any $g \in G$ (this is the same thing as replacing $H$ by $gHg^{-1}$). Thus, thanks to Proposition \ref{intersection maharam extension} and Proposition \ref{intersection space vs intersection tree}, we only have to show that $\bigcap_{g \in G} gS=\emptyset$. 

Now, remember that by the second part of Corollary \ref{subgroup of large critical exponent}, we can choose $H < G$ such that $G/H$ is not compact. Then by Lemma \ref{limit set cocompact}, we know that $\Lambda_H$ is a proper subset of $\Lambda_G$. Take $x \in \Lambda_G\setminus  \Lambda_H$. Take a sequence $g_n \in G$ such that $\lim_n g_n v=x$ for some, hence any, $v \in T$. Since $x$ is not in $\Lambda_H=\partial S$, then for every $v \in T$, there exists $n \in \N$ such that $g_n v \notin S$. This shows that $\bigcap_{n \in \N} g_n^{-1} S=\emptyset$ as we wanted.

Finally, the second part of the theorem follows from Proposition \ref{dissipative nonconjugate}, Proposition \ref{general type spectral gap} and Theorem \ref{strongly ergodic}.
\end{proof}

\begin{thm} \label{main trees Gaussian}
Let $\sigma : G \curvearrowright T$ be an action of a locally compact group $G$ on a locally finite tree $T$ such that $G$ does not fix any point or pair of points in $T \cup \partial T$ and let $\alpha : G \curvearrowright \cH$ be the associated affine isometric action. Define $\delta$ as the infinimum of all $s > 0$ such that
$$ \sum_{ y \in G \cdot x} e^{-sd(y,x)} < + \infty$$
for some (hence any) $x \in T$. Then for all $t \geq 0$, the Krieger type of the Gaussian action $\widehat{\alpha}^t$ is given by:
\begin{center}
\renewcommand{\arraystretch}{1.5}
\begin{tabular}{l|l}

	Value of $t$ &  Ergodicity and type of $\widehat{\sigma}$ \\
	\hline
	$t=0$ & Ergodic of type $\II_1$.\\
	$0 < t < 2 \sqrt{2\delta}$ & Ergodic of type $\III_1$.\\
	$t = 2 \sqrt{2\delta}$& $?$\\
	$t > 2 \sqrt{2\delta}$ & Type $\mathrm{I}$ if $\sigma(G)$ is closed and type $\II_\infty$ otherwise.\\

	\end{tabular}
	\renewcommand{\arraystretch}{1}
\end{center}

\end{thm}

\begin{proof}
We may assume that $\sigma$ is faithful. If $\sigma(G)$ is closed in $\Aut(T)$ (i.e.\ if $\sigma$ is proper), then the result follows from Theorem \ref{dissipativity general type} and Theorem \ref{general type proper weakly mixing}. Now, suppose that $\sigma(G)$ is not closed and let $L$ be the closure of $\sigma(G)$ in $\Aut(T)$. Let $\beta : L \curvearrowright \cH$ be the associated affine isometric action, so that $\alpha=\beta \circ \iota$ where $\iota : G \rightarrow L$ is the inclusion map with dense range. We have $\delta=\delta(L)$ and $L$ is of general type. Thus if $0 < t < 2 \sqrt{2 \delta}$, the ergodicity of $\Mod(\widehat{\alpha}^t)=\Mod(\widehat{\beta}^t) \circ \iota$ follows from the ergodicity of $\Mod(\widehat{\beta}^t)$ because $\iota$ has dense range. If $t > 2 \sqrt{2 \delta}$, then $\widehat{\beta}^t$ is dissipative, i.e.\ every ergodic component of $\widehat{\beta}^t$ is of the form $L \curvearrowright L/K$ for some compact subgroup $K < L$. Note that $L$ preserves the pushforward of the Haar measure on $L/K$ which is semifinite but infinite. Then every ergodic component of $\widehat{\alpha}^t=\widehat{\beta}^t \circ \iota$ is of the form $G \curvearrowright L/K$ hence of type $\II_\infty$.
\end{proof}

\begin{thm} \label{spectral radius cayley}
Let $\alpha : \F_d \curvearrowright \cH$ be the affine isometric action associated to the action of the group $\F_d, \: d \geq 2$ on its Cayley tree. Let $\delta=\delta(\alpha)=\log(2d-1)$. Then $\widehat{\alpha}^t$ is nonamenable for all $t < 2\sqrt{\delta}$.
\end{thm}
\begin{proof}
Let $\mu$ be the canonical symmetric probabily measure on $\F_d$. By Theorem \ref{formula spectral radius} and Theorem \ref{computation free group random walk}, the $\mu$-spectral radius of $\widehat{\alpha}^t$ satisfies 
$$\rho_\mu(\widehat{\alpha}^t) > \rho_\mu(G)=\frac{\sqrt{2d-1}}{d}$$
for all $t < 2 \sqrt{\log(2d-1)}$. We conclude by Proposition \ref{weakly contained amenable action}.
\end{proof}

\begin{thm}
Let $\alpha : \F_d \curvearrowright \cH$ be the affine isometric action associated to the action of the group $\F_d, \: d \geq 2$ on its Cayley tree. Let $\delta=\delta(\alpha)=\log(2d-1)$. Then $\widehat{\alpha}^t$ is dissipative for $t=2 \sqrt{2\delta}$.
\end{thm}
\begin{proof}
This follows from Theorem \ref{dirichlet domain} and Theorem \ref{speed regular tree} which precisely says that the Gaussian Dirichlet domain has positive measure at $t=2 \sqrt{2\delta}$.
\end{proof}

\section{Nonsingular Bernoulli actions of groups acting on trees} \label{section bernoulli trees}
Let $T$ be a locally finite tree. Let $\Omega(T)$ be the compact space of all possible orientations of $T$. For each $x,y \in T$, define a continuous function $c(x,y)$ on $\Omega(T)$ by
$$ c(x,y)(\omega)=n_e-m_e$$
where $n_e$ (resp.\ $m_e$) is the number of edges of $[x,y]$ oriented towards $x$ (resp.\ $y$) by the orientation $\omega \in \Omega$. Observe that we have the cocycle relation $c(x,y)+c(y,z)=c(x,z)$ for all $x,y,z \in T$.

Fix $p \in ]0,1[$. For every $x \in T$, we define a random orientation of $T$ as follows: each edge $e \in E(T)$ independently is oriented towards $x$ with probability $p$ and oriented away from $x$ with probability $1-p$. This defines a probability measure $\mu_x^p$ on $\Omega(T)$. Take another $y \in T$. Orienting an edge $e$ towards $x$ or towards $y$ is the same thing, unless $e$ lies in $[x,y]$. Since $[x,y]$ is finite, this implies that $\mu_y^p$ is equivalent to $\mu_x^p$ and the Radon-Nikodym derivative is given by
$$ \frac{\rd \mu_y^p}{\rd \mu_x^p}=\left( \frac{1-p}{p} \right)^{-c(y,x)}.$$
In particular, we see that the probability measures $(\mu_x^p)_{x \in T}$ are all in the same $\lambda$-measure class $m$ where $\lambda=\min \left( \frac{1-p}{p},\frac{p}{1-p} \right) \in ]0,1]$. We denote by $\Omega^p(T)=(\Omega(T),m)$ the associated $\lambda$-nonsingular space (see the appendix \ref{lambda nonsingular}).

Observe that $\Aut(T)$ acts naturally on $\Omega$ by homeomorphisms. Since the family of equivalent probability measures $(\mu_x^p)_{x \in T}$ is defined naturally in terms of the tree structure, it is obvious that for every $g \in \Aut(T)$, the induced homeomorphism $\widehat{g}$ of $\Omega$ satisfies $\widehat{g}_*\mu_x^p=\mu_{gx}^p$ for all $x \in T$. In particular, $\widehat{g}$ preserves the $\lambda$-measure class of $(\mu_x^p)_{x \in T}$, hence it induces a $\lambda$-nonsingular automorphism of $\Omega^p(T)$. Now, let $\sigma : G \curvearrowright \Aut(T)$ be an action of a locally compact group $G$ on the tree $T$. It induces a  $\lambda$-nonsingular action $\widehat{\sigma}^p: G \curvearrowright \Omega^p(T)$. Our goal is to study these nonsingular actions $\widehat{\sigma}^p$ when $p$ varies in $]0,1[$.

\begin{rem}
We observe that the nonsingular actions $\widehat{\sigma}^p : G \curvearrowright \Omega^p(T)$ can be identified (in a non-natural way) with \emph{nonsingular generalized Bernoulli shifts}. Indeed, assume that $G$ acts on $T$ without edge inversion, or equivalently, assume that there exists an orientation $\omega_0 \in \Omega$ which is fixed by $G$ (note that we can always assume this up to restricting $\sigma$ to an index $2$ subgroup of $G$). Define a homeomorphism $\theta : \Omega \rightarrow \{0,1\}^{E(T)}$ which encodes each orientation $\omega \in \Omega$ by assigning $0$ or $1$ to each edge $e \in E(T)$ according to whether $\omega$ differs or coincides with $\omega_0$ on $e$. Since $\omega_0$ is fixed by $G$, the map $\theta$ is $G$-equivariant where $G$ acts by generalized Bernoulli shifts on $\{0,1\}^{E(T)}$. Moreover, for every $p \in ]0,1[$ and every $x \in T$, the probability measure $\theta_* \mu_x^p$ on $\{0,1\}^{E(T)}$ is clearly an infinite product of probability  measures of the form $p\delta_0+(1-p)\delta_1$ or $(1-p)\delta_0+p\delta_1$ depending on the edge in $E(T)$. This shows that $\widehat{\sigma}^p : G \curvearrowright \Omega^p(T)$ is conjugate to a nonsingular generalized Bernoulli shift. For example, if $\sigma$ is the action of a free group on its Cayley tree, then this is really a nonsingular Bernoulli shift exactly as the ones studied in \cite{VW18} or \cite{BKV19}.
\end{rem}

\begin{rem} \label{subtree bernoulli} Let $S \subset T$ be a subtree. Then there is a natural homeomorphism $\theta : \Omega(T) \rightarrow \{0,1\}^{E(T) \setminus E(S)} \times \Omega(S)$. Indeed, on $E(T) \setminus E(S)$ an orientation $\omega \in \Omega$ can be encoded by assigning to each edge $e \in E(T) \setminus E(S)$ the value $0$ if $\omega$ orients it towards $S$ and $1$ otherwise. If one also knows the restriction $\omega|_S \in \Omega(S)$, than one can recover $\omega$. This defines our map $\theta$. Observe that for all $p \in ]0,1[$, we have $\theta_* \mu_x^p= \mu^p \otimes \mu_x^p$ for any $x \in S$, where $\mu^p$ is the usual Bernoulli product measure on $\{0,1\}^{E(T) \setminus E(S)}$. We will thus omit $\theta$ and simply make the identification
$$ \Omega^p(T)=(\{0,1\}^{E(T) \setminus E(S)},\mu^p) \otimes \Omega^p(S).$$
Now, suppose that we have a group action $\sigma : G \curvearrowright T$ which leaves $S$ globally invariant. Then, under the above identification, the nonsingular action $\widehat{\sigma}^p : G \curvearrowright \Omega^p(T)$ is identified with the diagonal action $$\beta^p \otimes \widehat{\sigma|_S} : G \curvearrowright (\{0,1\}^{E(T) \setminus E(S)},\mu^p) \otimes \Omega^p(S)$$ where $\beta^p : G \curvearrowright (\{0,1\}^{E(T) \setminus E(S)},\mu^p)$ is the generalized Bernoulli shift comming from the action of $G$ on $E(T) \setminus E(S)$.
\end{rem}

\begin{thm} \label{dissipativity general type bernoulli}
Let $T$ be a locally finite tree and $G \subset \Aut(T)$ a closed subgroup of general type and let $\delta=\delta(G)$. Take $p \in ]0,1[$ and let $\widehat{\sigma}^p : G \curvearrowright \Omega^p(T)$ be the nonsingular action associated to the action $\sigma: G \curvearrowright T$. Then $\widehat{\sigma}^p$ is recurrent if $2 \sqrt{p(1-p)} > e^{-\delta}$ and dissipative if $2 \sqrt{p(1-p)} < e^{-\delta}$.
\end{thm}
\begin{proof}
Assume that $\widehat{\sigma}^p$ is recurrent. Take any $x \in T$. Then, assuming that the Haar measure of $G_x=\{ g \in G \mid gx=x\}$ is $1$, we have
$$\int_G \frac{\rd \mu_{gx}^p}{\rd \mu_x^p} \rd g=   \sum_{y \in G \cdot x} \frac{\rd \mu_y^p}{\rd \mu_x^p}  = \sum_{y \in G \cdot x} \left(\frac{1-p}{p}\right)^{-c(y,x)} = + \infty.$$
Taking the square roots, we get
	\[\sum_{y \in G \cdot x} \left(\frac{1-p}{p}\right)^{-\frac{1}{2}c(y,x)} = +\infty.\]
	A computation shows that 
	$$\int_{\Omega^p} \left(\frac{1-p}{p}\right)^{-\frac{1}{2}c(y,x)} \rd \mu^p_x = \left( 2\sqrt{p(1-p)} \right)^{d(y,x)}.$$
Therefore, we get
	\[\sum_{y \in G \cdot x} \left( 2\sqrt{p(1-p)} \right)^{d(y,x)}= + \infty.\]
	This implies that $2\sqrt{p(1-p)} \geq e^{-\delta}$.
	
Conversely, assume that $2\sqrt{p(1-p)} > e^{-\delta}$. Then by Corollary \ref{subgroup of large critical exponent}, we can find a compactly generated closed subgroup $G' \subset G$ such that $2\sqrt{p(1-p)} > e^{-\delta'}$ where $\delta'=\delta(G')$. Let $T'$ be the minimal $G'$-invariant subtree of $T$. Then $G'$ acts cocompactly on $T'$ and $\delta'=\dim_H \partial T'$. Fix a root $\rho \in T'$. Define the random variables $S_v=c(v,\rho)$ with respect to the probability measure $\mu^p_\rho$. Observe that the random process $(S_v)_{v \in T'}$ is a tree-indexed Bernoulli random walk of parameter $p$. It follows from Example \ref{bernoulli random walk} that almost surely, we have $S_v < 0$ for infinitely many $v \in T'$ and $S_v > 0$ for infinitely many $v \in T'$. Therefore, we conclude that almost surely, we have
$$ \sum_{ v \in T'} \frac{\rd \mu_v^p}{\rd \mu_\rho^p}=\sum_{ v \in T'} \left( \frac{1-p}{p} \right)^{-S_v}=+\infty.$$
Since $G'\backslash T'$ is finite, this implies easily that
$$ \sum_{ v \in G' \cdot \rho } \frac{\rd \mu_v^p}{\rd \mu_\rho^p}=+\infty$$
with positive probability, hence almost surely by Kolmogorov's $0$-$1$ law. This shows that $\widehat{\sigma}^p$ is recurrent.
\end{proof}

%

	Since the image of the Radon--Nikodym derivative $\lambda^{S^\rho_{g \rho}}$ lies in the multiplicative subgroup $\lambda^{\mathbb Z} \subset \mathbb{R}^\times$, we may take the $\lambda$-modular bundle $\Mod_{\lambda}(X)$, which is a bundle over $X$ whose fiber is  $\lambda^{\mathbb Z}$ instead of the usual modular bundle $\mathrm{Mod}(X)$.

\begin{thm} \label{general type proper weakly mixing bernoulli}
Let $T$ be a locally finite tree and $G \subset \Aut(T)$ a closed subgroup of general type and let $\delta=\delta(G)$. Take $p \in ]0,1[$ and let $\widehat{\sigma}^p : G \curvearrowright \Omega^p(T)$ be the nonsingular action associated to the action $\sigma: G \curvearrowright T$. If $p \neq \frac{1}{2}$ and $2 \sqrt{p(1-p)} > e^{-\delta}$ then $\widehat{\sigma}^p$ is weakly mixing of stable type $\III_\lambda$ where $\lambda = \min\left(\frac{1-p}{p}, \frac{p}{1-p} \right)$.
\end{thm}	
	
	\begin{proof}
We have to show that $\widehat{\sigma}^p$ has a weakly mixing $\lambda$-Maharam extension. By Corollary \ref{subgroup of large critical exponent}, we can find a closed comapctly generated subgroup of general type $H \subset G$ such that $2 \sqrt{p(1-p)} > e^{-\delta(H)}$ such that $G/H$ is not compact. By Theorem \ref{dissipativity general type}, $\widehat{\sigma}^p|_H$ is recurrent. Let $S \subset T$ be the minimal $H$-invariant subtree. Then we have a $H$-equivariant identification
$$\Mod_\lambda(\Omega^p(T)) = \Mod_\lambda(\Omega^p(S)) \otimes (\{0,1\}^{E(T) \setminus E(S)} , \mu^p )$$
where $H$ acts diagonally on the right hand side (see Remark \ref{subtree bernoulli}). The pmp generalized Bernoulli action of $H$ on  $(\{0,1\}^{E(T) \setminus E(S)} , \mu^p )$ is mixing (because the edge stabilizers in $H$ are compact since $H$ is closed in $\Aut(T)$) and the action of $H$ on $\Mod_\lambda(\Omega^p(S))$ is recurrent by Theorem \ref{dissipativity general type bernoulli}. Therefore, for every pmp ergodic action $\rho : G \curvearrowright (Y,\nu)$, we know by Theorem \ref{schmidt} that every $H$-invariant function in $\rL^\infty(\Mod_\lambda(\Omega^p(T)) \otimes Y )$ lives in $\rL^\infty(\Mod_\lambda(\Omega^p(S)) \otimes Y)$. By applying the same argument to $gHg^{-1}$ and the  $gHg^{-1}$-invariant subtree $gS$, we obtain
$$ \rL^\infty(\Mod_\lambda(\Omega^p(T)) \otimes Y)^G \subset \bigcap_{g \in G} \rL^\infty(\Mod_\lambda(\Omega^p(gS)) \otimes Y).$$
Let us show that $\bigcap_{g \in G} gS=\emptyset$. By Lemma \ref{limit set cocompact}, we know that $\Lambda_H$ is a proper subset of $\Lambda_G$. Take $x \in \Lambda_G\setminus  \Lambda_H$. Take a sequence $g_n \in G$ such that $\lim_n g_n v=x$ for some, hence any, $v \in T$. Since $x$ is not in $\Lambda_H=\partial S$, then for every $v \in T$, there exists $n \in \N$ such that $g_n v \notin S$. This shows that $\bigcap_{n \in \N} g_n^{-1} S=\emptyset$. Finally, let us conclude from this that 
$$\bigcap_{g \in G} \rL^\infty(\Mod_\lambda(\Omega^p(gS)) \otimes Y)=\rL^\infty(Y).$$
In order to show this, it suffices to note that if $e \in E(T) \setminus E(S)$, then every function in $\rL^\infty(\Mod_\lambda(\Omega^p(S))$ is invariant by the $\Mod_\lambda(\tau_e)$ where $\tau_e \in \Aut(\Omega^p(T))$ is the involution which flips the orientation of the edge $e$. Since $\bigcap_{g \in G} gS=\emptyset$, we see that every function $\bigcap_{g \in G} \rL^\infty(\Mod_\lambda(\Omega^p(gS)) \otimes Y)$ is invariant by $\Mod_\lambda(\tau_e) \otimes \id$ for all $e \in E(T)$. It is easy (we leave it to the reader) to see that the involutions $(\tau_e)_{e \in E(T)}$ generate a group which acts ergodically on $\Mod_\lambda(\Omega^p(T))$ and the conclusion follows.
\end{proof}

\begin{thm} \label{main trees Bernoulli}
Let $\sigma : G \curvearrowright T$ be an action of a locally compact group $G$ on a locally finite tree $T$ such that $G$ does not fix any point or pair of points in $T \cup \partial T$. Define the Poincar\'e exponent $\delta$ as the infinimum of all $s > 0$ such that
$$ \sum_{ y \in G \cdot x} e^{-sd(y,x)} < + \infty$$
for some (hence any) $x \in T$. Then for all $p \in ]0,1[$, the type of $\widehat{\sigma}^p : G \curvearrowright \Omega^p(T)$ is given by:

\begin{center}
\renewcommand{\arraystretch}{1.5}
\begin{tabular}{l|l}

	Value of $p$ &  Ergodicity and type of $\widehat{\sigma}$ \\
	\hline
	$p=1/2$ &  Ergodic of type $\II_1$. \\
	$e^{-\delta} < 2\sqrt{p(1-p)} < 1$ & Ergodic of type $\III_\lambda$ where $\lambda = \min \left(\frac{1-p}{p},\frac{p}{1-p} \right)$.\\
	$2\sqrt{p(1-p)} =e^{-\delta}$ & $?$\\
	$2\sqrt{p(1-p)}  < e^{- \delta}$ & Type $\mathrm{I}$ if $\sigma(G)$ is closed and type $\II_\infty$ otherwise.\\

	\end{tabular}
	\renewcommand{\arraystretch}{1}
\end{center}
\end{thm}
\begin{proof}
We may assume that $\sigma$ is faithful. If $\sigma(G)$ is closed in $\Aut(T)$ (i.e.\ if $\sigma$ is proper), then the result follows from Theorem \ref{dissipativity general type bernoulli} and Theorem \ref{general type proper weakly mixing bernoulli}. Now, suppose that $\sigma(G)$ is not closed and let $L$ be the closure of $\sigma(G)$ in $\Aut(T)$. Let $\beta : L \curvearrowright T$ be the associated action, so that $\sigma=\beta \circ \iota$ where $\iota : G \rightarrow L$ is the inclusion map with dense range. We have $\delta=\delta(L)$ and $L$ is of general type. Thus if $e^{-\delta} < 2\sqrt{p(1-p)} < 1$, the ergodicity of $\Mod_\lambda(\widehat{\sigma}^p)=\Mod_\lambda(\widehat{\beta}^p) \circ \iota$ follows from the ergodicity of $\Mod_\lambda(\widehat{\beta}^p)$ because $\iota$ has dense range. If $ 2\sqrt{p(1-p)} < e^{-\delta}$, then $\widehat{\beta}^p$ is dissipative, i.e.\ every ergodic component of $\widehat{\beta}^p$ is of the form $L \curvearrowright L/K$ for some compact subgroup $K < L$. Note that $L$ preserves the pushforward of the Haar measure on $L/K$ which is semifinite but infinite. Then every ergodic component of $\widehat{\sigma}^p=\widehat{\beta}^p \circ \iota$ is of the form $G \curvearrowright L/K$ hence of type $\II_\infty$.
\end{proof}

\appendix

\section{Nonsingular measure theory}

\subsection{Basic terminology and notations}
\begin{df}
The \emph{nonsingular category} is the category whose objects are triples $(X,\mathcal{B},m)$ where $(X,\mathcal{B})$ is a standard borel space and $m$ is an equivalence class of $\sigma$-finite measures on $(X,\mathcal{B})$. A morphism in this category is an equivalence class of nonsingular borel maps $\psi : (X,\mathcal{B},m) \rightarrow (X',\mathcal{B}',m')$ for the relation of equality almost everywhere. 
\end{df}
When we consider an object of this category, we will simply denote it $X$ and omit both the borel $\sigma$-algebra and the measure class. We will refer to $X$ as a \emph{nonsingular space}. Similarly, a morphism $\psi :X \rightarrow Y$ will be called a \emph{nonsingular map} and equalitites between nonsingular maps are always implicitely assumed to hold almost everywhere. The set of all equivalence classes of measurable subsets of $X$ is denoted $\mathfrak{P}(X)$. We will use the usual set-theoretic notations $A \cap B$, $A \subset B$ etc. when dealing with elements of $\mathfrak{P}(X)$ even though we always neglect null-sets. The null-set (the equivalence class consisting of all measurable subsets with zero measure) will be denoted $0$. When we consider a measure $\mu$ on a nonsingular space $X$, we always mean a measure which is absolutely continuous with respect to the measure class of $X$. We will say that $\mu$ is \emph{faithful} if $\mu$ belongs to the measure class of $X$ (equivalently this means that $\mu(A) \neq 0$ for all $0 \neq A \subset X$). Since we will use the letter $\sigma$ to denote nonsingular action, we will often use the terminology \emph{semifinite} for measures instead of $\sigma$-finite.

\subsubsection{Measurable maps into Polish spaces}
If $M$ is a Polish space, we denote by $\rL^0(X,M)$ the space of all measurable maps from $X$ to $M$ modulo equality almost everywhere and we equip it with the topology of convergence in measure (which only depends on the measure class of $X$). Then $\rL^0(X,M)$ becomes itself a Polish space. We will often use the shorter notation $\rL^0(X)$ when $M=\R$ or $M=\C$. We also denote by $\rL^\infty(X) \subset \rL^0(X)$ the subalgebra of (essentially) bounded functions. Note that $\mathfrak{P}(X)$ can be identified with $\rL^0(X,\{0,1\})$, so that it is also a Polish space.

\subsubsection{Quotients}
Suppose that we have a nonsingular map $\psi : X \rightarrow Y$ such that $\psi^{-1}(A)=0$ implies $A=0$ for all $A \in \mathfrak{P}(Y)$. Then we say that $\psi$ is a \emph{quotient map} and that $Y$ is a \emph{quotient} of $X$. Under this assumption, the map $\psi^* : \rL^0(Y,M) \rightarrow \rL^0(X,M)$ is injective for any Polish space $M$. Therefore, we will often view $\rL^0(Y,M)$ as a subset of $\rL^0(X,M)$. In particular $\rL^\infty(Y)$ will be viewed as a subalgebra of $\rL^\infty(X)$. Note that $\rL^\infty(Y)$ is closed in $\rL^\infty(X)$ for the topology of convergence in measure. Conversely, if $A \subset \rL^\infty(X)$ is a subalgebra which is closed for the topology of convergence in measure, then there exists a quotient $Y$ of $X$ such that $A=\rL^\infty(Y)$. Thus, we have a one-to-one correspondence between quotients of $X$ and closed subalgebras of $\rL^\infty(X)$.

\subsubsection{Products of two nonsingular spaces}
 If $X$ and $Y$ are two nonsingular spaces, then we can define a nonsingular space $X \otimes Y$ whose measure class is the product of the measure classes of $X$ and $Y$. If $\mu$ and $\nu$ are two measures on $X$ and $Y$ respectively, we denote by $\mu \otimes \nu$ the product measure on $X \otimes Y$. One has a natural identification $\rL^0(X \otimes Y,M)=\rL^0(X,\rL^0(Y,M))$ for any Polish space $M$.

\subsubsection{Nonsingular actions}
We denote by $\Aut(X)$ the group of all nonsingular automorphisms of $X$ and we define a topology on $\Aut(X)$ by declaring that a sequence $\theta_n \in \Aut(X)$ converges to the identity if and only if $\| \mu \circ \theta_n - \mu \|_1 \to 0$ for every probability measure $\mu$ on $X$. With this topology, $\Aut(X)$ is a Polish group. A \emph{nonsingular action} $\sigma : G \curvearrowright X$ of a locally compact group $G$ is simply a continuous group homomorphism $\sigma : G \rightarrow \Aut(X)$.

\subsection{The modular bundle} \label{the modular bundle}
Here we will see how to associate to every nonsingular space, a canonical bundle equipped with a measure scaling flow that we call the \emph{modular bundle}. This construction is usually introduced only in the presence of a nonsingular action and under the name of \emph{Maharam extension} but Proposition \ref{intersection maharam extension}, for example, shows that it can be useful to study it on its own. We particularly emphasize its canonical nature.

\begin{df}
A \emph{measure scaling flow} is a triple $(X,\tau,\theta)$ where $(X,\tau)$ is a $\sigma$-finite measure space and $\theta  : \R^*_+ \curvearrowright X$ is a nonsingular action such that $\tau \circ \theta_\lambda=\lambda \tau$ for all $\lambda \in \R^*_+$.
\end{df}

The structure of measure scaling flows turns out to be trivial, but not in a \emph{canonical} way.

\begin{thm}
Let $(X,\tau,\theta)$ be a measure scaling flow. Let $X_0=\R^*_+ \backslash X$. Then the following holds:
\begin{enumerate}[ \rm (i)]
\item The identity map $\id : X_0 \rightarrow X_0$ can be lifted to a nonsingular isomorphism $\pi : X \rightarrow X_0 \otimes \R^*_+$ which conjugates $\theta$ with $\id \otimes \rho$ where $\rho : \R^*_+ \curvearrowright \R^*_+$ is the action by multiplication.
\item For any $\pi$ as in $(\rm i)$, there exists a faithful semifinite measure $\mu$ on $X_0$ such that $\pi_* \tau=\mu \otimes \mathrm{d} \lambda$. Conversely, for any faithful semifinite measure $\mu$ on $X_0$, there exists a unique $\pi$ as in $(\rm i)$ such that $\pi_* \tau=\mu \otimes \mathrm{d} \lambda$.
\end{enumerate}
\end{thm}
\begin{proof}
$(\rm i)$. We claim that $\theta$ is dissipative (see Section \ref{appendix dissipative}). Fix $0 < \lambda < 1$ and let $\alpha=\theta_\lambda$. By Proposition \ref{lattice dissipative} It is enough to show that $\alpha$ is dissipative (as an action of $\Z$). Take $A \subset X$ such that $0 < \tau(A) < +\infty$. Let $B=\bigvee_{n \in \N} \alpha^n(A)$. Then $\tau(B) \leq \sum_n \lambda^n \tau(A) < +\infty$. Observe that $\alpha(B) \subset B$ and $\tau(\alpha(B))=\lambda \tau(B)$. Write $C=B \setminus \alpha(B)$. Then $0 < \tau(C) < +\infty$ and the subsets $\alpha^n(C)$ for $n \in \Z$ are pairwise disjoint. This implies that $\sum_n \alpha^n(1_C) < +\infty$. Note that $C \subset A$ and the choice of $A$ is arbitrary. We conclude that the fibered counting measure of $\alpha$ is semifinite, i.e.\ $\alpha$ is dissipative.  Thus we proved that $\theta$ is dissipative. Now, since $\R^*_+$ has no compact subgroups, item $(\rm i)$ follows from Theorem \ref{dissipative trivialization}.

$(\rm ii)$. Take $\pi$ as in $(\rm i)$. Let $\nu=(1 \otimes h) \pi_* \tau$ where $h : \R^*_+ \rightarrow \R^*_+$ is defined by $h(\lambda)=\lambda^{-1}$. Then $\nu$ is invariant under the action $\id \otimes \rho$ of $\R^*_+$. Thus $\nu=\mu \otimes m$ where $m$ is the Haar measure of $\R^*_+$ which is given by $\rd m= h\rd \lambda$. It follows that $\pi_* \tau = \mu \otimes \rd \lambda$. To check the uniqueness of $\pi$ for a given faithful semifinite measure $\mu$ on $X_0$, it is enough to check that the only nonsingular automorphism pf $X_0 \otimes \R^*_+$ which commutes with $\id \otimes \rho$ and preserves the measure $\mu \otimes \rd \lambda$ is the identity map.
\end{proof}

We call a map $\pi$ as in $(\rm i)$ a \emph{trivialization} of $(X,\tau,\theta)$. Item $(\rm ii)$ says that such trivializations are in one-to-one correspondence with faithful semifinite measures on $X_0=\R^*_+ \backslash X$. For any such measure $\mu$, we will denote by $\pi_\mu$ the corresponding trivialization of $X$. Observe that we have the following formula for the transition map between two trivializations:
$$ \pi_{\mu} \circ \pi_{\nu}^{-1}(x,\lambda)=\left( x,\lambda\frac{\rd \nu}{\rd \mu}(x) \right), \quad (x,\lambda) \in X_0 \otimes \R^*_+.$$

Clearly if $\psi : (X,\tau,\theta) \rightarrow (X',\tau',\theta')$ is a measure preserving isomorphism conjugating two measure scaling flows, then it induces a nonsingular isomorphism $$\psi_0 : X_0=\R^*_+ \backslash X \rightarrow X_0'=\R^*_+ \backslash X'.$$
Conversely, any nonsingular isomorphism $\psi_0 : X_0  \rightarrow X_0'$ can be lifted in a unique way to a measure preserving conjugacy $\psi : (X,\tau,\theta) \rightarrow (X',\tau',\theta')$. Indeed, by using trivializations, the map $\psi$ is given by
$$ \pi_\nu \circ \psi \circ \pi_\mu^{-1} = \psi_0 \otimes \id : X_0 \otimes \R^*_+ \rightarrow X_0' \otimes \R^*_+$$
where $\mu$ is any faithful semifinite measure on $X_0$ and $\nu =(\psi_0)_* \mu$.

In conclusion, we obtain an equivalence of categories from the category of mesure scaling flows and measure preserving conjugacies to the category of nonsingular spaces and nonsingular isomorphisms. We choose once and for all an inverse functor which associates to every nonsingular space $X$ a measure scaling flow, that we denote by $(\Mod(X),\tau,\theta)$, such that $\R^*_+ \backslash \Mod(X) = X$. We will call it the \emph{modular bundle} of $X$. Since we identify $X$ with the quotient of $\Mod(X)$ by the action $\theta$, we view $\rL^\infty(X)$ as the fixed point subalgebra of $\rL^\infty(\Mod(X))$ under the action $\theta$. Intuitively, it is helpful to think of $\Mod(X)$ as a canonical line bundle over $X$ whose sections are exactly the faithful semifinite measures of $X$. Each choice of such a measure $\mu$ provides a trivialization of the bundle via the map $\pi_\mu$. For example, if $X$ is a smooth manifold, then $\Mod(X)$ can be identifed with the canonical bundle of (positive) $1$-densities over $X$.

\subsubsection{Modular bundle of product spaces}
Let $X$ and $Y$ be two nonsingular spaces and let $(\Mod(X),\tau_X,\theta^X)$ and $(\Mod(Y),\tau^Y,\theta^Y)$ be their modular bundles. Then, we have the following identification
 $$ \Mod( X \otimes Y)  \cong \Mod(X) \otimes_{\R^*_+} \Mod(Y) $$
 where $\Mod(X) \otimes_{\R^*_+} \Mod(Y)$ is the quotient of $\Mod(X) \otimes \Mod(Y)$ by the flow $$\beta : \R^*_+ \ni \lambda \mapsto \theta^X_{\lambda} \otimes \theta^Y_{\lambda^{-1}} \in \Aut(\Mod(X) \otimes \Mod(Y)).$$
By construction, the two flows $\theta^X \otimes \id$ and $\id \otimes \theta^Y$ induce the same flow on $\Mod(X) \otimes_{\R^*_+} \Mod(Y)$. This flow is identified with $\theta^{X \otimes Y}$. Observe also that $\tau^X \otimes \tau^Y$ is $\beta$-invariant and $\beta$ is dissipative. Therefore, there exists a unique faithful semifinite measure $\tau$ on $\Mod(X) \otimes_{\R^*_+} \Mod(Y)$ such that 
$$\tau^X \otimes \tau^Y=\tau \circ T $$
where $T$ is the fibered Haar measure of $\beta$ (see Section \ref{appendix dissipative}). The measure $\tau$ is identified with $\tau^{X \otimes Y}$.

 We now define partial trivializations of $\Mod(X \otimes Y)$. Choose a faithful semifinite measure $\mu$ on $X$. Observe that $(X \otimes \Mod(Y),\mu \otimes \tau^Y, \id \otimes \theta^Y)$ is a measure scaling flow. Therefore, there is a unique way to lift the identity map $X \otimes Y \rightarrow X \otimes Y$ to a measure preserving conjugacy 
 $$\kappa_\mu : (\Mod(X \otimes Y), \tau^{X \otimes Y}, \theta^{X \otimes Y}) \rightarrow  (X \otimes \Mod(Y),\mu \otimes \tau^Y,\id \otimes \theta^Y).$$
In terms of trivializations, we have $\pi_{\mu \otimes \nu} =(\id \otimes \pi_{\nu}) \circ \kappa_\mu $ for any faithful semifinite measure $\nu$ on $Y$. We call $\kappa_\mu$ a \emph{partial trivialization} over $Y$.

\subsection{Densities and canonical $\rL^p$-spaces} \label{densities Lp spaces} We explain a construction of canonical $\rL^p$-spaces which is due to Haagerup. Let $X$ be a nonsingular space. Take $z \in \C$. A \emph{density of degree $z$} over $X$ is a function $\zeta \in \rL^0(\Mod(X),\C)$ such that $\theta_\lambda(\zeta)=\lambda^z\zeta$ for all $\lambda \in \R^*_+$. We denote by $\Lambda^z(X)$ the set of all densities of degree $z$. When $z=s \in \R$, we denote by $\Lambda^s_+(X)$ the space of all non-negative $s$-densities. Note that $\Lambda^{z}(X)\Lambda^{w}(X) \subset  \Lambda^{z+w}(X)$ for all $z,w \in \C$. In particular, $\Lambda^{z}(X)$ is a module over $\Lambda^0(X)=\rL^0(X)$ for all $z \in \C$.

Let $\zeta \in \Lambda^{1}_+(X)$. Then the semifinite measure $\zeta \tau$ on $\Mod(X)$ is $\theta$-invariant. Therefore, by Corollary \ref{dissipative invariant measure}, there exists a unique semifinite measure $\mu$ on $X$ such that $\zeta\tau = \mu \circ T$ where $T$ is the fibered Haar measure of $\theta$. Conversely, if $\mu$ is a semifinite measure on $X$, then 
$$\zeta=\frac{\rd (\mu \circ T)}{\rd \tau} \in \Lambda^1_+(X).$$
We will therefore identify $\Lambda^{1}_+(X)$ with the set of all semifinite measures on $X$. For every faithful semifinite measure $\mu$ on $X$, by viewing $\mu$ as an element of $\Lambda^1_+(X)$ hence as a positive function on $\Mod(X)$, we can define $\mu^z \in \Lambda^z(X)$ and the map $\rL^0(X) \ni f \mapsto f\mu^z \in \Lambda^z(X)$ is a linear bijection.

Take $p \in [1,+\infty[$ and $\zeta \in \Lambda^{1/p}(X)$. Then $|\zeta|^p \in \Lambda^1_+(X)$. Therefore, viewing $|\zeta|^p$ as a measure on $X$, we can define
$$ \| \zeta \|_p=\left( \int_X |\zeta|^p \right)^{1/p}.$$
The space $\rL^p(X)=\{ \zeta \in \Lambda^{1/p}(X) \mid \| \zeta \|_p < +\infty \}$ is a Banach space with respect to the norm $\| \cdot \|_p$. We call it the canonical $\rL^p$-space of $X$. Note that for every finite measure $\mu$ on $X$, we can define $\mu^{1/p} \in \rL^p(X)$ and we have $\nu^{1/p}=(\frac{\rd \nu}{\rd \mu})^{1/p} \mu^{1/p}$. Note that 
$$\| f \mu^{1/p} \|_p=\left( \int_X |f|^p \rd \mu \right)^{1/p}$$
for all $f \in \rL^\infty(X)$. Therefore, if $\mu$ is faithful, the map $f \mapsto f \mu^{1/p}$ extends to an isometric isomorphism from the usual $\rL^p$-space $\rL^p(X,\mu)$ to the canonical one $\rL^p(X)$.

Note that $\rL^1(X)$ is naturally identified with the space of all signed (or complex valued) finite measures on $X$. 

Note that we have a natural duality between $\rL^p(X)$ and $\rL^q(X)$ when $\frac{1}{p}+ \frac{1}{q}=1$ because $\rL^p(X) \rL^q(X) \subset \rL^1(X)$. The duality is then defined by integration $\langle \zeta, \zeta'\rangle =\int_X \zeta \zeta'$. In particular, $\rL^2(X)$ is a Hilbert space and we have the formula
$$ \langle \mu^{1/2},\nu^{1/2} \rangle = \int \left( \frac{\rd \nu}{\rd \mu} \right)^{1/2}  \rd \mu=\int \left( \frac{\rd \mu}{\rd \nu} \right)^{1/2} \rd \nu.$$

\subsection{The Maharam extension and Krieger invariants}
Thanks to its functorial properties, the modular bundle plays an important role in the study of nonsingular actions. Indeed, every automorphism $\alpha \in \Aut(X)$ can be lifted in a unique way to an automorphism $\Mod(\alpha) \in \Aut(\Mod(X))$ which preserves the measure $\tau$ and commutes with $\theta$. If we choose a trivialization $\pi_\mu$, then we have the following formula
$$ \pi_{\mu} \circ \Mod(\alpha) \circ \pi_{\mu}^{-1}(x,\lambda)=\left( \alpha(x),\lambda \frac{\rd \alpha_* \mu }{\rd \mu}(x) \right).$$
In particular, any nonsingular action $\sigma : G \curvearrowright X$ can be lifted to a nonsingular action $\Mod(\sigma) : G \curvearrowright \Mod(X)$ called the \emph{Maharam extension} of $\sigma$. The Maharam extension $\Mod(\sigma)$ preserves the measure $\tau$ and commutes with $\theta$.

\subsubsection{Krieger flow and Krieger type}
Let $\sigma : G \curvearrowright X$ be a nonsingular action. The flow $\theta : \R^*_+ \curvearrowright \Mod(X)$ descends to a flow $\widetilde{\theta}_{\sigma} : \R^*_+ \curvearrowright G \backslash \Mod(X)$ called the \emph{Krieger flow} of $\sigma$. Note that $\widetilde{\theta}_{\sigma}$ is ergodic if and only if $\sigma$ is ergodic because 
$$\R^*_+\backslash (G \backslash \Mod(X))=G \backslash ( \R^*_+ \backslash \Mod(X))=G \backslash X.$$
Suppose that $\sigma$ is ergodic. The flow $\widetilde{\theta}$ is dissipative if and only if $\sigma$ admits an invariant faithful semifinite measure. When this is not the case, we say that $\sigma$ is of type $\III$. Then one of the following situation occurs:
\begin{itemize}
\item  $\widetilde{\theta}$ is trivial, i.e.\ $\Mod(\sigma)$ is ergodic. Then we say that $\sigma$ is of type $\III_1$.
\item $\widetilde{\theta}$ is periodic with $\ker(\widetilde{\theta})=\lambda^\Z$ for some $\lambda \in ]0,1[$. Then we say that $\sigma$ is of type $\III_\lambda$.
\item $\widetilde{\theta}$ is properly ergodic (not transitive). Then we say that $\sigma$ is of type $\III_0$.
\end{itemize}

\subsubsection{The Krieger $T$-set}
Let $\sigma : G \curvearrowright X$ be an ergodic nonsingular action. Note that $G$ acts on the space of $z$-densities $\Lambda^z(X)$ for all $z \in \C$. The Krieger $T$-set of $\sigma$ is defined by
$$ T(\sigma) = \{ t \in \R \mid \Lambda^{\ri t}(X)^G \neq \{ 0 \} \}.$$
It is clearly a subgroup of $\R$. We have $T(\sigma)=\R$ if and only if $\sigma$ is not of type $\III$. If $\sigma$ is of type $\III_\lambda$ for some $\lambda \in ]0,1[$, then $T(\sigma)=\frac{2\pi}{|\log(\lambda)|} \Z$.

We observe that if $t \in T(\sigma)$, then there exists a faithful finite measure $\mu$ on $X$ such that $\mu^{\ri t}$ is $G$-invariant. Indeed, if $\zeta \in \Lambda^{\ri t}(X)^G$ with $\zeta \neq 0$, then $|\zeta|^2=\zeta\bar{\zeta} \in \rL^0(X)$ is $G$-invariant hence constant and we can assume that it is equal to $1$. Then, we can write $\zeta=e^{\ri h} \nu^{\ri t}=(e^{ h/t} \nu)^{\ri t}$ where $\nu$ is any given faithful probability measure on $X$ and $h \in \rL^0(X,[0,2\pi])$.

\subsubsection{The $\lambda$-modular bundle} \label{lambda nonsingular} Let $\lambda \in ]0,1[$. A \emph{$\lambda$-nonsingular space} is a standard borel space $X$ equiped with a class of $\lambda$-equivalent $\sigma$-finite measures. Two measures $\mu, \nu$ are \emph{$\lambda$-equivalent} if they are equivalent and the Radon-Nikodym derivative $\frac{\rd \mu }{\rd \nu}$ takes only values in $\lambda^{\Z}=\{ \lambda^n \mid n \in \Z\}$. There are obvious notions of $\lambda$-nonsingular maps and $\lambda$-nonsingular actions etc. The modular bundle can also be adapted: replace the measure scaling flows in all this section by $\lambda$-scaling actions of $\Z$, i.e.\ a nonsingular automorphism $\theta$ of a measure space $(Y,\tau)$ such that $\tau \circ \theta=\lambda \tau$. We obtain a notion of $\lambda$-modular bundle $\Mod_\lambda(X)$ for every $\lambda$-nonsingular space $X$. If we have a $\lambda$-nonsingular action $\sigma : G \curvearrowright X$ (i.e.\ $\sigma$ preserves the $\lambda$-equivalence class), then we get a $\lambda$-Maharam extension $\Mod_\lambda(\sigma)$. As a nonsingular action, $\sigma$ is of type $\III_\lambda$ if and only if $\Mod_\lambda(\sigma)$ is ergodic.


\subsection{The Koopman representation} \label{appendix koopman}
Let $\sigma : G \curvearrowright X$ be a nonsingular action. Then the Maharam extension $\Mod(\sigma)$ restricts to a natural action of $G$ on the canonical $\rL^2$-space $\rL^2(X)$ which is an orthogonal (or unitary) representation of $G$ called the \emph{Koopman representation} of $\sigma$. We review some useful properties of this representation. 

\begin{lem} \label{lem convergence proba}
Let $X$ be a nonsingular space. For any bounded sequence of positive vectors $\xi_n \in \rL^2(X)^+$, the following are equivalent:
\begin{enumerate}[ \rm (i)]
\item $\lim_n \langle \xi_n , \eta \rangle =0$ for some faithful $\eta \in \rL^2(X)^+$.
\item $\forall \eta \in \rL^2(X), \; \lim_n \langle \xi_n,\eta\rangle =0$.
\end{enumerate}
\end{lem}
\begin{proof}
Suppose that $(\rm i)$ holds. Since the sequence $(\xi_n)_n$ is bounded, it suffices to prove $(\rm ii)$ for a dense subset of $\rL^2(X)$. Let $f \in \rL^\infty(X)$ be a positive function. Since $\xi_n$ and $\eta$ are positive, we have $\langle \xi_n,f \eta\rangle \leq \| f\|_\infty \langle \xi_n,\eta\rangle$ for all $n$, hence $\lim_n \langle \xi_n , f \eta \rangle =0$. Since $\eta$ is faithful, the linear span of all such vectors $f \eta$ is dense in $\rL^2(X)^+$ and we are done.
\end{proof}

The following criterion is well-known, at least when $G$ is a countable group.

\begin{prop} \label{criterion invariant proba}
Let $\sigma:G \curvearrowright X$ be a nonsingular action and $\pi : G \curvearrowright \rL^2(X)$ the associated Koopman representation. Then the following are equivalent:
\begin{enumerate}[\rm (i)]
\item $\sigma$ has no invariant probability measure.
\item $\pi$ has no invariant vectors.
\item $\inf_{g \in G} \langle \pi(g) \xi,\xi\rangle =0$ for some faithful $\xi \in \rL^2(X)^+$.
\item $\pi$ is weakly mixing.
\end{enumerate}
\end{prop}
\begin{proof}
The equivalence $(\rm i) \Leftrightarrow (\rm ii)$ is clear as well as the implication $(\rm iv) \Rightarrow (\rm ii)$. Therefore, we only have to show that $(\rm ii) \Rightarrow (\rm iii) \Rightarrow (\rm iv)$.

$(\rm ii) \Rightarrow (\rm iii)$. Suppose that $c=\inf_{g \in G} \langle  \pi(g)\xi,\xi \rangle > 0$. Let $\eta$ be the element of minimal norm in the closed convex hull of $\{\pi(g)\xi \mid g \in G \}$. Then $\eta$ is an invariant vector for $\pi$ and $\eta \neq 0$ because $\langle \eta, \xi \rangle \geq c > 0$.

$(\rm iii) \Rightarrow (\rm iv)$. Take a sequence $g_n \in G$ such that $\lim_n \langle \pi(g_n)\xi,\xi \rangle =0$. Then by Lemma \ref{lem convergence proba}, we get $\lim_n \langle \pi(g_n)\xi,\eta \rangle=0$ for every $\eta \in \rL^2(X)$. This implies that $\lim_n \langle \pi(g_n)^{¨*} \eta, \xi \rangle =0$. Now, suppose that $\eta$ is positive. Then by Lemma \ref{lem convergence proba} again, we get $\lim_n \langle \pi(g_n)^{*} \eta, \eta' \rangle =0$ for all $\eta' \in \rL^2(X)$. Since $\rL^2(X)$ is closure of the linear span of $\rL^2(X)^+$, we conclude that $\pi$ is weakly mixing.
\end{proof}

With a similar proof, we also have the following equivalence.

\begin{prop} \label{mixing Koopman}
Let $\sigma:G \curvearrowright X$ be a nonsingular action and $\pi : G \curvearrowright \rL^2(X)$ the associated Koopman representation. Then the following are equivalent:
\begin{enumerate}[\rm (i)]
\item $\lim_{g \to \infty} \langle \pi(g) \xi,\xi\rangle =0$ for some faithful $\xi \in \rL^2(X)^+$.
\item $\pi$ is mixing.
\end{enumerate}
\end{prop}

\begin{df}
We say that a nonsingular action is of \emph{zero-type} if it satisfies the equivalent conditions of Proposition \ref{mixing Koopman}.
\end{df}

\begin{prop} \label{almost invariant vectors Koopman}
Let $\sigma:G \curvearrowright X$ be a nonsingular action and $\pi : G \curvearrowright \rL^2(X)$ the associated Koopman representation. Then the following are equivalent:
\begin{enumerate}[\rm (i)]
\item $\sigma$ has almost invariant probability measures, i.e.\ for every $\varepsilon > 0$ and every compact subset $K \subset G$, there exists a probability measure $\eta$ on $X$ such that
$$\forall g \in K, \; \| g_* \eta-\eta\| < \varepsilon.$$
\item $\sigma$ admits an invariant mean, i.e.\ a $G$-invariant state $\phi \in \rL^\infty(X)^*$.
\item $\pi$ has almost invariant vectors.
\item $\pi \otimes \rho$ has almost invariant vectors for some orthogonal representation $\rho : G \rightarrow \mathcal{O}(H)$.
\end{enumerate}
\end{prop}
\begin{proof}
The equivalence $(\rm i) \Leftrightarrow (\rm iii)$ follows from the inequalities 
$$ \| \mu^{1/2}- \nu^{1/2} \|^2 \leq \| \mu - \nu \| \leq 2\| \mu^{1/2}- \nu^{1/2} \|  $$
for every pair of probability measures $\mu, nu$ on $X$. The implication $(\rm i) \Rightarrow (\rm ii)$ is obtained by taking a weak$^*$-limit of almost invariant probability measures. The converse $(\rm ii) \Rightarrow (\rm i)$ follows from the Hahn-Banach theorem. It only remains to show that $(\rm iv) \Rightarrow (\rm iii)$. Suppose that $\pi \otimes \rho$ has almost invariant vectors for some representation $\rho$. Then $\pi \otimes \pi$ has almost invariant vectors. This means that the action $\sigma \otimes \sigma : G \curvearrowright X \otimes X$ has an invariant mean $\phi \in \rL^\infty(X \otimes X)^*$. By restricting this invariant mean to $\rL^\infty(X) \otimes 1$, we conclude that $\sigma : G \curvearrowright X$ has an invariant mean.
\end{proof}

\begin{prop} \label{product action invariant mean}
Let $\sigma:G \curvearrowright X$ and $\rho : G \curvearrowright Y$ be two nonsingular actions. Then $\sigma \otimes \rho$ has an invariant mean if and only if both $\sigma$ and $\rho$ have an invariant mean.
\end{prop}

\begin{df}
Let $\pi : G \curvearrowright \cH$ be an orthogonal representation of a locally compact group $G$. Let $\mu$ be a symmetric probability measure on $G$. The \emph{$\mu$-spectral radius} of $\pi$, denoted by $\rho_\mu(\pi)$, is defined by
$$ \rho_\mu(\pi)= \left \| \int_{G} \pi(g) \rd \mu(g) \right \|_{\B(\cH)}.$$
The $\mu$-spectral radius of $G$ is defined by $\rho_\mu(G)=\rho_{\mu}(\lambda_G)$ where $\lambda_G : G \curvearrowright \rL^2(G)$ is the left regular representation. Let $\sigma : G \curvearrowright X$ be a nonsingular action. The $\mu$-spectral radius of $\sigma$ is defined by $\rho_\mu(\sigma)=\rho_\mu(\pi)$ where $\pi$ is the Koopman representation of $\sigma$.
\end{df}

The spectral radius can be easily computed in concrete examples with the following formula. We are grateful to Stefaan Vaes for noticing a mistake in the proof of this proposition in an earlier version of this paper and we are grateful to Narutaka Ozawa for providing us with the following correct proof.

\begin{prop} \label{single vector spectral radius}
Let $X$ be a nonsingular space and let $P : \rL^2(X) \rightarrow \rL^2(X)$ be a self-adjoint bounded operator such that $P(\rL^2(X)^+) \subset \rL^2(X)^+$. Then for any faithful vector $\xi \in \rL^2(X)^+$, we have
$$ \|P\|=\lim_n  \| P^{n}\xi\|^{\frac{1}{n}}.$$
More generally, the same conclusion holds if we only assume that the family of positive vectors $(P^n\xi)_{n \in \N}$ is faithful, i.e.\ $\bigcup_n \supp( P^n \xi)=X$.
\end{prop}
\begin{proof}
Without loss of generality, we may assume that $\|P\|=1$. Take $\varepsilon > 0$ and let $\eta$ be a unit vector in $1_{[1-\varepsilon,1]}(P)\rL^2(X)$. Since the family of positive vectors $(P^j \xi)_{j \in \N}$ is faithful, there exists $\eta' \in \rL^2(X)$, $C > 0$, and $k \in \N$ such that 
$\|\eta-\eta'\|<1/2$ and $|\eta'|\le C\sum_{j=0}^{k-1} P^j\xi$. This implies that $$ \| P^n \eta' \|  \leq \| P^n |\eta'| \| \leq C \| P^n \sum_{j=0}^{k-1} P^j \xi \|$$
for all $n \in \N$. Moreover, since $\|\eta-\eta'\|<1/2$ and $\eta=1_{[1-\varepsilon,1]}(P)\eta$, we know that $\eta'':= 1_{[1-\varepsilon,1]}(P)\eta'$ has norm $>1/2$. Hence for every $n$ one has 
\begin{align*}
kC\|P^n\xi\| 
\geq C\|P^n (\sum_{j=0}^{k-1} P^j\xi)\| \geq \|P^n\eta'\|\geq \|P^n\eta''\| 
\geq \frac{1}{2}(1-\varepsilon)^n.
\end{align*}
This implies $\liminf \|P^n\xi\|^{1/n}\geq 1-\varepsilon$. Since $\varepsilon>0$ was arbitrary, we get $\liminf \|P^n\xi\|^{1/n}\geq 1$. The inequality $\limsup_n \| P^n \xi \|^{1/n} \leq 1$ being obvious, we conclude that $\lim_n \| P^n \xi \|^{1/n}=1$ as we wanted.
\end{proof}

It was pointed out to us that one can deduce from the previous proposition the following well-known fact, due to \cite{Kes59} for discrete groups and to \cite{BC74} in general.
\begin{cor}[\cite{Kes59, BC74}]
Let $G$ be a locally compact group and $\mu$ be a symmetric probability measure on $G$. Then for any relatively-compact neighborhood of the identity $V\subset G$, one has  
\[
\rho_\mu(G)=\limsup_n \mu^n(V)^{1/n}=\lim_n \mu^{2n}(V)^{1/2n}
\]
\end{cor}
\begin{proof}
Let $V$ be given and choose an open neighborhood of the identity $U$ such that $UU^{-1}\subset V$. 
Let $H$ be the subgroup generated by $\supp\mu$ and put $X:=HU\subset G$. 
For any $g\in G$, by applying the above theorem 
to $P:=\lambda(\mu)|_{L^2(Xg)}$ and $\xi:=1_{Ug}\in L^2(Xg)_+$, 
one obtains 
\[
\|P\| = \lim_n\|P^n\xi\|^{1/n}=\lim_n\langle \lambda(\mu^{2n})1_{Ug},1_{Ug} \rangle^{1/2n}
 \le \liminf_n \mu^{2n}(V)^{1/2n}.
\]
Since the subspaces $L^2(Xg)$ are $\lambda(\mu)$-invariant and have 
a dense span in $L^2(G)$, one has 
\[
\| \lambda(\mu) \|=\sup_g\| \lambda(\mu)|_{L^2(Xg)} \| \le \liminf_n \mu^{2n}(V)^{1/2n}.
\]
On the other hand, there is $\xi\in L^2(G)$ such that 
$(\xi*\xi^*)(g)=\int_G\xi(x)\overline{\xi(gx)}\,dx\geq 1$ for all $g\in V$ 
and so 
\[
\|\lambda(\mu)\|^n\|\xi\|^2\geq \langle \lambda(\mu)^n\xi,\xi \rangle=\int_G(\xi*\xi^*)(g)\,d\mu^{n}(g)\geq\mu^n(V),
\]
which yields $\|\lambda(\mu)\|\geq\limsup_n\mu^n(V)^{1/n}$. 
\end{proof}

For the following proposition, see \cite[Corollary 2.3.3]{AD03}.
\begin{prop} \label{spectral radius action group}
Let $\sigma : G \curvearrowright X$ be a nonsingular action and $\mu$ a symmetric probability measure on $G$. Then $\rho_\mu(\sigma) \geq \rho_\mu(G)$.
\end{prop}

\subsection{Amenability} \label{section amenability}

\begin{df}
A nonsingular action $\sigma : G \curvearrowright X$ is called \emph{amenable} (in Zimmer's sense) if there exists a $G$-equivariant conditional expectation $E : \rL^\infty(G \otimes X) \rightarrow \rL^\infty(X)$ where $G$ acts diagonnaly on $G \otimes X$.
\end{df}

The trivial action of $G$ on the one-point space is amenable if and only if $G$ is amenable in the usual sense. Conversely, if $G$ is amenable, then every nonsingular action of $G$ is amenable. Nonamenable groups admit amenable actions. Indeed, the left translation action of a locally compact group $G$ on itself is always amenable. We will need the following facts which can be found in the reference \cite{AD03}.

\begin{prop} \label{mean amenable}
Let $\sigma : G \curvearrowright X$ be a nonsingular action. If $\sigma$ is amenable and admits an invariant mean, then $G$ is amenable.
\end{prop}

\begin{prop} \label{amenable factor map}
Let $\sigma : G \curvearrowright X$ and $\rho : G \curvearrowright Y$ be two nonsingular actions. Suppose that $\rho$ is amenable and that we have a $G$-equivariant quotient map from $X$ to $Y$. Then $\sigma$ is amenable. 
\end{prop}

\begin{prop} \label{weakly contained amenable action}
Let $\sigma : G \curvearrowright X$ be an amenable nonsingular action. Then the Koopman representation of $\sigma$ is weakly contained in the left regular representation of $G$ and in particular $\rho_\mu(\sigma)=\rho_\mu(G)$ for any probability measure $\mu$ on $G$.
\end{prop}

\subsection{Relative entropy} \label{appendix entropy}

\begin{df}
Let $X$ be a nonsingular space and $\mu$, $\nu$ be two faithful probability measures on $X$. The \emph{relative entropy} of $\mu$ with respect to $\nu$ is defined by
$$ H(\mu \mid \nu) = -\int_X \log\left( \frac{\rd \mu}{\rd \nu} \right) \rd \nu.$$
\end{df}

\begin{prop} \label{entropy inequality}
Let $X$ be a nonsingular space and $\mu$, $\nu$ be two faithful probability measures on $X$. Then we have
$$ -2 \log \langle \mu^{1/2},\nu^{1/2} \rangle \leq H(\mu \mid \nu).$$
\end{prop}
\begin{proof}
By the concavity of $\log$, we have 
\begin{align*}
 - \log \langle \mu^{1/2},\nu^{1/2} \rangle &= - \log \int_X  \sqrt{\frac{ \rd \mu}{\rd \nu}}  \rd \nu \\
 & \leq -\int_X  \log \sqrt{\frac{ \rd \mu}{\rd \nu}}  \rd \nu \\
 &= \frac{1}{2}H(\mu | \nu ).
\end{align*}
\end{proof}

\begin{df} \label{almost vanishing entropy}
Let $\sigma: G \curvearrowright X$ be a nonsingular action. We say that $\sigma$ has \emph{almost vanishing entropy} if for every $\varepsilon > 0$ and every compact set $K \subset G$ there exists a faithful probability measure $\eta$ on $X$ such that
$$ \forall g \in K, \; H(g_* \eta \mid \eta) \leq \varepsilon.$$
When $\sigma$ is the left translation action $G \curvearrowright G$, we say that $G$ has \emph{almost vanishing entropy}.
\end{df}

\begin{prop} \label{zero entropy implies invariant mean}
Let $\sigma: G \curvearrowright X$ be a nonsingular action. If $\sigma$ has almost vanishing entropy then $\sigma$ has an invariant mean. In particular, if $G$ has almost vanishing entropy then $G$ is amenable.
\end{prop}

\begin{proof}
This follows from Proposition \ref{entropy inequality} and Proposition \ref{almost invariant vectors Koopman}.
\end{proof}

\subsection{Fibered measures}
\begin{df}
Let $\pi : X \rightarrow Y$ a quotient map between two nonsingular spaces. A fibered measure over $\pi$ is a map $T : \rL^0(X,[0,+\infty]) \rightarrow \rL^0(Y,[0,+\infty])$ which is:
\begin{enumerate}[ \rm (i)]
\item \emph{linear}: $T(0)=0$ and $T(\alpha f + \beta g)= \alpha T(f)+\beta T(g)$ for all $f,g \in \rL^0(X,[0,+\infty])$ and all $\alpha, \beta > 0$.
\item \emph{normal}: $T(\sup_n f_n)=\sup_n T(f_n)$ for any increasing sequence $f_n \in \rL^0(X,[0,+\infty])$. 
\item \emph{$\pi$-modular}: $T((f\circ \pi) g)=fT(g)$ for all $g \in \rL^0(X,[0,+\infty])$ and all $f \in \rL^0(Y,[0,+\infty])$.
\end{enumerate}
We say that $T$ is:
\begin{enumerate}[ \rm (i)]
\item \emph{faithful} if $T(f)=0$ implies $f=0$ for all $f \in \rL^0(X,[0,+\infty])$.
\item \emph{semifinite} if the subspace $\{ f \in \rL^0(X,[0,+\infty]) \mid T(f) < +\infty \}$ is dense in  $\rL^0(X,[0,+\infty])$.
\item \emph{purely infinite} if $T(f)=+\infty$ for all $f \in \rL^0(X,]0,+\infty])$.
\item a \emph{conditional expectation} if $T(1)=1$.
\end{enumerate}
\end{df}

\begin{example}
Let $Y=\{ * \} $ be the one-point space so that $\rL^0(Y,[0,+\infty])=[0,+\infty]$. Let $\pi : X \rightarrow \{ * \} $ be the unique quotient map. Then fibered measures $T$ over $\pi$ are in one-to-one correspondance with measures $\mu$ on $X$ via integration
$$ T(f)=\int_X f \rd \mu.$$
\end{example}

\begin{prop} \label{desintegration}
Let $\pi : X \rightarrow Y$ be a quotient map between two nonsingular spaces. Let $\mu$ and $\nu$ be two measures on $X$ and $Y$ respectively. Suppose that $\nu$ is faithful and semifinite. Then there exists a unique fibered measure $T$ over $\pi$ such that $\nu \circ T=\mu$. Moreover, $T$ is faithful (resp.\ semifinite) if and only if $\mu$ is faithful (resp.\ semifinite).
\end{prop}

\subsection{\'Etale maps}
Let $\pi : X \rightarrow Y$ be a quotient map. A \emph{section} of $\pi$ is a subset $A \subset X$ such that $\pi|_A: A \rightarrow X$ is an isomorphism. If $\pi|_A: A \rightarrow X$ is only injective, we say that $A$ is a \emph{partial section}. 

\begin{prop} \label{proposition etale}
Let $\pi : X \rightarrow Y$ be a quotient map. The following are equivalent:
\begin{enumerate}[ \rm (i)]
\item $X$ can be covered by sections of $\pi$.
\item $X$ can be covered by partial sections of $\pi$.
\item \label{index etale} There exists a semifinite fibered measure $T$ over $\pi$ such that $T(f) \geq f$ for all $f \in \rL^0(X, [0,+\infty])$.
\item There is an injective nonsingular map $i : X \rightarrow Y \otimes \N$ such that $\pi' \circ i=\pi$ where $\pi' : Y \otimes \N \rightarrow Y$ is the projection on the first coordinate.
\end{enumerate}
If these conditions are satisfied, there exists a unique fibered measure $\Lambda$ over $\pi$ such that $\Lambda(1_A)=1$ for every section $A$ of $\pi$.
\end{prop}

We say that $\pi$ is \emph{\'etale} if it satisfies the above equivalent conditions. The fibered measure $\Lambda$ is called the fibered \emph{counting} measure over $\pi$.

\subsection{Dissipative and recurrent actions} \label{appendix dissipative}

In this section we study the notion of dissipativity and recurrence for actions of locally compact groups. This simply generalizes what is already known for flows and single transformations as in \cite{Aa97}, or for countable groups as in \cite{VW18}.

\begin{df}
Let $\sigma : G \curvearrowright X$ be a nonsingular action of a locally compact group $G$. Let $\pi : X \rightarrow G \backslash X$ be the quotient map. Choose a left Haar measure $\rd g$ on $G$. The \emph{fibered Haar measure} over $\pi$ is defined by
$$ T(f) =\int_G \sigma_g(f) \rd g.$$
We say that $\sigma$ is \emph{dissipative} (or \emph{integrable}) if $T$ is semifinite. We say that $\sigma$ is \emph{recurrent} (or \emph{conservative}) if $T$ is purely infinite.
\end{df}

\begin{example}
Let $H$ be a closed subgroup of $G$. Then the left translation action $G \curvearrowright G/H$ is dissipative if and only if $H$ is compact, otherwise it is recurrent.
\end{example}

The following property justifies the term \emph{recurrent}.
\begin{prop}
Let $\sigma : G \curvearrowright X$ be a nonsingular action of a locally compact group $G$. If $\sigma$ is recurrent then for every subset $A \subset X$ with $A \neq 0$ and every compact set $K \subset G$, there exists $g \in G \setminus K$ such that $g A \cap A \neq 0$.
\end{prop}
\begin{proof}
Take $0 \neq A \subset X$. Suppose that there exists a compact set $K \subset G$ such that $g A \cap A = 0$ for all $g \notin K$. Then we have 
$$T_m(1_A)1_A=\int_K \sigma_g(1_A) 1_A \rd m g \leq m(K).$$
This shows that $T_m$ is not purely infinite, hence $\sigma$ is not recurrent.
\end{proof}

Item $(\rm iv)$ of the following theorem was pointed out to us by Narutaka Ozawa. Compare it with the definition of amenability in \ref{section amenability}.
\begin{thm} \label{dissipative trivialization}
Let $\sigma : G \curvearrowright X$ be a nonsingular action of a locally compact group $G$. Then the following are equivalent:
\begin{enumerate}[ \rm (i)]
\item $\sigma$ is dissipative 
\item for some faithful probability measure $\mu$ on $X$, we have
$$ \int_G \frac{\rd g_* \mu}{ \rd \mu}(x) \rd g < +\infty, \quad \text{ for almost every } x \in X.$$
\item for every faithful probability measure $\mu$ on $X$, we have
$$ \int_G \frac{\rd g_* \mu}{ \rd \mu}(x) \rd g < +\infty, \quad \text{ for almost every } x \in X.$$
\item there exists a $G$-equivariant normal conditional expectation $E : \rL^\infty(G \otimes X) \rightarrow \rL^\infty(X)$ where $G$ acts diagonnaly on $G \otimes X$.
\item every ergodic component in the ergodic decomposition of $\sigma$ is a left translation action of the form $G \curvearrowright G/K$ for some compact subgroup $K$.
\end{enumerate}
\end{thm}
\begin{proof}
The confition $$ \int_G \frac{\rd g_* \mu}{ \rd \mu}(x) \rd g < +\infty, \quad \text{ for almost every } x \in X$$
means precisely that the measure $\mu \circ T$ is semifinite where $T$ is the fibered Haar measure of $\sigma$. Thus the equivalence $(\rm i) \Leftrightarrow (\rm ii) \Leftrightarrow (\rm iii)$ follows from the second part of Proposition \ref{desintegration}.

$(\rm iii) \Rightarrow (\rm iv)$ Let $\mu$ be a faithful probability measure $\mu$ on $X$ and define
$$ F(x)=\left( \int_G \frac{\rd g_* \mu}{ \rd \mu}(x) \rd g \right)^{-1}.$$
Observe that
$$ \int_G F(g^{-1}x) \rd g = 1$$
for almost every $x \in X$. Therefore, we can define a $G$-equivariant normal conditional expectation $E  : \rL^\infty(G \otimes X) \rightarrow \rL^\infty(X)$ by the formula
$$ E(f)(x)=\int_G f(g,x) F(g^{-1}x) \rd g.$$

$(\rm iv) \Rightarrow (\rm i)$ Let $T$ be the fibered Haar measure of $G \curvearrowright X$ and $T'$ be the fibered Haar measure of $G \curvearrowright G \otimes X$. Then, since $E$ is $G$-equivariant, we have $T \circ E=E \circ T'$. Since $T'$ is semifinite, this forces $T$ to be also semifinite.

$(\rm i) \Rightarrow (\rm v)$. First observe that a nonsingular action is dissipative if and only if each one of its ergodic components is dissipative. Therefore, we may assume that $\sigma$ is ergodic, and it is enough to show that $\sigma$ must be transitive, i.e.\ that there exist a $G$-equivariant nonsingular map from $G$ to $X$. 

We first assume that $\sigma$ is free. Then here exists a pair $(S,p)$, called a \emph{cross-section} of $\sigma$ (see \cite[Section 4]{KPV15}), where $S$ is a nonsingular space and $p : G \otimes S \rightarrow X$ is a quotient map such that:
\begin{enumerate}[ \rm (i)]
\item $U \otimes S$ is a partial section of $p$ for some nonempty open subset $U$ in $G$.
\item $p$ is $G$-equivariant with respect to the actions $\rho \otimes \id : G \curvearrowright G \otimes S$ and $\sigma : G \curvearrowright X$.
\end{enumerate}
These conditions imply that the map $ p :G \otimes S \rightarrow X$ is \'etale since $gU \otimes S$ is a partial section of $p$ for all $g \in G$. Let $\Lambda$ be the fibered counting measure over $p$. Note that $\Lambda$ is $G$-equivariant. Therefore, for every $f \in \rL^0(S,[0,+\infty]) \subset \rL^0(G \otimes S, [0,+\infty])$ we have $\Lambda(f) \in \rL^0(X,[0,+\infty])^{G}$ which means that $\Lambda(f)$ is a constant in $[0,+\infty]$ because $\sigma$ is ergodic. This means that $\Lambda$ restricts to a measure $\mu$ on $S$. Note that $\mu(f) =\Lambda(f) \geq f$ for all $f \in \rL^0(S,[0,+\infty])$.

Let $\nu$ be the fibered Haar measure of $\sigma$ which is a faithful semifinite measure on $X$ because $\sigma$ is dissipative and ergodic. Let $T$ be the fibered Haar measure of $\rho \otimes \id : G \curvearrowright G \otimes S$. By the $G$-equivariance of $\Lambda$, we have $\mu \circ T=\Lambda \circ T=\nu \circ \Lambda$. Since $\Lambda$ and $\nu$ are semifinite, it follows that $\mu$ is also semifinite. But $\mu(f) \geq f$ for all $f \in \rL^0(S,[0,+\infty])$. This implies that $S$ is atomic. By restricting $p$ to an atom of $S$, we obtain a $G$-equivariant map from $G$ to $X$ hence $\sigma$ is transitive.

Now, suppose that $\sigma$ is not free. Take $\rho : G \curvearrowright Y$ a free mixing pmp action (which exists, see Remark \ref{freeness}). Then $\sigma \otimes \rho  : G \curvearrowright X \otimes Y$ is free and dissipative. Suppose that $\sigma$ is properly ergodic. Then $\sigma \otimes \rho$ is ergodic by \cite{SW81}. By the first part of the proof, we conclude that $\sigma \otimes \rho$ is transitive which implies that $\sigma$ is transitive (contradicting the assumption that $\sigma$ is properly ergodic).

$(\rm v) \Rightarrow (\rm i)$ is clear.
\end{proof}

Recall that a nonsingular action $\sigma$ of a locally compact group $G$ is called of type $\mathrm{I}$ if every ergodic component of its ergodic decomposition is transitive, i.e.\ of the form $G \curvearrowright G/H$ for some closed subgroup $H \subset G$. 
\begin{cor}
Every dissipative nonsingular action is of type $\mathrm{I}$. The converse is also true if the action is free.
\end{cor}

\begin{cor} \label{dissipative amenable}
A dissipative nonsingular action is amenable.
\end{cor}

\begin{cor} \label{dissipative invariant measure}
Let $\sigma : G \curvearrowright X$ be a nonsingular action. Suppose that $\sigma$ is dissipative and that $\mu$ is a $G$-invariant semifinite measure on $X$. Then there exists a $G$-invariant semifinite measure $\nu$ on $G \backslash X$ such that $\mu=\nu \circ T$ where $T$ is the fibered Haar measure of $\sigma$.
\end{cor}
\begin{proof}
By Theorem \ref{dissipative trivialization}, it is enough to consider the case where $\sigma$ is just a left translation action $G \curvearrowright G/K$ for some compact subgroup $K \subset G$. In that case, $\mu$ is simply a $G$-invariant measure on $G/K$ and $T$ is the pushforward of the Haar measure of $G$ on $G/K$. We indeed have $\mu=\lambda T$ for some $\lambda > 0$.
\end{proof}

The following result is due to Maharam in the case of flows or single transformations, and to Schmidt in the case of countable groups. We generalize it to locally compact groups.
\begin{prop}
Let $\sigma : G \curvearrowright X$ be a nonsingular action. Then $\sigma$ is dissipative if and only if its Maharam extension $\Mod(\sigma)$ is dissipative.
\end{prop}
\begin{proof}
Let $T$ and $T'$ be the fibered Haar measure of $\sigma$ and $\Mod(\sigma)$ respectively. Let $\tilde{\theta} : \R^*_+ \curvearrowright G \backslash \Mod(X)$ be the Krieger flow induced by $\theta$. Let $S$ and $\tilde{S}$ be the fibered Haar measures of $\theta$ and $\tilde{\theta}$ respectively. Then we have $\tilde{S} \circ T'=T \circ S$. Since $\theta$ is dissipative, we know that $S$ is semifinite. This already shows that if $\sigma$ is dissipative, i.e.\ $T$ is semifinite, then $T'$ must also be semifinite hence $\Mod(\sigma)$ is dissipative.

Conversely, suppose that $\Mod(\sigma)$ is dissipative. Since the modular measure $\tau$ on $\Mod(X)$ is $G$-invariant, there exists a faithful semifinite measure $\nu$ on $G \backslash \Mod(X)$ such that $\tau=\nu \circ T'$. Since $\tilde{\theta} \circ T=T \circ \theta$, we get $\nu \circ \tilde{\theta}_\lambda=\lambda \nu$ for all $\lambda \in \R^*_+$. Thus $\tilde{\theta}$ is a measure scaling flow. In particular, it is dissipative and $\tilde{S}$ is semifinite. We conclude that $\tilde{S} \circ T'=T \circ S$ is semifinite, hence $T$ is semifinite which means that $\sigma$ is dissipative.
\end{proof}

\begin{prop} \label{L2 criterion recurrent}
Let $\sigma : G \curvearrowright X$ be a nonsingular action. If $\sigma$ is recurrent, then 
$$ \int_G \left( \frac{\rd g_* \mu }{\rd \mu } \right)^s  \rd g=+\infty $$
for every faithful semifinite measure $\mu$ on $X$ and every $s \in \R$.
\end{prop}
\begin{proof}
Since $\sigma$ is recurrent, $\Mod(\sigma)$ is also recurrent. Therefore, the fibered Haar measure $T$ of $\Mod(\sigma)$ is purely infinite. In particular, viewing $\mu$ as a $1$-density $\mu \in \Lambda^1(X)^+$, we get $T(\mu^s) = +\infty$ for all $s \in \R$. In particular, we get $T(\mu^{s})\mu^{-s}=+\infty$ for all $s \in \R$, which is precisely the desired conclusion.
\end{proof}

\begin{prop} \label{lattice dissipative}
Let $\sigma : G \curvearrowright X$ be a nonsingular action and let $H < G$ be a closed subgroup. If $\sigma$ is dissipative then $\sigma|_{H}$ is dissipative. If $G/H$ is compact and $\sigma|_{H}$ is dissipative, then $\sigma$ is dissipative.
\end{prop}

\bibliographystyle{plain}

\end{document}